\documentclass[11pt,reqno,tbtags]{amsart}

\usepackage{amsthm,amssymb,amsfonts,amsmath}
\usepackage{amscd} 
\usepackage{mathrsfs}
\usepackage[numbers, sort&compress]{natbib}
\usepackage{subfigure, graphicx}
\usepackage{vmargin,caption}
\usepackage[pagebackref=true]{hyperref}
\usepackage{color}
\usepackage[normalem]{ulem}
\usepackage{stmaryrd} 
\usepackage{framed}
\usepackage[usenames,dvipsnames]{xcolor}
\newcommand{\R}{\mathbb{R}}
\newcommand{\Z}{\mathbb{Z}}
\newcommand{\N}{{\mathbb N}}

\newcommand{\E}[1]{{\mathbf E}\left[#1\right]}
\newcommand{\e}{{\mathbf E}}
\newcommand{\V}[1]{{\mathbf{Var}}\left\{#1\right\}}

\newcommand{\p}[1]{{\mathbf P}\left\{#1\right\}}

\newcommand{\I}[1]{{\mathbf 1}_{[#1]}}

 \newcommand{\bag}{\begin{align}}
\newcommand{\bags}{\begin{align*}}
\newcommand{\eag}{\end{align*}}
\newcommand{\eags}{\end{align*}}

\newtheorem{thm}{Theorem}[section]
\newtheorem{lem}[thm]{Lemma}
\newtheorem{prop}[thm]{Proposition}
\newtheorem{cor}[thm]{Corollary}
\newtheorem{dfn}[thm]{Definition}

\newtheorem{fact}[thm]{Fact}
\newtheorem{claim}[thm]{Claim}

\newtheorem{rem}[thm]{Remark}

\newcommand\cB{\mathcal B}
\newcommand\cC{\mathcal C}

\newcommand\cM{\mathcal M}

\newcommand\cO{\mathcal O}

\newcommand\cR{{\mathcal R}}

\newcommand\cT{{\mathcal T}}

\newcommand{\pran}[1]{\left(#1\right)}

\hypersetup{
    bookmarks=true,         
    unicode=false,          
    pdftoolbar=true,        
    pdfmenubar=true,        
    pdffitwindow=true,      
    pdftitle={My title},    
    pdfauthor={Author},     
    pdfsubject={Subject},   
    pdfnewwindow=true,      
    pdfkeywords={keywords}, 
    colorlinks=true,       
    linkcolor=blue,          
    citecolor=blue,        
    filecolor=blue,      
    urlcolor=blue           
}

\newcommand\urladdrx[1]{{\urladdr{\def~{{\tiny$\sim$}}#1}}}

\DeclareRobustCommand{\SkipTocEntry}[5]{}

\newcommand{\pl}{\ensuremath{\preceq_{\mathrm{lex}}}}

\newcommand{\plt}{\ensuremath{\preceq_{\mathrm{lex,T}}}}
\newcommand{\pc}{\ensuremath{\preceq_{\mathrm{ctr}}}}

\newcommand{\pcst}{\ensuremath{\prec_{\mathrm{ctr}}}}

\newcommand{\pct}{\ensuremath{\preceq_{\mathrm{ctr,T}}}}
\newcommand{\pcy}{\ensuremath{\preceq_{\mathrm{cyc}}}}
\newcommand{\bbrcy}[1]{\ensuremath{[#1]_{\mathrm{cyc}}}}
\newcommand{\pcyt}{\ensuremath{\preceq_{\mathrm{cyc,T}}}}

\newcommand{\eps}{\epsilon}
\newcommand{\angles}[1]{\langle #1 \rangle}
\newcommand{\eqdist}{\ensuremath{\stackrel{\mathrm{d}}{=}}}
\newcommand{\convdist}{\ensuremath{\stackrel{\mathrm{d}}{\rightarrow}}}
\newcommand{\convas}{\ensuremath{\stackrel{\mathrm{a.s.}}{\rightarrow}}}
\newcommand{\aseq}{\ensuremath{\stackrel{\mathrm{a.s.}}{=}}}
\newcommand{\be}{\mathbf{e}}
\newcommand{\rM}{\mathrm{M}} 
\newcommand{\rP}{\mathrm{P}} 
\newcommand{\rX}{\mathrm{X}} 
 
\newcommand{\dghp}{\ensuremath{d_{\mathrm{GHP}}}}
\newcommand{\dgh}{\ensuremath{d_{\mathrm{GH}}}}
\newcommand{\bbr}[1]{\ensuremath{\llbracket #1 \rrbracket}}
\providecommand{\S}{}
\renewcommand{\S}{\mathbb{S}}
\newcommand{\diam}{\mathrm{diam}}
\newcommand{\dis}{\mathrm{dis}}

\providecommand{\v}{}
\renewcommand{\v}{\mathrm{v}}

\newcommand{\psym}{\ensuremath{\bar}}

\newcommand{\branch}[2]{\ensuremath{\mathrm{path}(#1,#2)}}

\newcommand{\LGW}{\ensuremath{\mathrm{LGW}}}

\newcommand{\ssq}{{\begin{picture}(4,3)(0,0)\put(1,0){\line(0,1){3}}\put(1,0){\line(1,0){3}}\put(1,3){\line(1,0){3}}\put(4,0){\line(0,1){3}}\end{picture}\,} } 
\newcommand{\cTq}[1]{\ensuremath{\cT_{\ssq,#1}}}

\newcommand{\kl}[2]{\ensuremath{\kappa^{\ell}({#1},{#2})}}
\newcommand{\kr}[2]{\ensuremath{\kappa^r({#1},{#2})}}
\newcommand{\floor}[1]{\ensuremath\lfloor#1\rfloor}
\def \tend{\longrightarrow}

\newcommand{\rG}{\ensuremath{\mathrm{G}}}
\newcommand{\rT}{\ensuremath{\mathrm{T}}}
\newcommand{\rt}{\ensuremath{\mathrm{t}}}
\newcommand{\rr}{\ensuremath{\mathrm{r}}}\newcommand{\rR}{\ensuremath{\mathrm{R}}}
\newcommand{\rQ}{\ensuremath{\mathrm{Q}}}
\newcommand{\rV}{\ensuremath{\mathrm{V}}}
\newcommand{\rW}{\ensuremath{\mathrm{W}}}

\usepackage[refpage,noprefix,intoc]{nomencl}							
\makenomenclature											
\begin{document}

\title[The scaling limit of random simple triangulations]{The scaling limit of random simple triangulations and random simple quadrangulations}
\author{Louigi Addario-Berry \and Marie Albenque}
\address{Department of Math and Stats, McGill University, Montr\'eal, Qu\'ebec, H3A 2K6, Canada}
\address{LiX, Ecole Polytechnique, 91120 Palaiseau -- France}
\email{louigi.addario@mcgill.ca}
\email{albenque@lix.polytechnique.fr}
\date{June 20, 2013 } 
\urladdrx{http://www.problab.ca/louigi/}
\urladdrx{http://www.lix.polytechnique.fr/~albenque/}


\begin{abstract} 
Let $M_n$ be a simple triangulation of the sphere $\S^2$, drawn uniformly at random from all such triangulations with $n$ vertices. Endow $M_n$ with the uniform probability measure on its vertices. 
After rescaling graph distance by $(3/(4n))^{1/4}$, the resulting random measured metric space converges in distribution, in the Gromov--Hausdorff--Prokhorov sense, to the Brownian map. In proving the preceding fact, we introduce a labelling function for the vertices of $M_n$. Under this labelling, distances to a distinguished point are essentially given by vertex labels, with an error given by the winding number of an associated closed loop in the map. We establish similar results for simple quadrangulations. 
\end{abstract}
\maketitle
\graphicspath{{Pictures/}}

\begin{figure}[ht]
\vspace{-1cm}
\hspace{-0.5cm}
\includegraphics[width=.85\linewidth]{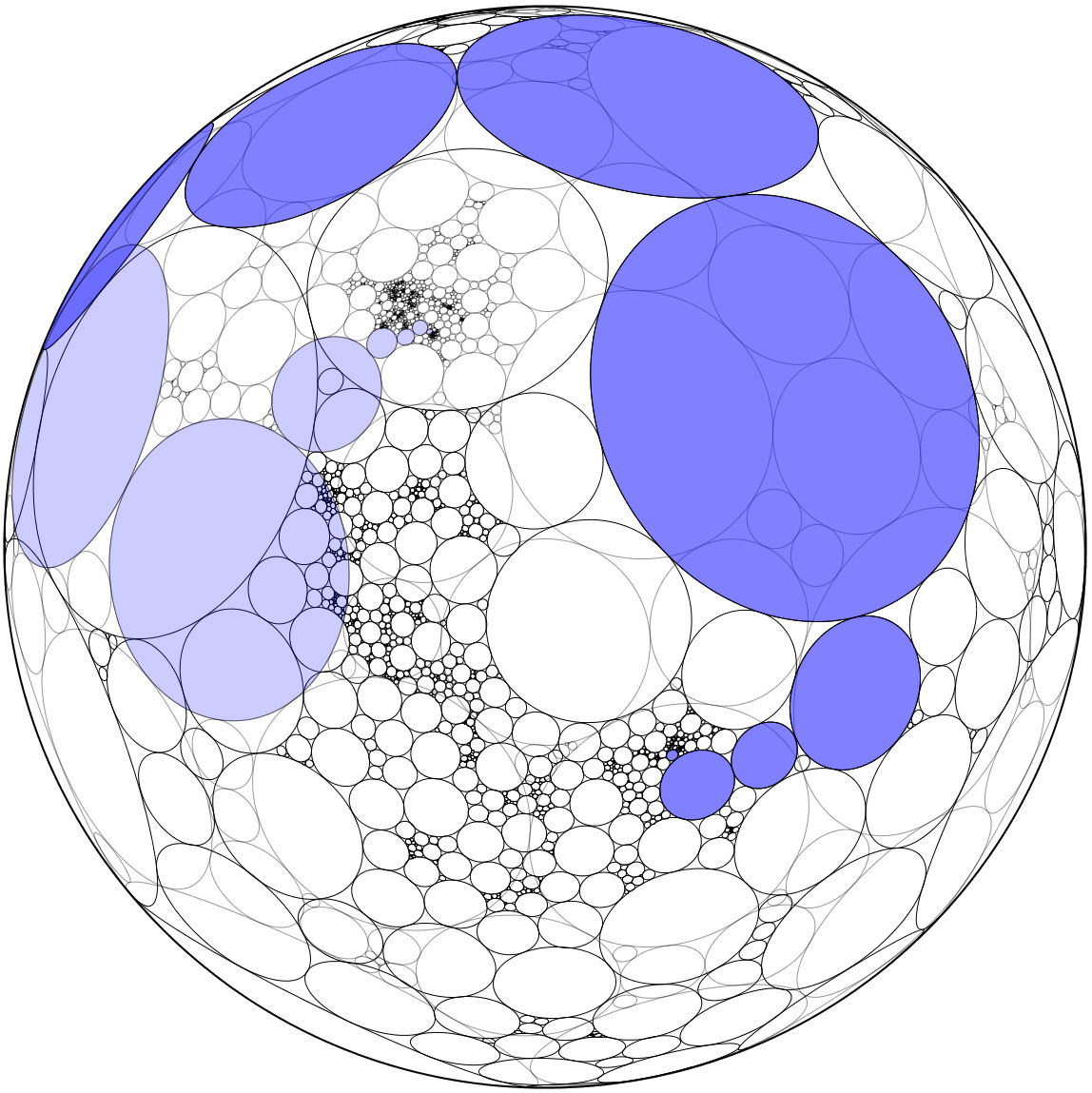}
\caption{\Small The circle packing associated to a uniformly random simple triangulation of $\S^2$ with $10^5$ vertices. Blue shaded circles form a shortest path between two uniformly random vertices (circles). 
Created using Ken Stephenson's CirclePack program.}
\label{fig:first}
\end{figure}

\tableofcontents
\section{Introduction}\label{sec:intro} 
We begin by heading straight for a statement of our main result.\footnote{Precise definitions of almost all the terminology used in the introduction appear in Sections~\ref{sec:pgpt} and~\ref{sec:notion_convergence}. After stating our main result, the remainder of introduction provides motivation and an overview of its proof, particularly the novel aspects of said proof.} A graph is {\em simple} if it has no loops or multiple edges. For integer $n \ge 3$, let $\triangle_n^{\circ}$ be the set of pairs $(M,\xi)$, where $M$ is an $n$-vertex simple triangulation of the sphere $\S^2$, and $\xi$ is a corner of $M$.
 Also, for integer $n \ge 4$, let 
$\Box_n^{\circ}$ be the set of pairs $(M,\xi)$ with $M$ an $n$-vertex simple quadrangulation of $\S^2$ and $\xi$ a corner of $M$. Then let $\cM=(\cM_n,n \ge 4)$ be one of the sequences $(\triangle_n^{\circ},n \ge 4)$ or $(\Box_n^{\circ},n \ge 4)$. 
\begin{thm}\label{thm:main}
For $n \ge 4$, let $(M_n,\xi_n)$ be a uniformly random element of $\cM_n$. Write $V(M_n)$ for the set of vertices of $M_n$, 
let $d_n:V(M_n) \to \N$ be graph distance in $M_n$ and let $\mu_n$ be the uniform probability measure on $V(M_n)$. Finally, let 
$c = (3/4)^{1/4}$ if $\cM=(\triangle_n^{\circ},n \ge 4)$ and let $c=(3/8)^{1/4}$ if 
$\cM=(\Box_n^{\circ},n \ge 4)$. Then, as $n \to \infty$, 
\[
(V(M_n),cn^{-1/4}d_n,\mu_n) \convdist (S,d,\mu),
\]
for the Gromov--Hausdorff--Prokhorov distance, where $(S,d,\mu)$ is the Brownian map. 
\end{thm}
We recall the definition of the Brownian map in Section~\ref{sec:limitobject}, below. Our proof relies upon the remarkable work of Miermont \cite{miermont13brownian} and, independently, Le Gall \cite{jf}, which both established convergence for general (non-simple) random quadrangulations. In particular, our results do not constitute an independent proof of uniqueness of the limit object. A discussion of the constants in the above theorem, and their relation with those from \cite{miermont13brownian,jf}, appears in Appendix~\ref{sec:notes}.

The part of Theorem~\ref{thm:main} pertaining to simple triangulations (sometimes called {\em type-III} triangulations; see \cite{MR1465433}) answers a question of Le Gall \cite{jf} and Le Gall and Beltran \cite{beltran13quad}. One general motivation for establishing convergence to the Brownian map is its conjectured role as a universal limit object for a wide range of random map ensembles. However, the case of simple triangulations holds additional interest due to the conjectured link between the Brownian map and the Liouville quantum gravity constructed by Duplantier and Sheffield \cite{duplantier2011liouville}; see \cite{garban2012bourbaki} for further discussion of this connection. Le Gall \cite{legall07topological} proved that the Brownian map is almost surely homeomorphic to the $2$-sphere (see also \cite{legall2008scaling,Miermont08}). However, homeomorphism equivalence is too weak, for example, to deduce conformal information or to prove dimensional scaling relations. For these, a canonical embedding of the Brownian map in $\S^2$ is needed (or at least would be very useful). 

For any simple triangulation $M$ of $\S^2$, the Koebe-Andreev-Thurston theorem (see, e.g., \cite{stephenson05circle}, Chapter 7) provides a {\em canonical circle packing} in $\S^2$, unique up to conformal automorphism, whose tangency graph is $M$; see Figure~\ref{fig:first} for an illustration of a random circle packing. (This uniqueness holds only for simple triangulations; for a uniformly random (non-simple) triangulation $N$ with $n$ vertices, for example, the number of degrees of freedom in a circle packing with tangency graph $N$ is typically linear in $n$.) The uniqueness provides hope that the conformal properties of the Brownian map can be accessed by studying the circle packings associated to large random simple triangulations

We deduce Theorem~\ref{thm:main} from a result which provides more general sufficient conditions for a sequence $(M_n,n \in \N)$ of random planar maps to converge in distribution to the Brownian map. 
More precisely, Theorem~\ref{prop:CSTriple} states conditions under which, after suitably rescaling distances, and endowed with the uniform probability measure on its vertex set, $M_n$ converges in distribution to the Brownian map for the Gromov--Hausdorff--Prokhorov distance. 

The approach of Theorem~\ref{prop:CSTriple} is based on bijective codings of maps by labelled plane trees. 
Its proof is a fairly routine generalization of existing arguments (mostly due to Jean-Fran\c{c}ois Le Gall). 
We have formulated Theorem~\ref{prop:CSTriple} in a general form as we expect it to be useful in proving convergence for other random map models, in particular for models falling within the framework of the ``master bijection'' of Bernardi and Fusy \cite{bernardi2012bij} and of the general bijection for blossoming trees recently described by Albenque and Poulalhon \cite{albenque2013generic}. We sketch the conditions under which Theorem~\ref{prop:CSTriple} applies in Section~\ref{sec:csrequirements}.

While the conditions under which we establish convergence to the Brownian map are rather general, {\em verifying} that a discrete random map ensemble satisfies these conditions can be rather involved. In many map ensembles of interest, the primary missing link is a  labelling rule for the vertices of a canonical spanning tree of the map, such that vertex labels encode distances to a specified root vertex. For the case of random simple triangulations and quadrangulations, we provide a labelling that does not {\em precisely} encode distances, but we show that the error is insignificant in the limit. Intriguingly, for distances to a specified root vertex, the error in the label is bounded by the {\em winding number} of an associated closed loop in the map. In Section~\ref{sec:overview}, we briefly describe the bijection between simple triangulations and certain labelled trees, on which our proof of Theorem~\ref{thm:main} is based, and further discuss the role of winding numbers. The appearance of a winding number hints that a discrete complex-analytic perspective may shed further light on the shape of geodesics in random simple triangulations and eventually in the Brownian map. 

One requirement of Theorem~\ref{prop:CSTriple} is the convergence of a suitable spatial branching process, after renormalization, to the Brownian snake. Such convergence is known in many settings, but in others lack of symmetry (symmetry between the labels of children of a single node, in the coding of maps by labelled trees) has posed an obstacle. We introduce a technique we call {\em partial symmetrization}, in which we choose a ``representative subtree'', then randomly permute the children of as many nodes of the subtree as possible without affecting the subtree's plane embedding. This introduces enough symmetry that we may appeal to known results to establish convergence to the Brownian snake. On the other hand, fixing a large subtree allows the partially symmetrized process to be related to the original labelled tree and so to the associated map. A detailed explanation of the partial symmetrization technique is easier to provide for a specific bijection, and we defer it to Section~\ref{sec:snake}.

We believe partial symmetrization may be used to show that the multi-type spatial branching processes coding random $p$-angulations (for odd $p\ge 5$) converge to the Brownian snake. Given the work of Miermont \cite{miermont13brownian} and of Le Gall \cite{jf}, this is the only missing element in a proof that $p$-angulations (and perhaps more general random maps with degrees given by suitable Boltzmann weights) converge to the Brownian map. We expect to return to this in a subsequent work. 

\addtocontents{toc}{\SkipTocEntry}
\subsection{The Brownian map} \label{sec:limitobject}
Given an interval $I\subset \R$ or $I \subset \N$ and a function $f:I \to \R$, for $s,t \in I$ with $s < t$ we \nomenclature[Fcheck]{$\check{f}$}{For a function $f:I \to \R$, $\check{f}(s,t)=\inf_{x \in [s,t] \cap I} f(s,t)$.}
write $\check{f}(s,t) = \inf_{x \in I \cap [s,t]} f(x)$, $\check{f}(t,s) = \inf_{x \in I \setminus (s,t)} f(x)$. We additionally let $\check{f}(s,s)=f(s)$ for all $s \in I$.

Let $\be=(\be(t),0 \le t \le 1)$ be a standard Brownian excursion 
\nomenclature[E]{$\mathbf{e}$}{A standard Brownian excursion, $\mathbf{e}=(\mathbf{e}(t),0 \le t \le 1)$.}
and, conditionally given $\be$, let $Z=(Z(t),0 \le t \le 1)$ be a centred Gaussian process such that $Z(0)=0$ and for $0 \le s \le t \le 1$, 
\[
\mathrm{Cov}(Z(s),Z(t)) = \check{\be}(s,t)\, .
\]
\nomenclature[Z]{$Z$}{``Brownian snake driven by $\mathbf{e}$''.}
We may and shall assume $Z$ is a.s.\ continuous; see~\citep[Section IV]{LeGallSnake} for a more detailed description of the construction of the pair $(\be,Z)$. 

Next, define an equivalence relation $\sim_{\be}$ as follows. 
\nomenclature[Asim]{$\sim_{\mathbf{e}}$}{Equivalence relation on $[0,1]$, $x \sim_{\mathbf{e}} y$ if 
$\mathbf{e}(x)=\mathbf{e}(y)=\check{\mathbf{e}}(x,y)$.}
For $0 \le x \le y \le 1$ let $x \sim_{\be} y$ if $\be(x)=\be(y)=\check{\be}(x,y)$. The Brownian Continuum Random Tree $(\cT_{\be},d_{\cT_{\be}})$ introduced in~\cite{AldCRT2} is defined as $[0,1]/\sim_{\be}$ equipped with distance $d_{\cT_{\be}}(x,y)=\mathbf{e}(x)+\mathbf{e}(y)-\check{\mathbf{e}}(x,y)$ for $0 \le x \le y \le 1$.

It can be verified that almost surely, for all $x,y \in [0,1]$, if $x \sim_{\be} y$ then $Z(x)=Z(y)$, so we may view $Z$ as having domain $\cT_{\be}$. Furthermore, $Z$ remains a.s.\ continuous on this domain.
Next, for $x,y \in [0,1]$ let 
\nomenclature[Dz]{$d_Z$}{For $x,y \in [0,1]$, $d_Z(x,y)=Z(x)+Z(y) - 2\max(\check{Z}(x,y),\check{Z}(y,x))$.}
\begin{equation}\label{eq:dzdef}
d_Z(x,y) = Z(x)+Z(y) - 2\max(\check{Z}(x,y),\check{Z}(y,x))\, .
\end{equation}
Then let $d^*$ be the largest pseudo-metric on $[0,1]$ satisfying that (a) for all $s,t \in [0,1]$, if $s \sim_{\be} t$ then $d_Z(s,t)=0$, and (b) $d^* \le d_Z$. 
\nomenclature[Dstar]{$d^*$}{Largest pseudo-metric on $[0,1]$ compatible with $\sim_{\mathbf{e}}$, with $d^* \le d_Z$.}
Let $S=[0,1]/\{d^*=0\}$, and let $d$ be the push-forward of $d^*$ to $S$. Finally, let $\mu$ be the push-forward of Lebesgue measure on $[0,1]$ to $S$. 
\nomenclature[Szdmu]{$(S,d,\mu)$}{The Brownian map}
The (measured) {\em Brownian map} is (a random variable with the law of) the triple $(S,d,\mu)$. This name was first used by Marckert and Mokkadem \cite{MR2294979}, who considered a notion of convergence for random maps different from that of the present work. 

For later use, let $\rho \in S$ be the equivalence class of the point $0$, and, 
\nomenclature[Rho]{$\rho$}{Equivalence class of $0$ in $S$}
writing $s^*\in [0,1]$ for the point where $Z$ attains its minimum value (this point is almost surely unique), 
\nomenclature[Ustar]{$u^*$}{Equivalence class in $S$ of point in $[0,1]$ where $Z$ attains its minimum value.} let $u^*\in S$ be the equivalence class of $s^*$.
Then Corollary~7.3 of \cite{jf} states that for $U$ and $V$ uniformly distributed on $[0,1]$, independent of $Z$ and of each other, 
\begin{equation}\label{eq:invariance_rerooting}
d^*(U,V) \eqdist d^*(U,s^{*}) \eqdist -\check{Z}(0,1) \eqdist Z(V)-\check{Z}(0,1). 
\end{equation}

\addtocontents{toc}{\SkipTocEntry}
\subsection{Sufficient conditions for convergence to the Brownian Map} \label{sec:csrequirements}
Our argument leans heavily on the {\em rerooting invariance} of the Brownian map ((\ref{eq:invariance_rerooting}), above). Given the convergence of some discrete ensemble to the Brownian map, if the discrete ensemble possesses rerooting invariance then this can be transferred to the Brownian map. However, to date this is the only known technique for establishing rerooting invariance of the Brownian map (and the key reason why our results depend on those of \cite{miermont13brownian,jf}). 

Informally, to prove convergence we need that the random rooted map $M_n$ can in some sense be described by a suitable pair of random functions $C_n:[0,1] \to [0,\infty)$ and $Z_n:[0,1] \to \R$. Often $C_n$ will be the (spatially and temporally rescaled, clockwise) contour process of some canonical rooted spanning tree $(T_n,\xi_n)$ of $M_n$, and for the sake of this informal description we assume this to be so. To establish convergence we require (versions of) the following. In what follows let $r_n \in [0,1]$ be such that $Z_n(r_n)=\min(Z_n(x),0 \le x \le 1)$, and write $d_{M_n}$ for (suitably rescaled) graph distance on $V(M_n)$. 
\begin{enumerate}
\item[1.] {\bf Distances to the minimum given by $Z_n$.} There is a vertex $u_n \in V(M_n)$ such that for all vertices $v$, if a clockwise contour exploration of $T_n$ visits $v$ at time $t$ then 
$Z_n(t)-Z_n(r_n)$ is $d_{M_n}(v,u_n)+o_n(1)$, where $o_n(1)$ represents an error that tends to zero in probability as $n \to \infty$. 
\item[2.] {\bf Distance bound via clockwise geodesics to the minimum.} For any pair of vertices $v,v'$ of $M_n$, if a clockwise contour exploration of $T_n$ visits $v$ and $v'$ at times $t$ and $t'$, respectively, then $d_{M_n}(v,v')$ is bounded from above by 
\[
Z_n(t)+Z_n(t')-2\max(\check{Z}_n(t,t'),\check{Z}_n(t',t)) + o_n(1). 
\]
\item[3.] {\bf Coding by the Brownian snake.} The pair $(C_n,Z_n)$ converges in distribution to $(\be,Z)$, for the topology of uniform convergence on $C([0,1],\R)^2$. 
\item[4.] {\bf Invariance under rerooting.} If $U_n,V_n$ are independent, uniformly random vertices of $M_n$, then $d_{M_n}(U_n,V_n)$ is asymptotically equal in distribution to $d_{M_n}(u_n,V_n)$.
\end{enumerate}
Briefly, given these properties the proof then proceeds as follows. Our argument closely follows one used by Le Gall to prove convergence of rescaled random (non-simple) triangulations to the Brownian map, once convergence for quadrangulations is known (\cite[Section 8]{jf}). It is useful to reparameterize so that all the metrics and pseudo-metrics under consideration are functions from $[0,1]^2$ to $[0,\infty)$; this can be accomplished by identifying the vertices of each metric space $M_n$ with a subset of $[0,1]$ and using bilinear interpolation. 

First, 1.\ and 2.\ together can be used to prove tightness of the sequence of laws of the functions $(d_{M_n},n \in \N)$, which implies convergence along subsequences. Thus, let $d:[0,1]^2 \to [0,\infty)$ be a subsequential limit of $d_{M_n}$. Our aim is to show that almost surely $d$ and $d^*$ (defined in Section~\ref{sec:limitobject}) are equal in law. 

Next,~1.\ says that distances to the point of minimum label are given by $Z_m$, a limiting analogue of which is also true in the Brownian map. Invariance under rerooting 4.\ and (\ref{eq:invariance_rerooting}) then yields that for $U,V$ independent and uniform on $[0,1]$, $d(U,V)$ is the limit in distribution of $-Z_n(r_n)$, so by 3.\ we obtain  $d(U,V)\eqdist -\min(Z(x),0 \le x \le 1)=d^*(U,V)$. 

Finally, 2.\ gives a bound for $d_{M_n}$ that is a finite-$n$ analogue of the bound (\ref{eq:dzdef}) for $d_Z$. Since $d^*$ is maximal subject to $d^* \le d_Z$, 3.\ then yields that $d$ is stochastically dominated by $d^*$. In other words, by working in a suitable probability space, we may assume $d(x,y) \le d^*(x,y)$ for almost every $(x,y) \in [0,1]^2$. The fact that $d(U,V) \eqdist d^*(U,V)$ then implies $d$ and $d^*$ are almost everywhere equal, so have the same law. 

\addtocontents{toc}{\SkipTocEntry}
\subsection{Labels and geodesics, and an overview of the proof}\label{sec:overview}
In this section (and throughout much of the rest of the paper), we restrict our attention to simple triangulations, as the details for simple quadrangulations are nearly identical. 

Fix a pair $(G,\xi)$ with $G$ a simple triangulation of $\S^2$ and $\xi$ a corner of $G$. View $G$ as embedded in $\R^2$ so the face containing $\xi$ is the unique unbounded (outer) face. With this embedding, list the vertices of the face containing $\xi$ in clockwise order as $v,A,B$, with $v$ incident to $\xi$. 
\nomenclature[3orient]{$3$-orientation}{Orientation of a triangulation so all vertices not on a distinguished face have outdegree $3$; also see Section~\ref{sec:def-ori}}
A \emph{$3$-orientation} of $(G,\xi)$ is an orientation $\overrightarrow{E}$ of $E(G)$ such that in $\overrightarrow{E}$, $A,B$, and $v$ have outdegrees $0,1$, and $2$, respectively, and all other vertices have outdegree three.\footnote{This is equivalent to, but differs very slightly from, the standard definition.}
Schnyder \cite{Schnyder} showed $(G,\xi)$ admits a 3-orientation if and only if $G$ is simple, and in this case admits a {\em 	unique} 3-orientation containing no counterclockwise cycles (we say an oriented cycle is {\em clockwise} if $\xi$ is on its left, and otherwise say it is counterclockwise); this $3$-orientation is called {\em minimal}. Let $\overrightarrow{E}$ be the minimal 3-orientation of $(G,\xi)$. 

The definitions of the following paragraph are illustrated in Figure~\ref{fig:intro}.
A subtree of $G$ containing the vertex $v$ incident to $\xi$ is {\em oriented} if all edges of the subtree are oriented towards $v$ in $\overrightarrow{E}$. It turns out there is a unique oriented subtree $T$ of $G$ on vertices $V(G)\setminus\{A,B\}$ which is {\em minimal} in the sense that for all edges $uw \in \overrightarrow{E}$ with $\{u,w\} \not\in E(T)$, if $uw$ attaches to $u$ and $w$ in corners $c$ and $c'$, respectively, then $c$ precedes $c'$ in a clockwise contour exploration of $T$ starting from $\xi$. We endow this tree $T$ with a labelling $Y:V(T) \to \N$ as follows. For $e=uw \in \overrightarrow{E}$ with $\{u,w\} \in E(G)$, the {\em leftmost oriented path} from $e$ to $A$ is the unique oriented path 
$(u_0,u_1,\ldots,u_k)$ with the following two properties: (i) $u_0=u$, $u_1=w$; (ii) for $1 \le i <k$, if $\{u_i,y\} \in E(G)$ and this edge attaches to the path $(u_0,\ldots,u_k)$ on the left, then $yu_i \in \overrightarrow{E}$. For each vertex $u \in V(T)$ distinct from $v$, there are three such paths starting at $u$ (since $u$ has outdegree three in $\overrightarrow{E}$); we let $P(u)=P_{G,\xi}(u)$ be one of the shortest such paths. Then let $Y(u) = |P(u)|$, the number of vertices in $P(u)$. 

\begin{figure}[ht]
\hspace{-0.5cm}
 \subfigure[A simple triangulation endowed with its unique 3-orientation with no counterclockwise cycles.]{\includegraphics[width=.28\linewidth,page=1]{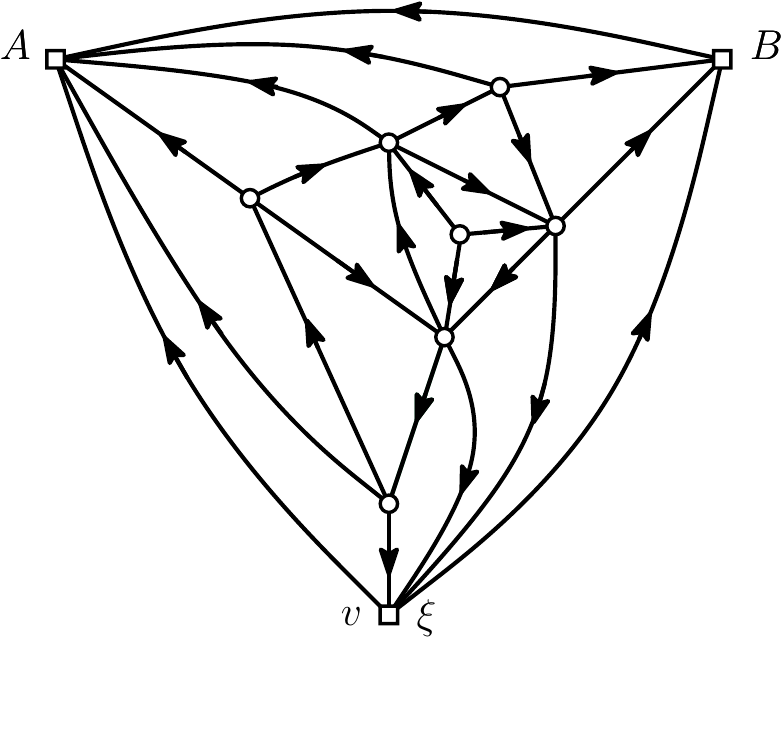}}\qquad
  \subfigure[The minimal oriented tree is drawn in dashed blue lines.]{\includegraphics[width=.28\linewidth,page=3]{Pictures/triangulation-intro.pdf}}
\qquad
\subfigure[The thick green paths are both leftmost oriented paths from $u$ to $A$; the solid path is $P(u)$, so $Y(u)=3$.]{\includegraphics[width=.28\linewidth,page=2]{Pictures/triangulation-intro.pdf}}
\caption{Orientations, spanning trees, and leftmost paths in simple triangulations} 
\label{fig:intro}
\end{figure}

Surprisingly, $(G,\xi)$ may be recovered from the pair $(T,Y)$. More strongly, the above transformation is a bijection mapping planted simple planar triangulations to a certain set of ``validly labelled'' planted plane trees. This bijection is essentially due to Poulalhon and Schaeffer~\cite{PouSch06}, but the connection of vertex labels with the lengths of certain oriented paths is new. 

Since $Y(u)$ is the number of vertices on a certain path from $u$ to $A$, $Y(u)-1$ is an upper bound on $d_G(u,A)$, the graph distance between $u$ and $A$ in $G$. It turns out that $Y(u)-d_G(u,A)-1$ is bounded by twice the number of times a {\em shortest} path in $G$ from $u$ to $A$ winds clockwise around the {\em leftmost} path $P_{G,\xi}(u)$. More strongly, if $P(u)=(u_0,u_1,\ldots,u_k)$ and $Q$ is a path from $u_i$ to $u_j$ disjoint from $P(u)$ except at its endpoints, then $|Q| \ge j-i-1$, and $|Q|\le j-i+1$ (i.e.\ $Q$ is a shortcut from $u_i$ to $u_j$) only if $Q$ leaves $u_i$ on the right and rejoins $u_j$ on the left. This fact allows $Y(u)-d_G(u,A)-1$ to be controlled as follows. 

Let $n=|V(G)|$. If $Q$ is a shortcut from $u_i$ to $u_j$ then the union of $Q$ and $u_{i+1},\ldots,u_{j-1}$ forms a cycle $C$ with $2(j-i)-1$ or $2(j-i)-2$ vertices. If there are $2k$ shortcuts between $u$ and $A$ and $Q$ is the $k$'th one, then all vertices of $C$ have distance at least $k$ both from $A$ and from $u$. It will follow that typically (i.e., for random $G$), when $k$ and $d_G(u_j,A)$ are both large (of order $n^{1/4}$) then $j-i$ should also be large (of order $n^{1/4}$), or else $G$ would contain a cycle of length $o(n^{1/4})$ separating two macroscopic regions. On the other hand, a ``shortcut'' of length of order $n^{1/4}$ is rather long; we will straightforwardly show that typically the diameter of $G$ will be $O(n^{1/4})$, in which case there can be at most a bounded number of such long shortcuts on any path. A rigorous version of this argument allows us to show that typically, for all $u \in V(T)\setminus\{v\}=V(G)\setminus\{v,A,B\}$, $Y(u)-d_G(u,A)-1$ is much smaller than $n^{1/4}$. In other words, after rescaling, the labels $Y$ with high probability provide good approximations for distances to the root $A$. This essentially proves 1.\ from Section~\ref{sec:csrequirements}. 

A modification of the above argument 
establishes without too much difficulty that for $u,w \in V(T)$ with $u$ preceding $w$ in lexicographic order, $d_G(u,w)$ is bounded by $Y(u)+Y(w)-2\check{Y}(u,w)+2$, where $\check{Y}(u,w)$ is the smallest value $Y(y)$ for any vertex $y$ following $u$ and preceding $w$ in lexicographic order. This will establish (2) from Section~\ref{sec:csrequirements}. 

To establish (3) we use ``partial symmetrization'' as previously discussed. Finally, rerooting invariance, (4), will be a straightforward consequence of choosing a random root corner. Having verified all the conditions of our general convergence result (whose proof was already sketched), Theorem~\ref{thm:main} for simple triangulations then follows immediately. An essentially identical development establishes Theorem~\ref{thm:main} for simple quadrangulations. 

\addtocontents{toc}{\SkipTocEntry}
\subsection{Outline}
We conclude the introduction by fixing some basic notation, in Section~\ref{sec:notation}. In Section~\ref{sec:pgpt} we provide definitions related to planar maps and plane trees, many of which are standard. In Section~\ref{sec:notion_convergence} we introduce the Gromov--Hausdorff distance and mention some of its basic properties. 
In Section~\ref{sec:cstriples} we formally state our ``universality'' result, providing general sufficient conditions for a random map ensemble to converge to the Brownian map; proofs are deferred to Appendix~\ref{sec:cstripleproof}. 
In Section~\ref{sec:bij-bloss-trees} we describe the bijections for simple triangulations and quadrangulations on which our proof of Theorem~\ref{thm:main} is based. In Section~\ref{sec:snake} we prove convergence of the spatial branching process associated to a random simple triangulation to the Brownian snake; this is where partial symmetrization appears. In Section~\ref{sec:label_dist_determinist} we study the relation of distances with labels; this is where winding numbers appear. In Section~\ref{sec:label_approx}, we use the bounds of Section~\ref{sec:label_dist_determinist} to show that our labelling provides a sufficiently close approximation of distances in random simple triangulations that the associated conditions of Theorem~\ref{prop:CSTriple} are satisfied. In Section~\ref{sec:mainthm} we establish rerooting invariance and so complete the proof of Theorem~\ref{thm:main}. Finally, Section~\ref{seq:quad} proves Theorem~\ref{thm:main} for quadrangulations, and Appendix~\ref{sec:notes} contains a derivation of the numerical constants from Theorem~\ref{thm:main}.

\addtocontents{toc}{\SkipTocEntry}
\subsection{Notation}\label{sec:notation}
For the remainder of the paper, all graphs are connected, finite, simple (i.e. without loops nor multiple edges) and planar. Let $G=(V(G),E(G))$ be such a graph. Given a vertex $v \in V(G)$ we write $\deg_G(v)=|\{e
\in E(G): v \in e\}|$ for the degree of $v$ in $G$. If $v \in e$ we say $e$ is {\em incident} to
$v$. We write $d_G:V(G) \times V(G) \to \N$ for graph distance on $G$. Given $W \subset V(G)$, we write $G[W]$ for the graph with vertices $W$ and edges $\{\{u,v\} \in E(G): u,v \in W\}$. 

An {\em oriented edge} of $G$ is an ordered pair $uw$, where $\{u,w\} \in E(G)$; we call $uw$ an {\em orientation} of $\{u,w\}$. 
\nomenclature[Edirect]{$\overrightarrow{E}$}{An orientation of the edges of a graph $G$; usually, the minimal $3$-orientation of a planted triangulation. See also Section~\ref{sec:def-ori}.}
An orientation of $G$ is a set $\overrightarrow{E}=\{\overrightarrow{e}:e \in E(G)\}$, where for each $e \in E(G)$, $\overrightarrow{e}$ is an orientation of $e$. The {\em outdegree} of $v \in V(G)$ (with respect to $\overrightarrow{E}$) is $\deg^+(v)=\deg^+_{\overrightarrow{E}}(v)=|\{w \in V(G): vw \in \overrightarrow{E}\}|$. 

If $S=(s_1,\ldots,s_r)$ is any sequence of objects, we say that $S$ has length $r$ and write $|S|=r$. A {\em path} in $G$ is a sequence $P=(u_0,u_1,\ldots,u_k)$ of vertices of $G$ with $\{u_i,u_{i+1}\} \in E(G)$ for $0 \le i < k$; we say $P$ is a path from $u_0$ to $u_k$, and note that $|P|=k+1$.  A path is {\em simple} if all its vertices are distinct.  A {\em cycle} in $G$ is a path $(u_0,u_1,\ldots,u_k,u_{k+1})$ such that $u_{k+1}=u_0$; it is {\em simple} if $(u_0,\ldots,u_k)$ is a simple path. If $G$ is a {\em tree} (connected and acyclic) then for $u,w \in G$ we write $\bbr{u,v}$ for the unique (shortest) path in $G$ from $u$ to $v$. 
Finally, for a non-negative integer $k$, write $[k]=\{0,1,\ldots,k\}$, 

\section{Planar maps and plane trees}\label{sec:pgpt}
\addtocontents{toc}{\SkipTocEntry}
\subsection{Planar maps}\label{sec:pg}
A {\em planar embedding} of $G$ is a function $\phi:V(G) \cup E(G) \to \S^2$ satisfying the following properties. 
\begin{enumerate}
\item The restriction $\phi|_{V(G)}$ is injective.
\item For each $e=uv \in E(G)$, $\phi(e)$ is a simple curve with endpoints $\phi(u)$ and $\phi(v)$. 
\item For any two distinct edges $e,f \in E(G)$, the curves $\phi(e)$ and $\phi(f)$ are disjoint except possibly at their endpoints. 
\end{enumerate}
The pair $(G,\phi)$ is called a {\em planar map}. The faces of $(G,\phi)$ are the connected components of $\S^2\setminus \bigcup_{x \in V(G) \cup E(G)} \phi(x)$. 
Given a face $f$ the vertices and edges incident to $f$ are given by the set $\phi^{-1}(\partial f)$, where $\partial f$ is the boundary of $f$. 

Two planar maps are {\em isomorphic} if there exists an orientation-preserving
homeomorphism of $\S^2$ that sends one to the other. It is easily verified
that planar map isomorphism is an equivalence relation. 

For any planar map $(G,\phi)$, for each vertex $v \in V(G)$ there is a unique cyclic (clockwise) ordering $\cO_v$ of the edges incident to $v$. Furthermore, up to isomorphism, the set of orderings $\{\cO_v:v \in V(G)\}$ uniquely determines $(G,\phi)$. We may therefore specify the isomorphism equivalence class of $(G,\phi)$ by providing $G$ and the set of cyclic orderings associated to $(G,\phi)$. We will henceforth denote (a representative from the isomorphism equivalence class of) a planar map simply by $G$, leaving implicit both $\phi$ and its associated cyclic orderings. 

For the remainder of Section~\ref{sec:pg}, consider a fixed planar map $G$. A {\em corner} of $G$ is an ordered pair $\xi=(e,e')$ where $e$ and $e'$ are incident to a common vertex $v$, and $e'$ immediately follows $e$ in the clockwise order around $v$.\footnote{We allow that $e=e'$, which can happen if $d_G(v)=1$.}  
\nomenclature[v]{$\mathrm{v}$}{$\mathrm{v}(\xi)=\mathrm{v}_G(\xi)$ is the vertex incident to corner $\xi$ in $G$}
We write $\v(\xi)=\v_{G}(\xi)=v$ and say that $\xi$ is incident to $v$ (and also to $e$ and $e'$). We write $\cC(G)$ for the set of corners of $G$. 
\nomenclature[Cg]{$\mathcal{C}(G)$}{Set of corners of planar map $G$}
For $\xi,\xi' \in \cC(G)$ we let $d_G(\xi,\xi') = d_G(\v(\xi),\v(\xi'))$ be the graph distance between the vertices incident to $\xi$ and $\xi'$, and likewise let $d_G(\xi,w)=d_G(\v(\xi),w)$ for $w \in V(G)$.

If $e=\{u,v\}$ and $e'=\{v,w\}$, and $f$ is the face on the left when following $e$ and $e'$ from $u$ through $v$ to $w$, then we say $\xi=(e,e')$ is incident to $f$ and vice-versa. The {\em degree} of $f$ is the number of corners incident to $f$. The planar map $G$ is a {\em triangulation} or a {\em quadrangulation} if all its faces have respectively degree 3 or degree 4. 

Given $e=\{u,v\} \in E(G)$, write $\kappa^{\ell}(u,v)=\kappa^{\ell}_G(u,v)$ (respectively, $\kappa^{r}(u,v)=\kappa^{r}_G(u,v)$) for the corner incident to $u$ and to $\{u,v\}$ that is on the left (respectively, on the right) when following $e$ from $u$ to $v$. 

A {\em planted planar map} is a pair $(G,\xi)$, where $G$ is a planar map and $\xi \in \cC(G)$. We call $\xi$ the {\em root corner} of $(G,\xi)$, call $\v(\xi)$ its {\em root vertex}, and call the face of $G$ incident to $\xi$ its {\em root face}. If $G'$ is a connected subgraph of $G$ containing $\xi$, then 
$(G',\xi)$ is again a planar map, and we call it a {\em planted submap} of $(G,\xi)$. 

\addtocontents{toc}{\SkipTocEntry}
\subsection{Plane trees}\label{sub:PlaneTrees}
A plane tree (resp. planted plane tree)  is a planar map $G$ (resp.\ planted planar map $(G,\xi)$) such that $G$ is a tree\footnote{It is relatively common to define a planted plane tree as a pair $(T,v)$ where $T$ is a plane tree and $v$ is a degree-one vertex of $T$. Our definition, which is equivalent, can be recovered by deleting the plant vertex and its incident edge, and rooting at the corner thereby created.}.
If $\rT=(T,\xi)$ is a planted plane tree then recalling that $\v(\xi)$ is the root vertex of $\rT$, we may speak of parents, children, ancestors, descendants in the usual way. For each $w \in V(T)$ we write $|w|=d_T(\xi,w)$, and call $|w|$ the generation of $w$. We also write $k(w)=k_T(w)$ for the number of children of $w$, and if $w \ne \v(\xi)$ then we write $p(w)=p_{\rT}(w)$ for the parent of $w$.

\nomenclature[Ut]{$\mathrm{U}_{\mathrm{T}}$}{Ulam--Harris encoding of $\mathrm{T}$}
The {\em Ulam--Harris encoding} is the injective function $U=U_{\rT}:V(T) \to \bigcup_{i \ge 0} \mathbb{N}^i$ defined as follows (let $\N^0=\{\emptyset\}$ by convention). First, set $U(\v(\xi)) = \emptyset$. 
For every other vertex $w \in V(T)$, consider the unique path $\v(\xi)=v_0,v_1,\ldots,v_k=w$ from $\v(\xi)$ to $w$. 
For $1 \le i \le k$ let $n_i$ be such that $v_i$ is the $n_i$'th child of $v_{i-1}$, in cyclic order around $v_{i-1}$ starting from $\kr{v_{i-1}}{v_{i-2}}$ if $i \ge 2$ or from $\xi$ if $i=1$. Then set $U(w)=n_1n_2\ldots n_k \in \N^{k}$. 
In other words, the root receives label $\emptyset$ and for each $i \ge 1$ the label of any $i$'th child is obtained recursively by concatenating the integer $i$ to the label of its parent. It is easily verified that (the isomorphism class of) $\rT$ can be recovered from the set of labels $\{U(v): v \in V(T)\}$. 

When there is no ambiguity, we identify planted plane trees with their Ulam-Harris encodings. In particular, in this case the root vertex is denoted $\emptyset$ and if $v$ is a vertex of $T$, then its children are denoted $v1,\ldots,vk$, where $k=k_T(v)$. 
\smallskip

The {\em lexicographic ordering} $\pl=\plt$ of $V(T)$ is the total order of $V(T)$ induced by the lexicographic order on $\{U(v): v \in V(T)\}$. This ordering induces a lexicographic ordering of $E(T)$ 
\nomenclature[Apl]{$\preceq_{\mathrm{lex,T}}$}{Lexicographic ordering of vertices or edges of planted plane tree $\mathrm{T}$}
(also denoted $\pl=\plt$ by a slight abuse of notation) by defining $\{u,v\}\plt \{u',v'\}$ if and only if $u,v\plt u'$ or $u,v\plt v'$. These are the orders in which a clockwise contour exploration of the plane tree $T$ starting from $\xi$ first visits the vertices and edges of $T$, respectively.

\nomenclature[Rt]{$r_{\mathrm{T}}$}{Contour exploration of planted plane tree $\mathrm{T}$, $r_{\mathrm{T}}:[2\lvert V(T)\rvert-2] \to V(T)$.}
The {\em contour} exploration $r=r_{\rT}:[2|V(T)|-2] \to V(T)$ is inductively defined as follows. Let $r(0)=\v(\xi)$. Then, for $1 \le i \le 2|V(T)|-2$, let $r(i)$ be the lexicographically first child of $r(i-1)$ that is not an element of $\{r(0),\ldots,r(i-1)\}$, or let $r(i)$ be the parent of $r(i-1)$ if no such node exists.
Note that each vertex $v \in V(T)\setminus \{\v(\xi)\}$ appears $\deg_T(v)$ times in the contour exploration, and $\v(\xi)$ appears $\deg_T(\v(\xi))+1$ times.

The contour exploration induces an ordering of $\cC(T)$, as follows. 
\nomenclature[Apl]{$\preceq_{\mathrm{ctr,T}}$}{Contour ordering of corners of planted plane tree $\mathrm{T}$}
For $0 \le i < 2|V(T)|-2$, let $e(i)=e_{\rT}(i)=\{r(i),r(i+1)\}$. 
Then let $\xi(0)=\xi_{\rT}(0)=\xi$, and for $1 \le i < 2|V(T)|-2$ let $\xi(i)=\xi_{\rT}(i)=(e(i-1),e(i))$.  The {\em contour ordering}, denoted $\pc=\pct$, is the total order of $\cC(T)$ induced by $(\xi(i),0 \le i < 2|V(T)|-2)$.  
 For convenience, also let $\xi(2|V(T)|-2)=\xi_{\rT}(2|V(T)|-2) = \xi$. 
 \nomenclature[Apl]{$\preceq_{\mathrm{cyc,T}}$}{Cyclic ordering of corners of $\mathrm{T}$ induced by $\preceq_{\mathrm{ctr,T}}$}
 Finally, write $\pcy=\pcyt$ for the cyclic order on $\cC(T)$ induced by $\pct$. It can be verified that $\pcy$ does not depend on the choice of root corner $\xi$. We define cyclic intervals accordingly: for $c,c' \in \cC(T)$, let
 \[
 \bbrcy{c,c'} = 	\begin{cases}
 				\{c'': c \pc c'' \pc c'\}& \mbox{ if } c \pc c'\, , \\
				\{c'': c'' \pc c \mbox{ or } c'\pc  c''\} & \mbox{ if } c' \pc c\, .
 			\end{cases}
 \]

Given $u, v \in V(T)$, we say that $v$ is the {\em successor} of $u$ if $u \pl v$ and for all $w \in V(T)$, if $u \pl w \pl v$ then $w=u$ or $w=v$. We define successorship for corners similarly. 

Given a plane tree $\rT=(T,\xi)$ and a set $R \subset V(T)$ with $\v(\xi) \in R$, 
\nomenclature[Tred]{$\mathrm{T}\angles{R}$}{Subtree of $\mathrm{T}$ spanned by $R$.}
Also, the subtree of $\rT$ {\em spanned} by $R$, denoted $\rT\angles{R}$, is the subtree of $\rT$ induced by the union of the shortest paths between all pairs of vertices in $R$. Note that $\rT\angles{R}$ naturally inherits a planted plane tree structure from $\rT$.

\addtocontents{toc}{\SkipTocEntry}
\subsection{The contour process and spatial plane trees}\label{sec:contourproc}
A {\em spatial plane tree} is a triple $\rT=(T,\xi,D)$, where $(T,\xi)$ is a planted plane tree and $D:E(T) \to \R$ is an arbitrary function. Given a labelled plane tree, define a function 
\nomenclature[Xt]{$X_{\mathrm{T}}$}{For $(T,\xi,D)$ a spatial planted plane tree and $v \in V(T)$, $X(v)$ is sum of displacements on root-to-$v$ path.}
$X=X_{\rT}:V(T)\rightarrow \R$ as follows. First, let $X(\v(\xi))=0$. Next, given $u \in V(T)$ with $X(u)$ already defined, for $1 \le i \le k_\rT(u)$ let $X(ui)=X(u)+D(u,ui)$. We call $X_{\rT}$ the \emph{labelling function} of $\rT$.

Now define $C([0,1],\R)$ functions $C_{\rT}$ and $Z_{\rT}$ by setting 
\[
C_{\rT}(i/(2|V(T)|-2)) = d_{T}(\xi,r_{(T,\xi)}(i)) \quad \mbox{and} \quad Z_{\rT}(i/(2|V(T)|-2))= X_{\rT}(r_{(T,\xi)}(i))\, ,
\]
for $i \in \{0,1,\ldots,2|V(T)|-2\}$, and extending each function to $[0,1]$ by linear interpolation. 
\nomenclature[Ct]{$C_{\mathrm{T}}$}{Contour process of $\mathrm{T}$}
\nomenclature[Zt]{$Z_{\mathrm{T}}$}{The spatial process of $\mathrm{T}$; $Z_{\mathrm{T}}$ is $X_{\mathrm{T}}$, continuized, temporally rescaled to have domain $[0,1]$.}
We refer to $C_{\rT}$ and $Z_{\rT}$ as the {\em contour} and {\em labelling} processes of $\rT$, respectively. Note that the definition of $C_{\rT}$ does not depend on the function $D$.
\addtocontents{toc}{\SkipTocEntry}
\subsection{Spanning trees in planar maps.}\label{sec:spantree}

Given a planar map $G$, a  {\em spanning tree} of $\rG$ is a subgraph $T$ of $G$ such that $T$ is a tree with $V(T)=V(G)$. 
If $(G,\xi)$ is a planted planar map and $T$ is a spanning tree of $G$ then we call $(T,\xi)$ a planted spanning tree of $(G,\xi)$. 
 
 Finally, given a planted planar map $\rG=(G,\xi)$ and an orientation $\overrightarrow{E}$ of $E(G)$, we say that a planted spanning tree $(T,\xi)$ of $\rG$ is {\em oriented} with respect to $\overrightarrow{E}$ if in the orientation of $E(T)$ obtained from $\overrightarrow{E}$ by restriction, all edges are oriented towards $\v(\xi)$.

\section{Distances between metric spaces: Gromov, Hausdorff, and Prokhorov}\label{sec:notion_convergence}

\addtocontents{toc}{\SkipTocEntry}
\subsection*{The Gromov--Hausdorff distance}

For proofs of the assertions in this section, and for further details, we refer the reader to \cite{bbi,MiermontTessellations}. 
Let $\mathrm X = (X,d)$ and $\mathrm X' = (X',d')$ be compact metric spaces.
Given $C \subset X\times X'$, the {\em distortion} of $C$, denoted $\mathrm{dis}(C)$, is the quantity  
\[
\mathrm{dis}(C) = \sup\{|d(x,y)-d'(x',y')|: (x,x') \in C, (y,y') \in C\}.
\]
\nomenclature[Dis]{$\mathrm{dis}(C)$}{Distortion of the correspondence $C$; equal to $\sup\{\lvert d(x,y)-d'(x',y')\rvert: (x,x') \in C, (y,y') \in C\}$.}
A \emph{correspondence} between $\rX$ and $\rX'$ is a set $C \subset X \times X'$ such that for every $x \in X$ 
there is $x' \in X'$ such that $(x,x') \in C$ and vice versa.  
We write $C(X,X')$ for the set of correspondences between $X$ and $X'$. 
\nomenclature[Cxx]{$C(X,X')$}{Set of correspondences between $X$ and $X'$.}
The Gromov--Hausdorff distance $\dgh(\rX,\rX')$ between metric spaces $\rX=(X,d)$ and $\rX'=(X',d')$ is
\nomenclature[Dgh]{$\dgh(\mathrm X,\mathrm X')$}{Gromov--Hausdorff distance between $\rX$ and $\rX'$; equal to $\frac{1}{2}\inf\{ \dis(C): C \in C(X,X')\}$. Same section for $\dgh^k$, $\dghp$.}
\[
\dgh(\rX,\rX') = \frac{1}{2}\inf\{ \dis(C): C \in C(X,X')\}.
\]
We list without proof some basic properties of $\dgh$. Let $\cM$ be the set of isometry classes of compact metric spaces. 
\begin{enumerate}
\item Given metric spaces $\rX=(X,d)$ and $\rX'=(X',d')$, there exists $C \in C(X,X')$ such that $\dgh(\rX,\rX')=\dis(C)/2$. 
\item If $\rX_1$ and $\rX_2$ are isometric, and $\rX_1'$ and $\rX_2'$ are isometric, then $\dgh(\rX_1,\rX_1')=\dgh(\rX_2,\rX_2')$. In other words, $\dgh$ is a {\em class function} for $\cM$. 
\item The push-forward of $\dgh$ to $\cM$ (which we continue to denote $\dgh$) is a distance on $\cM$, and $(\cM,\dgh)$ is a complete separable metric space. 
\end{enumerate}
\nomenclature[Md]{$(\mathcal{M},d_{\mathrm{GH}})$}{Set of isometry classes of compact metric spaces with GH distance; see same section for $(\mathcal{M}^{(k)},d_{\mathrm{GH}}^{k})$ and $(\mathcal{M}_w,d_{\mathrm{GHP}})$.}

A {\em $k$-pointed metric space} is a triple $(X,d,(x_1,\ldots,x_k))$ where 
$(X,d)$ is a metric space and $x_i \in X$ for $1 \le i \le k$. 
We say $k$-pointed metric spaces $\rX=(X,d,(x_1,\ldots,x_k))$ and $\rX'=(X',d',(x'_1,\ldots,x'_k))$
are {\em isometry-equivalent} if there exists a bijective isometry $f:X\to X'$ such that $f(x_i)=x_i'$ for $1 \le i \le k$. 
The {\em $k$-pointed Gromov--Hausdorff distance} $\dgh^k$ between $\rX,\rX'$ is given by 
\[
\dgh^k(\mathrm X, \mathrm X') = \frac{1}{2}\inf\left\{\dis(C) :C \in
  C(X,X') \mbox{ and } (x_i,x'_i) \in C, 1 \leq i\leq
  k\right\}.
\]
Much as before, if $\cM^{(k)}$ is the set of isometry-equivalence classes of $k$-pointed compact metric spaces, then $\dgh^k$ is a class function for $\cM^{(k)}$ so may be viewed as having domain $\cM^{(k)}$, and $(\cM^{(k)},\dgh^k)$ then forms a complete separable metric space. 

\addtocontents{toc}{\SkipTocEntry}
\subsection*{The Gromov--Hausdorff--Prokhorov distance}
Following \cite{MiermontTessellations}, a {\em weighted metric space} is a triple $(X,d,\mu)$ such that $(X,d)$ is a metric space and $\mu$ is a Borel probability measure on $(X,d)$. Weighted metric spaces $(X,d,\mu)$ and $(X',d',\mu')$ are {\em isometry-equivalent} if there exists a measurable bijective isometry $\phi:X \to X'$ such that $\phi_*\mu=\mu'$, where $\phi_*\mu$ denotes the push-forward of $\mu$ under $\phi$. Write $\cM_{w}$ for the set of isometry-equivalence classes of weighted compact metric spaces. 

Given weighted metric spaces $\rX=(X,d,\mu)$ and $\rX'=(X',d',\mu')$, a {\em coupling} between $\mu$ and $\mu'$ is a Borel measure $\nu$ on $X\times X'$ (for the product metric) with $\pi_* \nu=\mu$ and $\pi'_* \nu=\mu'$, where $\pi:X\times X' \to X$ and $\pi':X \times X' \to X'$ are the projection maps. 
\nomenclature[Cyoupling]{Coupling}{A coupling of prob.\ measures $\mu$ on $X$, $\mu'$ on $X'$ is a prob.\ measure $\nu$ on $X\times X'$ with marginals $\mu,\mu'$.}
Let $M(\mu,\mu')$ be the set of couplings between $\mu$ and $\mu'$. The {\em Gromov--Hausdorff--Prokhorov distance} is defined by 
\[
\dghp(\rX,\rX')=\inf\left\{\eps> 0: \exists C \in C(X,X'),\exists \nu \in M(\mu,\mu'), \nu(C) \ge 1-\eps,\dis(C) \le 2\eps  \right\}.
\]
The push-forward of $\dghp$ to $\cM_w$, which we again denote $\dghp$, is a distance on $\cM_w$, and $(\cM_w,\dghp)$ is a complete separable metric space (see \cite[Section 6]{MiermontTessellations} and \cite[Section 2]{EvansWinter}).

 \section{Map encodings}\label{sec:cstriples}
The purpose of this section is to state sufficient conditions for a family of random maps to converge to the Brownian map after rescaling. The framework we describe enables us to use use the convergence argument the same line of argument as in Le Gall \cite{jf} with only minor modifications (which are essentially to ensure that the convergence holds in the Gromov-Hausdorff-Prokhorov sense and not only in the Gromov-Hausdorff sense). Our choice to work in a slightly more abstract setting was motivated by potential applications to several models of maps for which convergence to the Brownian map is yet to be established. We return to this point at the end of the section.
\smallskip

\nomenclature[P]{$\mathrm{P}$}{Map encoding, $\mathrm{P}=(\mathrm{M},\mathrm{T})$.}
A {\em map encoding} is a pair $\rP=(\rM,\rT)$ where  
$\rM=(M,\zeta)$ is a planted planar map and $\rT=(T,\xi,D)$ is a spatial plane tree with $V(T) \subset V(M)$.
Note that although $T$ shares its vertices with $M$, it need not be a subgraph of $M$. 

Fix a sequence $\rP=(\rP_n,n \ge 1)$ of random map encodings. 
Write $\rP_n=(\rM_n,\rT_n)$, write $C_n$ and $Z_n$ for the contour and label processes of $\rT_n$, respectively, and write $X_n$ and $r_n$ for the labelling function of $\rT_n$ and for the contour exploration of $\rT_n$, respectively. The sequence $\rP$ is {\em good} 
if there exist sequences $(a_n,n \in \N)$ and $(b_n,n \in \N)$ such that 
 the following three properties hold. 
 
\begin{framed}
\noindent {\bf 1.} As $n \to \infty$, $\pran{a_nC_n,b_nZ_n} \convdist (\be,Z)$ in the topology of uniform convergence on $C([0,1],\R)^2$, 
where $(\be,Z)$ is as described in Section~\ref{sec:limitobject}. 

\noindent {\bf 2.} (i) 
For all $\eps > 0$, 
\[
\lim_{n \to \infty} \p{b_n\cdot \max_{v \in V(M_n)} d_{M_n}(v,V(T_n)) > \eps } = 0\, .
\]
(ii) Write $d_{\mathrm{Prok}}$ for the Prokhorov distance between Borel measures on $\mathbb{R}$. For each $n$, conditionally given $\rP_n$, let $U_n,V_n$ be independent uniformly random elements of $V(T_n)$. Then 
\[
\lim_{n \to \infty} b_n \cdot d_{\mathrm{Prok}}(d_{M_n}(\zeta_n,\xi_n),d_{M_n}(U_n,V_n))
=0\, .
\] 
\noindent {\bf 3.}  
(i) Let $m=m(n)=2|V(T_n)|-2$. Then for all $\eps > 0$, 
\begin{align*}
 \lim_{n \to \infty} & \mathbf{P}\left\{\exists i,j \in [m]:d_{M_n}(r_n(i),r_n(j)) \ge 
 \right.  \\
& \qquad
\left.   
Z_n(i/m) + Z_n(j/m) - 
2\max\left(\check{Z}_n(i/m,j/m),\check{Z}_n(j/m,i/m)\right) + \eps b_n^{-1}\right\} \\
& =0\, . 
\end{align*}
(ii) For all $\eps > 0$, 
\[
\lim_{n \to \infty} 
\p{\exists j \in [m]:d_{M_n}(r_n(j),\zeta_n) \le 
Z_n(j/m) - \check{Z}_n(0,1) - \eps b_n^{-1}} = 0\, .
\]
\end{framed}
For later use, we note one consequence of {\bf 3.\ } 
Let $I_n$ be minimal such that $Z_{n}(I_n/m) = \check{Z}_n(0,1)$, {\bf 3.}(ii) implies that 
\[
\lim_{n \to \infty}\p{d_{M_n}(r_n(I_n),\zeta_n)| > \eps b_n^{-1}} = 0\, .
\]
Together with {\bf 3.}(i) and {\bf 3.}(ii) this yields that, for all $\eps>0$, 
\begin{align}\label{eq:3cons}
\lim_{n \to \infty} \p{\exists j \in [m]:|d_{M_n}(r_n(j),\zeta_n) - (X_n(r_n(j))-X_n(r_n(I_n)))| > \frac{\eps}{b_n}} = 0\, .
\end{align}
In other words, for $u \in V(T_n)$, the distance $d_{M_n}(u,\zeta_n)$ is essentially given by the difference between the label of $u$ and the infimum of labels in $T_n$. 
\begin{thm}\label{prop:CSTriple}
If $\rP$ is a good sequence of random map encodings then, writing $\mu_n$ for the uniform probability measure on $V(T_n)\subset V(M_n)$, we have 
\[
(V(M_n),b_nd_{M_n},\mu_n) \convdist (S,d,\mu)
\]
for $\dghp$, where $(S,d,\mu)$ is the Brownian map, as defined in Section~\ref{sec:limitobject}. 
\end{thm}
The proof of Theorem~\ref{prop:CSTriple}, which closely follows an argument of Le Gall \cite{jf} (as mentioned above), appears in Appendix~\ref{sec:cstripleproof}.  We conclude the section by mentioning one corollary of the theorem; we are slightly informal to avoid notational excess and as the argument is straightforward.  For $n,k \ge 1$, conditionally given $\rP_n$, let $U_{n,1},\ldots,U_{n,k}$ be independent with law $\mu_n$. Proposition 10 of \cite{MiermontTessellations} implies that if the convergence in Theorem~\ref{prop:CSTriple} holds then also 
\[
(V(M_n),b_nd_{M_n},(U_{n,1},\ldots,U_{n,k})) \convdist (S,d,(U_1,\ldots,U_k))\, ,
\]
for $\dgh^k$, where conditionally given $(S,d,\mu)$, $U_1,\ldots,U_k$ are independent with law $\mu$. 
By Proposition 8.2 of \cite{LeGallGeo}, conditionally given $(S,d,\mu)$, the points $\rho,u^* \in S$ are independent with law $\mu$; by {\bf 2.}(ii) it follows that $(V(M_n),b_nd_{M_n},(\xi_n,\zeta_n)) \convdist (S,d,(\rho,u^*))$ 
for $\dgh^2$.
\begin{rem}
The motivation underlying the introduction of good random map encodings is to define a general framework which can be used in future work as a ``black box'' to establish the convergence of various families of maps towards the Brownian map. In order to justify this, we provide some specific examples (though not an exhaustive list) of settings where we believe our generalization will be of use. 

Condition~{\bf 2.}(i) states that after rescaling distances by $b_n$, all vertices of the map $M_n$ are with high probability close to some tree vertex. In the present work, it turns out that only two vertices of the $M_n$ do not belong to $T_n$. In some models of maps, however (e.g. simple maps, see \cite{cartessimples}), it only holds that at least one vertex per face of the map belongs to the associated tree. The maximum face degree in a random simple map is typically logarithmic in the size of the map, so in that setting the strength of Condition~{\bf 2.}(i) is useful. 

Condition~{\bf 2.}(ii) requires the distance between the root of the map and the root of the tree to be asymptotically equal in distribution to the distance between two uniform vertices of the map. In the case of simple triangulations, the distance between the two roots is actually exactly distributed as the distance between two uniformly random points. However, it happens frequently that in bijections between maps and trees, the root of the map plays a special role, and is not precisely uniformly distributed. For example, in studying $3$-connected maps, a family of maps naturally arises for which all non-root faces are quadrangles, but the root face is a hexagon \cite{FusyPoulalhonSchaeffer}. 

Finally, for the classical case of uniform quadrangulations, conditions~{\bf 3.}(i) and~{\bf 3.}(ii) hold true without the term $\epsilon b_n^{-1}$. However in the present work, we can only prove that labels of the tree control distances in the maps up to an error term which is $o(b_n)$ in probability, so we require the full strength of {\bf 3.}(i) and~{\bf 3.}(ii). 
\end{rem}

\section{Bijections for simple triangulations}\label{sec:bij-bloss-trees} 
We start with a summary of the results of the section; to do so some definitions are needed. 
\nomenclature{Blossoming tree}{$T$ is $k$-blossoming if each non-leaf is incident to exactly $k$ leaves.}
For integer $k \ge 1$, a plane tree $T$ is a {\em $k$-blossoming tree} if each vertex of degree greater than one is incident to exactly $k$ vertices of degree one. 
If $T$ is a $k$-blossoming tree (for some $k$), 
\nomenclature[Bt]{$\mathcal{B}(T)$}{The blossoms of blossoming tree $T$}
we write $\cB=\cB(T)$ for the set of degree-one vertices of $T$. When it causes no ambiguity, we identify vertices of $\cB$ with their incident corners. Note that both $k$ and $\cB$ are uniquely determined by $T$. We call $\cB$ the {\em blossoms} of $T$, and $V(T)\setminus \cB$ the inner vertices of $T$. Also, an edge between two inner vertices is called an inner edge, and an edge between an inner vertex and a blossom is a \emph{stem}. A corner $c$ is an inner corner if $c \not \in \cB$. A {\em planted} $k$-blossoming tree is a planted plane tree $(T,\xi)$ such that $T$ is a $k$-blossoming tree and $\xi$ is an inner corner of $T$.
The bijections of Section~\ref{sec:bij-bloss-trees} concern $2$-blossoming trees, which we simply call blossoming trees for the remainder of the section. 

Write $\cT_n$ for the set of planted blossoming trees $(T,\xi)$ with $n$ inner vertices. 
\nomenclature[Tn]{$\mathcal{T}_n$}{Blossoming trees with $n$ inner vertices, planted at an inner corner.}
Fix $(T,\xi) \in \cT_n$, and note that $|E(T)|=|V(T)|-1=3n-1$ so $|\cC(T)|=6n-2=3|\cB(T)|-2$. 
We say $(T,\xi)$ is {\em balanced} if $\xi=(e,e')$ for distinct stems $e,e'$, and for all $c \in \cC(T)$, 
\begin{equation}\label{eq:balanced}
3\big|\bbrcy{\xi, c} \cap\cB\}\big| +1 \ge \big|\bbrcy{\xi, c}\big|
\end{equation}
(recall the definition of $\bbrcy{\xi, c}$ from Section~\ref{sub:PlaneTrees}). 
\nomenclature[Tna]{$\mathcal{T}_n^{\circ}$}{Balanced $2$-blossoming trees with $n$ inner vertices}
For $n \ge 1$ let $\cT_n^{\circ}\subset \cT_n$ be the set of balanced blossoming trees with $n$ inner vertices. 

A {\em valid labelling} of a planted plane tree $\rT=(T,\xi)$
is a labelling $d=(d_e, e \in E(T))$ of the edges of $T$ by elements of $\{-1,0,1\}$ such that for all $v \in V(T)$, writing $k=k_{\rT}(v)$, the sequence $d_{\{v,v1\}},\ldots,d_{\{v,vk\}}$ is non-decreasing. 
\nomenclature[Tnvl]{$\mathcal{T}_n^{\mathrm{vl}}$}{Validly labelled plane trees with $n$ vertices}
Let $\cT^{\mathrm{vl}}_n$ be the set of validly labelled plane trees with $n$ vertices. We emphasize that a validly labelled plane tree is a ``normal'' tree, not a blossoming tree.

Finally, recall that for $n \ge 3$, 
\nomenclature[Tzri]{$\triangle_n^{\circ}$}{Planted triangulations of $\S^2$ with $n$ vertices; see also Section~\ref{sec:intro}.}
$\triangle_n^{\circ}$ is the set of planted triangulations with $n$ inner vertices. 
\medskip
The following diagram summarizes the bijective relations between 
$\cT_n,\cT_n^{\circ}$, and $\triangle_{n+2}^{\circ}$ established in \cite{PouSch06} and recalled in the current section.
\begin{equation}\label{eq:diagram}
\begin{CD}
\cT_n^{\mathrm{vl}}		@<\phi_n;~\text{Prop.\ref{prop:label_bij}}<\mathrm{bij}<
\cT_n				@>\mathrm{projection}>(4n-2)-\mathrm{to}-2>
\cT_n^{\circ} @>\mathrm{bij}>\chi_n;~\text{Prop.\ref{prop:bijGilles}}>\triangle_{n+2}^{\circ}
\end{CD}
\end{equation}

After concluding with bijective arguments, in Section~\ref{sec:GW-trig} we explain how to sample uniformly random triangulations using conditioned Galton-Watson trees. We end the section by describing the inverse of the bijection $\chi_n:\cT_n^{\circ}\to \triangle_{n+2}^{\circ}$, which we use later.

\addtocontents{toc}{\SkipTocEntry}
\subsection{A bijection between triangulations and blossoming trees}\label{sec:clos-tree}

We first describe a bijection of Poulalhon and Schaeffer~\cite{PouSch06} between balanced blossoming trees and simple, planted triangulations of the sphere (see Figure~\ref{fig:clo-tri}; the orientations of the arrows in the figure are explained in Section~\ref{sec:def-ori}). Fix a blossoming tree $T$. Given a stem $\{b,u\}$ with $b \in \cB(T)$, if $bu$ is followed by two inner edges in a clockwise contour exploration of $T$ -- $uv$ and $vw$, say -- then the \emph{local closure} of $\{b,u\}$ 
consists in removing the blossom $b$ 
and its stem, and adding a new edge $\{u,w\}$ (such that $\kr u w =(\{u,w\},\{u,v\})$ and $\kl w u =(\{w,v\},\{w,u\})$). After performing the local closure, $uw$ always has a triangle on its right. The edge $\{u,w\}$ is considered to be an inner edge in subsequent local closures.

The \emph{partial closure} of a blossoming tree is the planar map obtained by performing all possible local closures. 
\nomenclature[Sc]{$s(c)$}{The ``successor'' of corner $c$ in a blossoming tree}
Equivalently, for each corner $c \in \cB$, let $s(c)$ be the inner corner $c'$ minimizing $|\bbrcy{c,c'}|$ 
subject to the condition that
\begin{equation}\label{eq:label_bijection}
3|\bbrcy{c,c'}\cap \cB| < |\bbrcy{c,c'}|,
\end{equation}
if such a corner exists (recall the definition of $\pcy$ from Section~\ref{sub:PlaneTrees}). 
The partial closure operation identifies $\v(c)$ with $\v(s(c))$ whenever $c \in \cB$ and $s(c)$ is defined; it follows from the latter description that the partial closure does not depend on the order in which local closures take place. Say $\v(c)$ is {\em closed} if $s(c)$ is defined, and otherwise say $\v(c)$ is {\em unclosed}. 

It can be checked that the partial closure is a simple map and contains precisely one face $f$ of degree greater than three, and all unclosed blossoms are incident to $f$. 
Furthermore, simple counting arguments show that each inner corner incident to $f$ is adjacent to at least one unclosed blossom, and that there are 
precisely two corners, say $\xi^C$ and
$\xi^D$, that are incident to two unclosed blossoms. Note that $\xi^C$ and $\xi^D$ are both corners of $T$ (i.e., they are not created while performing the partial closure). Let $C=\v(\xi^C)$ and $D=\v(\xi^D)$.

Let $(T,\xi)$ be a balanced blossoming tree such that $\xi=(e,e')$, $e=(\v(\xi),v)$ and $e'=(\v(\xi),v')$. It follows straightforwardly from \eqref{eq:label_bijection} that $v$ and $v'$ are unclosed, or equivalently $\xi$ is equal to $\xi_C$ or $\xi_D$.

 We now suppose $\xi \in \{\xi^C,\xi^D\}$. 
Let $S_{CD}$ (resp.\ $S_{DC}$) be the set of non-blossom vertices $v$ of the distinguished face $f$ of the partial closure 
such that in the planted tree $(T,C)$ (resp.\ $(T,D)$) we have $v \pc D$ (resp.\ $v \pc C$).
In other words, vertices of $S_{CD}$ lie after $C$ and before $D$ in a clockwise tour of $f$, and likewise for $S_{DC}$. 

To finish the construction, remove the remaining blossoms and their stems. 
Add two additional vertices $A$ and $B$ within $f$, 
then add an edge between $A$ (resp.~$B$) and each of the vertices of $S_{CD}$ (resp.\ of $S_{DC}$). 
In the resulting map, define a corner $c$ by $c=(\{C,B\},\{C,A\})$ if $\v(\xi)=C$ or $c=(\{D,A\},\{D,B\})$ if $\v(\xi)=D$. 
Finally, add an edge between 
$A$ and $B$ in such a way that, after its addition, $A,B$, and $\v(\xi)$ lie on the same face $f$. 
The result is a planar map, rooted at $\xi$, called the {\em closure} of $T$. 
For later use, define a function $s':V(T) \to V(T)$ as follows. 
\nomenclature[Sprime]{$s'$}{A modification of the ``successor'' function $s$ defined for vertices instead of corners.}
\label{sstardef}
First, set $s'(v)=v$ for $v \in V(T)\setminus \cB$. For $v \in \cB$, let $u$ be the unique neighbour of $v$ and let $k$ be the unique corner incident to $v$. If $s(k)$ is defined then let $s'(v)=\v(s(k))$; otherwise, if $u \in S_{CD}$ let $s'(v)=A$ and if $u \in S_{DC}$ let $s'(v)=B$. 

\begin{figure}[ht]
\hspace{-0.5cm}
 \subfigure[A balanced blossoming tree]{\includegraphics[width=.3\linewidth,page=1]{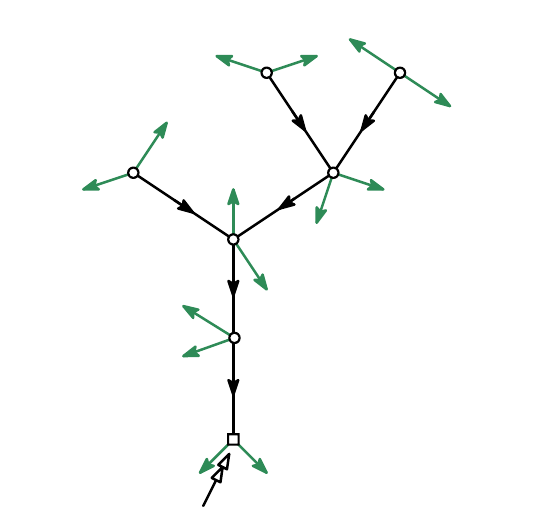}}\qquad
  \subfigure[all the local closures have been
performed]{\includegraphics[width=.3\linewidth,page=3]{Pictures/cloture-triangulation.pdf}}
\qquad
\subfigure[\label{fig:clo-tri-ori}The resulting rooted simple triangulation, endowed with its unique
minimal 3-orientation]{\includegraphics[width=.3\linewidth,page=4]{Pictures/cloture-triangulation.pdf}}
\caption{The closure of a balanced tree into a simple triangulation.} 
\label{fig:clo-tri}
\end{figure}

Write $\chi:\bigcup_{n \ge 1} \cT_n^{\circ} \to \bigcup_{n \ge 1} \triangle_{n+2}^{\circ}$ 
for the function sending a balanced blossoming tree to its closure, and for $n \ge 1$ let 
\nomenclature[Czhin]{$\chi_n$}{Closure bijection from $\mathcal{T}_n^{\circ}$ to $\triangle_{n+2}^{\circ}$; same paragraph for $\chi$}
$\chi_n:\cT_n^{\circ} \to \triangle^{\circ}_{n+2}$ be the restriction of $\chi$ to $\cT_n^{\circ}$. 
\begin{prop}[\cite{PouSch06}]\label{prop:bijGilles}
For all $n \ge 1$, $\chi_n$ is a bijection between $\cT_n^{\circ}$ and $\triangle^{\circ}_{n+2}$. 
\end{prop}
Note that if $(T,\xi)$ is a blossoming tree and $\chi(T,\xi)=(G,c)$ then it is natural to identify the inner vertices and inner edges of $T$ with subsets of $V(G)$ and $E(G)$, respectively. More formally, we may choose representatives from the isomorphism equivalence classes of the tree and its closure so that $V(T)\setminus \cB(T)=V(G)\setminus\{A,B\}$ and $\{\{u,v\} \in E(T):u,v \not\in\cB\} \subset E(G)$. We will adopt this perspective in the remainder of the paper.

\addtocontents{toc}{\SkipTocEntry}
\subsection{Bijection with labels} \label{sec:bij_with_labels}
We now present an alternative description of the bijection from Proposition~\ref{prop:bijGilles}, 
based on (\ref{eq:label_bijection}). 
Given a blossoming tree $(T,\xi)$, write $\rT=(T,\xi)$ and 
define $\lambda=\lambda_{\rT}:\cC(T) \to \Z$ as follows. 
\nomenclature[Lambdat]{$\lambda_{\mathrm{T}}$}{Corner labelling of planted blossoming tree $\mathrm{T}$}
Recall the definition of the contour ordering $(\xi(i),0 \le i \le 2|V(T)|-2)$ from Section~\ref{sub:PlaneTrees}, and in particular that $\xi(0)=\xi$. 
Let $\lambda(\xi(0))=2$ and, for $0 \le i < 2|V(T)|-3$, set 
\[
\lambda(\xi(i+1)) = 
\begin{cases}
\lambda(\xi(i)) -1  & \text{if } \xi(i)\not\in \cB(T), \xi(i+1) \not\in \cB(T) , \\
\lambda(\xi(i))  & \text{if } \xi(i) \not\in \cB(T), \xi(i+1)\in \cB(T) , \\
\lambda(\xi(i)) +1  & \text{if } \xi(i)\in \cB(T), \xi(i+1) \not\in \cB(T) , 
\end{cases}
\]
This labelling is depicted in Figure~\ref{subfig:LabelTree}. 
Informally, we perform a clockwise contour exploration of the tree and label 
the corners as we go. When leaving an inner vertex and arriving at an inner vertex, decrease the label by one; when leaving an inner vertex and arriving at a blossom, leave the label unchanged; when leaving a blossom and arriving at an inner vertex, increase the label by one. 

It is not hard to see that $\rT=(T,\xi)$ is balanced if and only if $\xi$ is incident to two stems and $\lambda(c) \ge 2$ for all $c \in \cC(T)$ (see Figure~\ref{subfig:LabelTree}). Assume $(T,\xi)$ is balanced and write $\xi'$ for the unique corner in $\cC(T)\setminus \{\xi\}$ for which $(T,\xi')$ is also balanced. Given a corner $c \in \cC(T)$ with $\v(c) \in \cB(T)$, recall the definition of $s(c)$ from (\ref{eq:label_bijection}). 
A counting argument shows that when $s(c)$ is defined, 
it is equal to the first corner $c'$ following $c$ in clockwise order for which $\lambda(c') < \lambda(c)$ (and in fact $\lambda(s(c)) = \lambda(c)-1$).
\label{eq:scdefine}%
Furthermore, $s(c)$ is defined if and only if either $\lambda(c) > 2$ and $c \pc \xi'$, or $\lambda(c)>3$ and $\xi'\pc c$. 

Next, add two vertices, say $A$ and $B$, within the unique face of the partial closure with degree greater than three. For each $c \in \cC(T)$ with $\v(c) \in \cB(T)$ and 
$s(c)$ undefined, identify $\v(c)$ with $A$ if $\lambda(c)=2$, and with $B$ if $\lambda(c)=3$. At this point, the unique face of degree greater than three is incident to $\xi,A,\xi'$ and $B$ in cyclic order. Finally, add a single edge between $A$ and $B$. The following fact, whose straightforward proof is omitted, states that the resulting planar map is $\chi(\rT)$. 

\begin{fact}
The triangulation obtained from a balanced blossoming tree by iterating local
closures and the one obtained by the label procedure coincide. \qed
\end{fact} 
\begin{figure}[ht]
\hspace{-0.5cm}
 \subfigure[\label{subfig:LabelTree}The corner labelling of a balanced blossoming
tree]{\includegraphics[width=.33\linewidth,page=1]{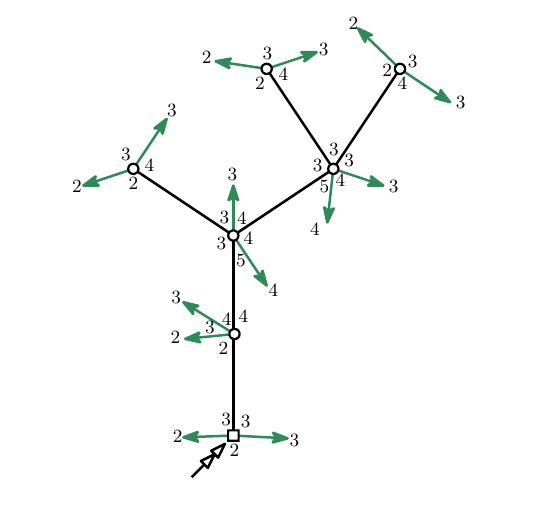}}
  \subfigure[\label{subfig:LabelClos}The labelled partial closure]{\includegraphics[width=.33\linewidth,page=3]{Pictures/triangulation-labeledb.pdf}}
\subfigure[\label{subfig:LabelTrig}The resulting corner-labelled simple triangulation]{\includegraphics[width=.33\linewidth,page=4]{Pictures/triangulation-labeledb.pdf}}
\caption{Closing a balanced tree via the corner labelling.} 
\label{fig:labelconstruct}
\end{figure}
The closure contains corners not present in the blossoming tree, and the new corners are labelled as follows. For any bud corner $c$ for which $s(c)$ isdefined, closing $\v(c)$ may be viewed as splitting a single corner in two, and the two new corners inherit the label of the corner that was split. An example is shown in Figure~\ref{subfig:LabelClos}; the dashed arcs denote corners that are ``split'' by the partial closure operation. Let $f$ be the face of $\chi(T,\xi)$ incident to $\xi$. 
Give the corner of $A$ (resp.~$B$) incident to $f$ label $0$ (resp.~$1$), and give all other corners incident to $A$ (resp.~$B$) label $1$ (resp.~$2$). 
\nomenclature[Lambdastar]{$\lambda^*$}{Push-forward of $\lambda_{\mathrm{T}}$ by $\chi$.}
We write $\lambda^*=\lambda^*_{(T,\xi)}$ for this corner labelling of $\chi(T,\xi)$, and note that $\lambda^*:\cC(\chi(T,\xi)) \to \Z^{\ge 0}$ since we have assumed $(T,\xi)$ is balanced.  
An example of the resulting corner-labelled triangulation is depicted in Figure~\ref{subfig:LabelTrig}.

\addtocontents{toc}{\SkipTocEntry}
\subsection{From labels to displacement vectors}\label{sec:labels-displacements}
We next explain the connection between blossoming trees and validly labelled plane trees. Fix $n \ge 1$, let $(T,\xi) \in \cT_n$ and let $\lambda$ be the labeling of corners as defined 
in Section~\ref{sec:bij_with_labels}. We define a function $Y=Y_{(T,\xi)}:V(T) \to \Z$ by setting 
\nomenclature[Yt]{$Y_{(T,\hat{\xi})}$}{Vertex labelling induced by $\lambda$, $Y_{(T,\hat{\xi})}(v)=\min\{\lambda(c): c \in \cC(T), \v(c)=v\}$}
$Y(v) = \min\{\lambda(c): c \in \cC(T), \v(c)=v\}$ for all $v \in V(T)$. 
Next, for each inner edge $e \in E(T)$, writing $e=\{v,p(v)\}$, with $v\in V(T)\setminus \{\v( \xi)\}$, set $D_e=D_e(T,\xi)=Y(v)-Y(p(v))$. The following easy fact, whose proof is omitted, allows 
us to recover the locations of stems from the edge labels. 

\begin{fact}\label{fact:de_bd}
For all $e=\{v,p(v)\} \in E(T)$, $D_e+1=|\{e' \pl e:e'~\mbox{a stem incident to }p(v) \}|$. \qed
\end{fact}
Now fix $v \in V(T)$, let $k=k_{(T,\hat{\xi})}(v)$, and for $1 \le i \le k$ let 
$e_i=\{v,vi\}$. 
It follows from the above fact that for $1 \le i \le k$ the number of stems $e$ incident to $v$ with $e \pl e_i$ is $D_{e_i}+1$. In particular $(D_{e_i},1 \le i \le k)$ is a non-decreasing sequence of elements of $\{-1,0,1\}$; this is what allows us to connect blossoming trees with validly labelled trees. 

\nomenclature[Phin]{$\phi_n$}{Bijection from $\mathcal{T}_n$ to $\mathcal{T}_n^{\mathrm{vl}}$}
For $n \ge 1$ define a map $\phi_n:\cT_n \to \cT^{\mathrm{vl}}_n$ as follows. 
Given $(T,\xi) \in \cT_n$, write ${\xi}=(e_-,e_+)$. Let $e$ be the last inner edge incident to $\v({\xi})$ preceding $e_-$ in clockwise order (with $e=e_-$ if $e_-$ is an inner edge), and let $e'$ be the first inner edge incident to $\v({\xi})$ following $e_+$ in clockwise order (with $e=e_+$ if $e_+$ is an inner edge). Write $\xi'=(e,e')$, let $T'$ be the subtree of $T$ induced by the inner vertices, let $D=(D_e(T,{\xi}), e \in E(T))$, and let $\phi_n(T,{\xi})=(T',\xi',D)$. The following proposition is an immediate consequence of Fact~\ref{fact:de_bd}. 

\begin{prop}\label{prop:label_bij}
The map $\phi_n:\cT_n \to \cT^{\mathrm{vl}}_n$ is a bijection. 
\qed
\end{prop}

The above bijection and definitions are illustrated in Figure~\ref{fig:vector-equiv}. In the next section, we explain how the above functions can be used to sample random simple triangulations with the aid of conditioned Galton--Watson trees. 
\begin{figure}[ht]
\hspace{-0.5cm}
 \subfigure[The corner labelling of a corner-rooted blossoming
tree $(T,\xi)$. ]{\includegraphics[width=.25\linewidth,page=1]{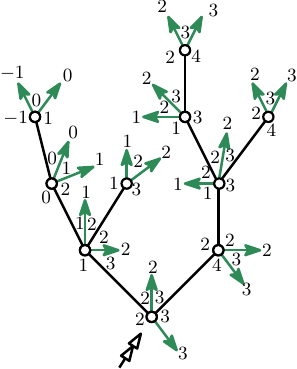}}\qquad
  \subfigure[The tree $T'$ with values $D^e$ marked on edges.]{\includegraphics[width=.25\linewidth,page=2]{Pictures/TreesDisp.pdf}}
\qquad
\subfigure[The tree $T'$ with values $Y(v)$ marked on vertices.
]{\includegraphics[width=.25\linewidth,page=3]{Pictures/TreesDisp.pdf}}
\caption{The equivalence between blossoming trees and validly vector-labelled plane trees. The root corner is indicated via a double arrow.} 
\label{fig:vector-equiv}
\end{figure}

\addtocontents{toc}{\SkipTocEntry}
\subsection{Corner-rooted triangulations via conditioned Galton--Watson trees}\label{sec:GW-trig}
Let $(T_n,\xi_n)$ be uniformly distributed on $\cT_n$. 
We are now able to describe the law of $(T_n,\xi_n)$ as a modification of the law of a critical Galton--Watson tree conditioned to have a given size. (Galton--Watson trees are naturally viewed as planted plane trees; see e.g.\ \cite{le2005random}.)
Let $G\eqdist \mathrm{Geometric}(3/4)$, and let $B$ have law given by 
\begin{equation}\label{eq:offspring}
\p{B=c} = \frac{{c+2 \choose 2} \p{G=c}}{\e{{G+2 \choose 2}}}\,, \mbox{ for } c \in \N.
\end{equation}
\begin{fact}
The distribution $B$ is critical, i.e. $\e{B}=1$. \qed
\end{fact}
This fact follows from simple computations involving the 3
first moments of a geometric law; its proof is omitted. 
\begin{prop}\label{prop:uniformtreelaw}
Let $(T',\xi')$ be a Galton--Watson tree with offspring distribution $B$ conditioned to have $n$ vertices. For each vertex $v$ of $T'$, add two stems incident to $v$, uniformly at random from among the ${k_{T'}(v)+2 \choose 2}$ possibilities. The resulting planted plane tree $(T,{\xi})$ is uniformly distributed over $\cT_n$. 
\end{prop}

\begin{proof}
Fix $\mathrm{t} \in \cT_n$ and let $\mathrm{t}'=\mathrm{t}\langle V(T)\setminus \mathcal{B}(T)\rangle$ be the tree $\mathrm{t}$ with its blossoms removed. List the vertices of $\mathrm{t}'$ in lexicographic order as $v_1,\ldots,v_n$ and recall that $k_{\mathrm{t}'}(v_i)$ is the number of children of $v_i$ in $\mathrm{t'}$. 

Then $(T,{\xi})$ is equal to $\mathrm{t}$ if and only if $(T',\xi')=\mathrm{t'}$ and for each $v \in V(T')$, the stems are inserted at the right place. Hence: 
\begin{align*}
\p{(T,{\xi})=\mathrm{t}} & \propto \prod_{i=1}^n\frac{1}{ {k_{\mathrm{t}'}(v_i) +2\choose 2}}\p{B = k_{\mathrm{t}'}(v_i)} \\
& =  \prod_{i=1}^n\frac{1}{ {k_{\mathrm{t}'}(v_i) +2\choose 2}}{k_{\mathrm{t}'}(v_i) +2\choose 2}\p{G = k_{\mathrm{t}'}(v_i)} \\
& = \frac{3^{n-1}}{4^{2n-1}}. 
\end{align*}
The last equality holds since $G$ is geometric and $\sum_{i=1}^n k_{\mathrm{t}'}(v_i) = n-1$. 
Since the last term 
does not depend on the shape of
$\mathrm{t}$, all elements of $\cT_n$ appear with
the same probability. 
\end{proof}

\begin{cor}\label{cor:uniftreelaw}
Let $(T,\hat{\xi})$ be uniformly random in $\cT_n$ and let $\xi^1,\xi^2 \in \cC(T)$ be such that $(T,\xi^i)$ is balanced for $i\in\{1,2\}$. Conditionally given $(T,\hat{\xi})$ choose $\xi \in \{\xi^1,\xi^2\}$ uniformly at random. Then $(G,c)=\chi(T,\xi)$ is 
uniformly distributed in~$\triangle_{n+2}^{\circ}$.  
\end{cor}

\begin{proof}
 Let $\rt=(t,c)$ be a balanced blossoming tree of size $n$. Consider the set of triples $\rt^\bullet = \{(t,c,\hat c)\,:\,\hat c \text{ an inner corner of }t\}$. Then, we have 
\[
\p{(T,\xi)=\rt} = \p{(T,\xi,\hat \xi) \in \rt^\bullet}=\frac{1}{2}\sum_{(t,c,\hat c)\in \rt^\bullet}\p{(T,\hat \xi)=(t,\hat c)} = \frac{2n-1}{|\cT_n|},
\]
where the last equality comes from the fact that $\rt$ has $4n-2$ inner corners (hence $|t^\bullet|=4n-2$) and that $(T,\hat \xi)$ is uniformly random in $\cT_n$. 
It follows that $(T,\xi)$ is uniformly random in $\cT_n^\circ$, which concludes the proof since $\chi$ is a bijection between $\cT_n^\circ$ and $\triangle_{n+2}^{\circ}$.
\end{proof}

Proposition~\ref{prop:label_bij} now allows us to describe the distribution of a uniformly random element $(T',\xi',D)$ of $\cT_n^{\mathrm{vl}}$. 
For each $k \ge 1$, let $\nu_k$ be the uniform law over non-decreasing vectors $(d_1,\ldots,d_k) \in \{-1,0,1\}^k$. 
\begin{cor}\label{cor:vldist}
Let $(T',\xi')$ be a Galton--Watson tree with offspring distribution $B$ conditioned to have $n$ vertices. Conditionally given $(T',\xi')$, independently for each $v \in V(T')$ let $(D_{\{v,vj)\}},1 \le j \le k(v))$ be a random vector with law $\nu_{k(v)}$. Finally, let $D=(D_e,e \in E(T'))$. Then $(T',\xi',D)$ is uniformly distributed in $\cT_n^{\mathrm{vl}}$. 
\end{cor}
\begin{proof}
By Proposition~\ref{prop:uniformtreelaw}, the tree $(T',\xi')$ has the same law as the subtree obtained from a uniformly random element of $\cT_n$ by removing all stems. The corollary then follows by Proposition~\ref{prop:label_bij}.
\end{proof}

For later use, we note the following fact. Recall the definition of the labelling function $X_{\rT}$ for $\rT$ a spatial plane tree, from Section~\ref{sec:contourproc}. 
\begin{fact}\label{fact:rerooting}
Fix two inner corners $\xi_1,\xi_2$ of $\cC(T)$, and let $\rT_1=(T',\xi_1',D_1)=\phi_n(T,\xi_1)$ and $\rT_2=(T',\xi_2',D_2)=\phi_n(T,\xi_2)$. Then for all $v \in V(T')$, $X_{\rT_1}(v)=Y_{(T,\xi_1)}(v)-2$, and 
$\left|\left(Y_{(T,\xi_1)}(v)-Y_{(T,\xi_2)}(v)\right)-X_{(T,\xi_1)}(\v(\xi_2))\right| \le 3$. 
\end{fact}
In other words the labellings $X_{\rT}$ and $Y_{\rT}$ are related by an additive constant of $2$, and rerooting shifts all labels according to the label of the new root under the old labelling, up to an additive error of $3$. This is a direct consequence of Fact~\ref{fact:de_bd} and the definitions of $X_{\rT}$ and $Y_{\rT}$; its proof is omitted. 

We conclude Section~\ref{sec:bij-bloss-trees} by explaining the inverse of the bijection $\chi_n$. The description of the inverse relies the properties of so called 3-orientations for simple triangulations. We make use of such orientations in Section~\ref{sec:label_dist_determinist} when studying the relation between vertex labels and geodesics. 

\addtocontents{toc}{\SkipTocEntry}
\subsection{Orientations and the opening operation}\label{sec:def-ori}
In a planted map endowed with an orientation, a directed cycle is said to
be clockwise if the root corner is situated on its left and 
counterclockwise otherwise. An orientation is called \emph{minimal} if it has no counterclockwise cycles. 
Let $(G,\xi)$ be a planted planar triangulation, and recall from Section~\ref{sec:overview} that that $(G,\xi)$ admits a unique minimal $3$-orientation. We next describe how to obtain this $3$-orientation via the bijection described in Proposition~\ref{prop:bijGilles}. 
   
Given a balanced 2-blossoming tree
$\rT=(T,\xi)$, orient all stems towards their incident blossom, 
and orient all other edges towards $\v(\xi)$. In the triangulation $\chi(\rT)$, all edges except $\{A,B\}$ inherit an orientation from $T$; orient $\{A,B\}$ from $B$ to $A$. Then all inner vertices of $\rT$ not incident to $\xi$ have outdegree $3$ in $T$ and the closure operation does not change this outdegree. It follows easily that the resulting orientation of $\chi(\rT)$ is a $3$-orientation. Furthermore,  the ``clockwise direction'' of the local closures implies that closure never creates counterclockwise cycles, so the $3$-orientation is minimal. 

Given a planted planar triangulation $\rG=(G,\xi)$, the balanced blossoming tree $\chi^{-1}(G,\xi)$ can be recovered as follows. 
Let $\overrightarrow{E}$ be the unique minimal $3$-orientation of $E(G)$. 
List the vertices of the face incident to $\xi$ in clockwise order as $(\v(\xi),A,B)$. 
Remove the edge $\{A,B\}$, and perform a clockwise contour exploration of $\rG$ starting from $\xi$. Each time we see an edge $uv$ for the first time, if it is oriented in the opposite direction from the contour process then keep it; otherwise replace it by a stem $\{u,b_{uv}\}$. This procedure is depicted in Figure~\ref{fig:ope-tri}. 
\begin{figure}[ht]
\hspace{-0.5cm}
 \subfigure[A simple triangulation endowed with its 3-orientation.]{\includegraphics[width=.28\linewidth,page=1]{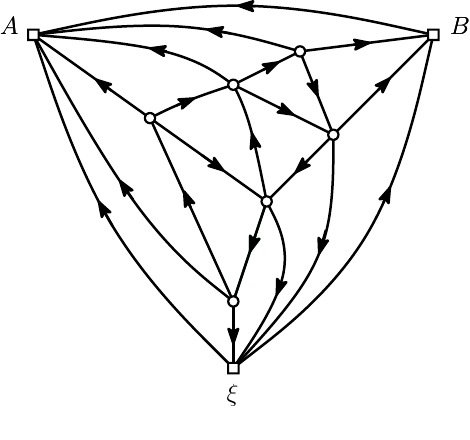}}\qquad
  \subfigure[Here the root edges have been removed and the contour process has
started.]{\includegraphics[width=.28\linewidth,page=2]{Pictures/opening-triangulation.pdf}}
\qquad
\subfigure[\label{fig:ope-tri-tree} The blossoming tree obtained after the
completion of the algorithm.]{\includegraphics[width=.28\linewidth,page=3]{Pictures/opening-triangulation.pdf}}
\caption{The opening of a simple triangulation into a 2-blossoming tree.} 
\label{fig:ope-tri}
\end{figure}

\section{Convergence to the Brownian snake}\label{sec:snake}
Fix a probability distribution $\mu$ on $\N$, and a sequence $\nu=(\nu_k,k \ge 1)$, where for each $k \ge 1$, $\nu_k$ is a probability distribution on $\R^k$. It is convenient to introduce the  notation $\nu_k=(\nu_k^i,1 \le i \le k)$, where $\nu_k^i$ are the marginals of $\nu_k$.

For $n \in \N$, we write 
$\LGW(\mu,\nu,n)$ for the law on spatial plane trees $\rT=(T,\xi,D)$ such that: 
\begin{itemize}
\item The pair $(T,\xi)$ has the law of the genealogical tree of a Galton-Watson process with offspring distribution $\mu$, conditioned to have total progeny $n$.\footnote{To avoid trivial technicalities, we assume $\mu$ is such that the support of $\mu$ has greatest common divisor $1$, so that such conditioning is well-defined for all $n$ sufficiently large.}
\item Conditionally on $(T,\xi)$, $D:E(T) \to \R$ has the following law. 
Independently for each $u\in V(T)$, if $u$ has $k$ children then $\big(D(\{u,u1\}),\ldots,D(\{u,uk\})\big)$
is distributed according to $\nu_k$. 
\end{itemize}
Here is the connection with random simple triangulations. If $(T_n,\xi_n,D_n)$ is uniformly distributed in $\cT_n^{\mathrm{vl}}$, then Corollary~\ref{cor:vldist} states that the law of $(T_n,\xi_n,D_n)$ is $\LGW(\mu,\nu,n)$, where $\mu$ is the law defined in (\ref{eq:offspring}) and for $k \ge 1$, $\nu_k$ is the uniform law on non-decreasing vectors in $\{-1,0,1\}^k$. 

Recall the definition of the pair $(\be,Z)$ from Section~\ref{sec:limitobject}, and the definitions of the functions $C_{\rT}$, $X_{\rT}$ and $Z_{\rT}$ from Section~\ref{sec:contourproc}. 
We establish the following convergence. 
\begin{prop}\label{prop:snake}
For $n \ge 1$ let $\rT_n=(T_n,\xi_n,D_n)$ be uniformly random in $\cT_n^{\mathrm{vl}}$. 
Then as $n \to \infty$, 
\begin{equation}\label{eq:cvsnake}
\left( (3n)^{-1/2}C_{\rT_n}(t), (4n/3)^{-1/4}Z_{\rT_n}(t)\right)_{0\le
t \le 1}\,
\convdist \, ({\bf e}(t), Z(t))_{0\le t \le 1},
\end{equation}
for the topology of uniform convergence on $C([0,1],\R)^2$.
\end{prop}
Before proving this proposition, we place it in the context of the existing literature on invariance principles for spatial branching processes. 
Fix $\mu$ and $\nu$ and let $(\rT_n,n \in \N)$ be such that $\rT_n=(T_n,\xi_n,D_n)$ has law $\LGW(\mu,\nu,n)$ for $n \in \N$. In what follows, given a measure $\eta$ on $\R$ and $p > 0$ write 
$|\eta|_p = (\int_{\R} |x|^p\mathrm{d}\eta)^{1/p}$. 
Aldous (\cite{AldCRT2}, Theorem 2) showed that if $|\mu|_1=1$ and $\sigma_{\mu}^2 := |\mu|_2^2-|\mu|_1^2 \in (0,\infty)$, then 
\begin{equation}\label{eq:aldous}
\left(\frac{\sigma_\mu}{2}\frac{C_{\rT_n}(t)}{n^{1/2}}\right)_{0 \le t \le 1}\convdist \be
\end{equation}
as $n \to \infty$, for the topology of uniform convergence on $C([0,1],\R)$. 
Now additionally suppose that for each $k$, the marginals $\{\nu_k^i:1 \le i \le k\}$ of $\nu_k$ are  identically distributed, that 
$|\nu_k^1|_1 < \infty$, that $\nu_k$ is {\em centred} (i.e., $\int_{\R} x \mathrm{d}\nu_k^1(x)=0$ 
for every $1\leq i\leq k$), and that 
\[
\sup_k \nu_k^1([y,\infty)] = o(y^{-4}) \quad \text{for every }k\geq 1. 
\]
Under these conditions, writing $\sigma_{\nu}=\sigma_{\nu_1^1}$, Janson and Marckert (\cite{JanMar}, Theorem 2) prove that 
\begin{equation}\label{eq:cvsnake_jm} 
\left(\frac{\sigma_{\mu}}{2}\frac{C_{\rT_n}(t)}{n^{1/2}}, \frac{(\sigma_\mu/2)^{1/2}}{\sigma_\nu}\frac{Z_{\rT_n}(t)}{n^{1/4}}\right)_{0\le
t \le 1}\,
\overset{(d)}{\underset{n\rightarrow \infty}\tend}\, ({\bf e}(t), Z(t))_{0\le t \le 1}\, .
\end{equation}
in the same topology as in Proposition~\ref{prop:snake}
(In fact Theorem~2 of~\cite{JanMar} is stated with the additional assumption that $\nu_k$ is a product measure for all $k$. However, it is not difficult to see, and was explicitly noted in \cite{JanMar}, that straightforward modifications of the proof allow this additional assumption to be removed.) Under the same assumptions, the convergence in (\ref{eq:cvsnake_jm}) can also be obtained as a special case of \cite[Theorem 8]{MarckertMiermont07}. In the latter article, the marginals of $\nu_k$ are not required to be identically distributed but they are assumed to be \emph{locally centred} meaning that for all $1\leq i \leq k$, $\int_{\R}x\mathrm{d}\nu_k^i(x)=0$. 
In our setting, the law of the spatial plane tree is given by Corollary~\ref{cor:vldist}. 
In this case the entries are clearly not identically distributed, and neither are they locally centred: observe for instance that $\int_{\R}x \mathrm{d}\nu_2^1(x)=-1/3$. 

In~\cite{MarckertLineage}, the ``locally centred'' assumption is replaced by a \emph{global} centering assumption, namely that
\[
\sum_{k\geq 0} \mu(\{k\}) \sum_{i=1}^k \int_{\R} x \mathrm{d}\nu_k^i(x)=0,
\]
which is satisfied by our model. However, \cite{MarckertLineage} requires that $\mu$ has bounded support, which is not the case in Corollary~\ref{cor:vldist}. 

We expect that the technique we use to prove Proposition~\ref{prop:snake} can be used to extend the results of~\cite{MarckertLineage} to a broad range of laws $\LGW(\mu,\nu,n)$ for which $\mu$ does not have compact support, under the slightly stronger centering assumption that $\sum_{i=1}^k \int_{\R} x \mathrm{d}\nu_k^i(x)=0$ for every $k$. However, for the sake of concision we have chosen to focus on the random spatial plane trees that arise from random simple triangulations (the treatment for random simple quadrangulations differs only microscopically and is omitted).
~\\

For the remainder of Section~\ref{sec:snake}, let $\mu$ be as defined in (\ref{eq:offspring}), and for $k \ge 1$ let $\nu_k$ be the uniform law on non-decreasing vectors in $\{-1,0,1\}^k$. For $n \ge 1$ let $\rT_n=(T_n,\xi_n,D_n)$ be uniformly random in $\cT_n^{\mathrm{vl}}$ as in Proposition~\ref{prop:snake}; by the comments preceding that proposition, $\rT_n$ has law $\LGW(\mu,\nu,n)$. 
To prove Proposition~\ref{prop:snake}, we establish the following facts. 
\begin{lem}[Random finite-dimensional distributions]\label{lem:fdds}
Let $(U_i,i \ge 1)$ be independent Uniform$[0,1]$ random variables, 
independent of the trees $(\rT_n,n \ge 1)$, and for $j \ge 1$ let 
$(U_i^{\uparrow},1 \le i \le j)$ be the increasing ordering of $U_1,\ldots,U_j$. Then for all $j \ge 1$, 
\begin{equation}\label{eq:fdds}
\left( (3n)^{-1/2}C_{\rT_n}(U_i^{\uparrow}), (4n/3)^{-1/4}Z_{\rT_n}(U_i^{\uparrow})\right)_{1 \le i \le j}\,
\convdist \, ({\bf e}(U_i^{\uparrow}), Z(U_i^{\uparrow}))_{1 \le i \le j}\, ,
\end{equation}
as $n \to \infty$. 
\end{lem}
\begin{lem}[Tightness]\label{lem:tightness}
The family of laws of the processes $((4n/3)^{-1/4}Z_{\rT_n},n \ge 1)$ is tight for the space of probability measures on $C([0,1],\R)$. 
\end{lem} 
\begin{proof}[Proof of Proposition~\ref{prop:snake}]
It is immediate from (\ref{eq:aldous}) and Lemma~\ref{lem:tightness} that the collection of laws of the processes $(((3n)^{-1/2} C_{\rT_n},(4n/3)^{-1/4}Z_{\rT_n}),n \ge 1)$ forms a tight family in the space of probability measures on $C([0,1],\R)^2$. It therefore remains to 
establish convergence of (non-random) finite-dimensional distributions. In other words, we must show that for all $m \ge 1$, $0 \le t_1 < t_2 < \ldots < t_m \le 1$ and  $c_1,\ldots,c_m > 0$, $z_1,\ldots,z_m \in \R$, 
\[
\begin{split}
\p{\bigcap_{i=1}^m \{ (3n)^{-1/2}C_{\rT_n}(t_i) \le c_i\} \cap \{ (4n/3)^{-1/4}Z_{\rT_n}(t_i) \le z_i\}}\\
\to
\p{\bigcap_{i=1}^m \{{\bf e}(t_i) \le c_i\} \cap \{Z(t_i) \le z_i\}}\, .
\end{split}
\]
For the remainder of the proof, we fix $m$ and $(t_i,c_i,z_i,1 \le i \le m)$ as above. 

Tightness implies (see \cite{billingsley}, Theorem 8.2) that for all $\delta > 0$ there exists $\alpha=\alpha(\delta)$ such that 
\begin{equation} \label{eq:fin_tight}
\limsup_{n \to \infty} \p{\sup_{x,y \in [0,1],|x-y| \le \alpha} \left( \frac{|C_{\rT_n}(x)-C_{\rT_n}(y)|}{(3n)^{1/2}}+\frac{|Z_{\rT_n}(x)-Z_{\rT_n}(y)|}{(4n/3)^{1/4}}\right) > \delta} < \delta\, . 
\end{equation} 
Since $\be$ and $Z$ are almost surely uniformly continuous, by decreasing $\alpha(\delta)$ if necessary we may additionally ensure that 
\begin{equation} \label{eq:lim_tight}
\p{\sup_{x,y \in [0,1],|x-y| \le \alpha} \left(|{\bf e}(x)-{\bf e}(y)| + |Z(x)-Z(y)|\right) > \delta} < \delta. 
\end{equation}

Since (\ref{eq:fin_tight}) holds, to prove Proposition~\ref{prop:snake} it remains to establish convergence of (non-random) finite-dimensional distributions. In other words, we must show that for all $m \ge 1$, $0 \le t_1 < t_2 < \ldots < t_m \le 1$ and $c_1,\ldots,c_m,z_1,\ldots,z_m \in \R$ such that $c_i$ and $z_i$ are continuity points of the distributions of ${\bf e}(t_i)$ and $Z(t_i)$, respectively, 
\[
\begin{split}
\p{\bigcap_{i=1}^m \{ (3n)^{-1/2}C_{\rT_n}(t_i) \le c_i\} \cap \{ (4n/3)^{-1/4}Z_{\rT_n}(t_i) \le z_i\}}\\
\to
\p{\bigcap_{i=1}^m \{{\bf e}(t_i) \le c_i\} \cap \{Z(t_i) \le z_i\}}\, .
\end{split}
\]
For the remainder of the proof, we fix $m$ and $(t_i,c_i,z_i,1 \le i \le m)$ as above. 
\smallskip

Given $\delta > 0$, let $\alpha=\alpha(\delta)$ be as above, and let $j=j(\delta)> 2/\alpha $ be large enough that 
\[
\p{\max_{1 \le i \le j} \left| U_j^{(i)} - \frac{i}{j} \right| \ge \frac{\alpha}{2}} < \delta. 
\]
Since $j > 2/\alpha$, we may choose integers $k_1,\ldots,k_m$ so that for $1 \le i \le m$, $|k_i/j-t_i| < \alpha/2$. It follows that 
\begin{equation}\label{eq:unif_llm}
\p{\max_{1 \le i \le m} |U_j^{(k_i)}-t_i|  \ge \alpha} < \delta\, .
\end{equation}
Write $A,B$ and $C$ for the events whose probabilities are bounded in (\ref{eq:fin_tight}), (\ref{eq:lim_tight}) and (\ref{eq:unif_llm}), respectively, and let $E=(A \cup B \cup C)^c$. Note that $\p{E} >1-3\delta$. 
Furthermore, when $E$ occurs, $|{\bf e}(t_i) - {\bf e}(U_j^{(k_i)})|+  |Z(t_i) - Z(U_j^{(k_i)})|< \delta$, so 
\[
E \cap \bigcap_{i=1}^m \{{\bf e}(t_i) \le c_i-\delta\} \cap \{Z(t_i) \le z_i-\delta\} \subset E \cap \bigcap_{i=1}^m \{{\bf e}(U_j^{(k_i)}) \le c_i\} \cap \{Z(U_j^{(k_i)}) \le z_i\}\, .
\]
We thus have 
\begin{align*}
& \p{\bigcap_{i=1}^m \{{\bf e}(t_i) \le c_i-\delta\} \cap \{Z(t_i) \le z_i-\delta\}}-3\delta\\
< & \p{E \cap \bigcap_{i=1}^m \{{\bf e}(t_i) \le c_i-\delta\} \cap \{Z(t_i) \le z_i-\delta\}} \\
\le & \p{E \cap \bigcap_{i=1}^m \{{\bf e}(U_j^{(k_i)}) \le c_i\} \cap \{Z(U_j^{(k_i)}) \le z_i\}} \\
\le & \p{\bigcap_{i=1}^m \{{\bf e}(U_j^{(k_i)}) \le c_i\} \cap \{Z(U_j^{(k_i)}) \le z_i\}}\, .
\end{align*}
A similar argument shows that 
\begin{align*}
& \p{\bigcap_{i=1}^m \{ (3n)^{-1/2}C_{\rT_n}(t_i) \le c_i+\delta\} \cap \{ (4n/3)^{-1/4}Z_{\rT_n}(t_i) \le z_i+\delta\}}\\
\ge &\p{E \cap \bigcap_{i=1}^m \{ (3n)^{-1/2}C_{\rT_n}(t_i) \le c_i+\delta\} \cap \{ (4n/3)^{-1/4}Z_{\rT_n}(t_i) \le z_i+\delta\}}\\
\ge&  \p{E \cap \bigcap_{i=1}^m \{ (3n)^{-1/2}C_{\rT_n}(U_j^{(k_i)}) \le c_i\} \cap \{ (4n/3)^{-1/4}Z_{\rT_n}(U_j^{(k_i)}) \le z_i\}} \\
> & \p{\bigcap_{i=1}^m \{ (3n)^{-1/2}C_{\rT_n}(U_j^{(k_i)}) \le c_i\} \cap \{ (4n/3)^{-1/4}Z_{\rT_n}(U_j^{(k_i)}) \le z_i\}}-3\delta \, .
\end{align*}
By Lemma~\ref{lem:fdds}, as $n \to \infty$, 
\[
\begin{split}
\p{\bigcap_{i=1}^m \{ (3n)^{-1/2}C_{\rT_n}(U_j^{(k_i)}) \le c_i\} \cap \{ (4n/3)^{-1/4}Z_{\rT_n}(U_j^{(k_i)}) \le z_i\}} \\
 \to \p{\bigcap_{i=1}^m \{{\bf e}(U_j^{(k_i)}) \le c_i\} \cap \{Z(U_j^{(k_i)}) \le z_i\}}\, ,
\end{split}
\]
which together with the preceding bounds implies that for all sufficiently large $n$, 
\begin{align*}
& \p{\bigcap_{i=1}^m \{{\bf e}(t_i) \le c_i-\delta\} \cap \{Z(t_i) \le z_i-\delta\}}\\
< & \p{\bigcap_{i=1}^m \{ (3n)^{-1/2}C_{\rT_n}(t_i) \le c_i+\delta\} \cap \{ (4n/3)^{-1/4}Z_{\rT_n}(t_i) \le z_i+\delta\}}+6\delta \, .
\end{align*}
A symmetric argument establishes that for all sufficiently large $n$, 
\begin{align*}
& \p{\bigcap_{i=1}^m \{{\bf e}(t_i) \le c_i+3\delta\} \cap \{Z(t_i) \le z_i+3\delta\}}\\
> & \p{\bigcap_{i=1}^m \{ (3n)^{-1/2}C_{\rT_n}(t_i) \le c_i+\delta\} \cap \{ (4n/3)^{-1/4}Z_{\rT_n}(t_i) \le z_i+\delta\}}-6\delta \, ,
\end{align*}
and the result follows. 
\end{proof}

The remainder of the section is thus devoted to proving Lemmas~\ref{lem:fdds} and~\ref{lem:tightness}. 
Before proceeding to this, we state a definition which plays a key role. Given a probability measure $\pi$ on $\R_k$, its {\em symmetrization} $\hat{\pi}$ is obtained by permuting the marginals uniformly at random. More precisely, if $(X_1,\ldots,X_k)$ has law $\pi$ and, independently, $\sigma$ is a uniformly random permutation of $\{1,\ldots,k\}$, then $(X_{\sigma(1)},\ldots,X_{\sigma(k)})$ has law $\hat{\pi}$. 

For the remainder of Section~\ref{sec:snake}, let $\mu$ be the law of the random variable $B$ defined in (\ref{eq:offspring}). Also, for $k \ge 1$ let $\nu_k$ be as in Corollary~\ref{cor:vldist}, and let $\hat{\nu}_k$ be the symmetrization of $\nu_k$. Note that since $\sum_i\int_{\R} x \mathrm{d}\nu_k^i(x)=0$, we have $\int_{\R} x \mathrm{d}\hat{\nu}_k^i(x)=0$ for each $1 \le i \le k$; in other words, $\hat{\nu}=(\hat{\nu}_k,k \ge 1)$ is locally centred. The proofs of Lemmas~\ref{lem:fdds} and~\ref{lem:tightness} both rely on couplings between $\LGW(\mu,\nu,n)$ and $\LGW(\mu,\hat\nu,n)$.

\subsection{Symmetrization of plane trees}
\addtocontents{toc}{\SkipTocEntry}
\label{sub:fulsym}
Fix a spatial plane tree $\rt=(t,\xi,d)$. For the remainder of the section it is convenient to conflate $t$ and its Ulam-Harris encoding. This allows us to identify $t$ with its vertex set; also, since with this coding the root vertex is always $\emptyset=\v(\xi)$, we write $\rt=(t,d)$ instead of $\rt=(t,\xi,d)$.

\nomenclature[Sigma]{$\mathfrak{S}(T)$}{Set of vectors of permutations indexed by the vertices of $T$}
Denote by $\mathfrak{S}(t)$ the set of vectors $\sigma=(\sigma^v:v \in t,k_t(v) > 0)$ indexed by the non-leaf vertices of $\rt$, with $\sigma^v$ a permutation of $\{1,\ldots,k_{t}(v)\}$. 
For $\sigma\in \mathfrak{S}(t)$, the \emph{symmetrization of $t$ with respect to $\sigma$} is the tree $t^{\sigma}$ obtained from $t$ by permuting the order of the subtrees rooted at the children of $v$ according to $\sigma^v$, for each $v \in t$. 
\nomenclature[Tsigma]{$\mathrm{T}^{\sigma}$}{``Permutation'' of $\mathrm{T}=(T,\xi,D)$ 
by $\sigma=(\sigma^v:v \in t,k_{\mathrm{t}}(v)>0)$}
More formally, 
\[
t^\sigma = \{\sigma(v),v \in t\},
\]
where if $v=n_1n_2\ldots n_k \in t$ then 
\[
\sigma(v) = \sigma^{\emptyset}(n_1)\sigma^{n_1}(n_2)\ldots \sigma^{n_1\ldots n_{k-1}}(n_k)\, .
\]
We then let $\rt^{\sigma}=(t^{\sigma},d^{\sigma})$, where 
 $d^{\sigma}(\sigma(u),\sigma(ui)) = d(u,ui)$ 
for all edges $\{u,ui\}$ of $t$. 
\nomenclature[Tsigma]{$\mathrm{T}^{\sigma}$}{``Permutation'' of $\mathrm{T}=(T,\xi,D)$ 
by $\sigma=(\sigma^v:v \in t,k_{\mathrm{t}}(v)>0)$}
Visually, displacements are attached to edges, and follow their edges when the tree is permuted. 
Observe that $\rt$ and $\rt^\sigma$ are isomorphic as rooted edge-labeled trees (but need not be isomorphic as spatial plane trees). The local effect of symmetrization is depicted in Figure~\ref{fig:fullsym}.

\begin{figure}[t]
\hspace{-0.5cm}
 \subfigure[\label{fig:fullsym}Symmetrization.]{\includegraphics[width=.35\linewidth,page=3]{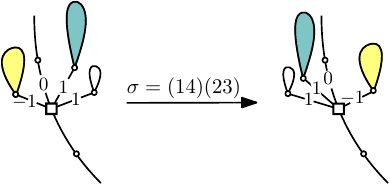}}\qquad \qquad \qquad
\subfigure[\label{fig:partsym}Partial symmetrization.]{\includegraphics[width=.35\linewidth,page=2]{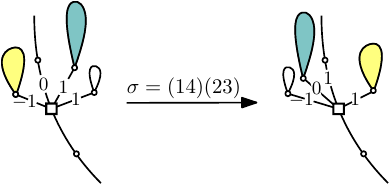}}
\caption{Examples of the local rules for symmetrization and partial symmetrization at the square vertex (which in~\ref{fig:partsym} is presumed to lie in $\branch{t}{\rr}$). 
}
\label{fig:comp-sym}
\end{figure}

\begin{claim}\label{cla:full_sym_law}
Let $\rT=(T,D)$ have law $\LGW(\mu,\nu,n)$, and let $\Sigma$ be a uniformly random element of $\mathfrak{S}(\rT)$. 
Then $\rT^{\Sigma}$ has law $\LGW(\mu,\hat \nu,n)$. 
\end{claim}
\begin{proof}
Since $T$ and $T^\Sigma$ are isomorphic as rooted trees, it follows from the branching property of Galton-Watson processes that they have the same law. The definition of $\hat\nu$, and the fact that $\Sigma$ uniformly permutes labels at every vertex, then imply that $\rT^{\Sigma}$ has law $\LGW(\mu,\hat \nu,n)$. 
\end{proof}

\begin{cor}\label{cor:conv_toosymmetric}
For $n \ge 1$ let $\rT_n=(T_n,D_n)$ be uniformly random in $\cT_n^{\mathrm{vl}}$, and let $\Sigma_n$ be a uniformly random element of $\mathfrak{S}(\rT_n)$. Then as $n \to \infty$, 
\begin{equation}\label{eq:cvsnake_toosymmetric}
\left( (3n)^{-1/2}C_{\rT_n^{\Sigma_n}}(t), (4n/3)^{-1/4}Z_{\rT_n^{\Sigma_n}}(t)\right)_{0\le
t \le 1}\,
\convdist \, ({\bf e}(t), Z(t))_{0\le t \le 1},
\end{equation}
for the topology of uniform convergence on $C([0,1],\R)^2$.
\end{cor}
\begin{proof}
By Corollary~\ref{cor:vldist}, $\rT_n$ has law $\LGW(\mu,\nu,n)$, so $\rT_n^{\Sigma_n}$ has law $\LGW(\mu,\hat \nu,n)$ by Claim~\ref{cla:full_sym_law}. 
Since $\hat{\nu}$ is bounded and locally centred and its marginals are uniformly distributed on $\{-1,0,1\}$, the result follows by (\ref{eq:cvsnake_jm}), 
and the computation (see Appendix~\ref{sec:notes}) that $|\mu|_1=1$, $\sigma_{\mu}/2=3^{-1/2}$, and $\sigma_\nu=(2/3)^{1/2}$. 
\end{proof}

Now fix a vector $\rr$ of vertices of $t$, and let $\branch{t}{\rr}$\nomenclature[path]{$\branch{t}{\rr}$}{The set of ``path-points'' of $t\angles{\rr}$} be the set of ``path-points'' of $t\angles{\rr}$: the vertices of $t\angles{\rr}$ that have at exactly one child in $t\angles{\rr}$. 
Write 
$\mathfrak{S}(t,\rr)$ for the set of vectors $\sigma=(\sigma^v: v \in \branch{t}{\rr})$ with each $\sigma^v$ a permutation of $\{1,\ldots,k(v)\}$. 
Given $\sigma \in \mathfrak{S}(t,\rr)$, extend $\sigma$ to a vector $\tau \in \mathfrak{S}(t)$ by setting 
\[
\tau^v =    \begin{cases}
            \sigma^v    & \mbox{ if } v \in \branch t \rr\, ,\\
            \mathrm{Id}^{k(v)}    & \mbox{ otherwise. }   \\
        \end{cases}
\] 
Then the {\em partial} symmetrization of $\rt$ with respect to $\rr$ and $\sigma$ is the labeled tree $\psym\rt=(\psym t,\psym d)$\nomenclature[tbar]{$\psym\rt$}{The partial symmetrization of $\rt$} with vertices $\psym t = \{\tau(v), v \in t\}=t^{\tau}$ and displacements $\psym d$ given by $\psym d(\tau(u),\tau(u)i) = d(u,ui)$ for all edges $\{u,ui\}$ of $t$. 
Visually, the vector $(d(v,vi),1 \le i \le k(v))$ is now attached to the vertex $v$; this vector follows the vertex when the tree is permuted, but does not change the order of its entries. The partial symmetrization depends on $\sigma$ and on $\rr$, but we omit this from the notation. The local rule for partial symmetrization is illustrated in Figure~\ref{fig:partsym}, and Figure~\ref{fig:coupling} contains an example of partial symmetrization of an entire tree.

In what follows, for $v \in t$ we write $\psym v=\tau(v)$ for the image of $v$ under the partial symmetrization. If $\rr=(r_{1},\ldots,r_{j})$, then we write $\psym{\rr}=\big(\psym{r}_{1},\ldots,\psym{r}_{j}\big)$. We also let $\psym{\sigma}$ be the pushforward of $\sigma$ to $\branch{\psym{t}}{\psym{\rr}}$, so  $\psym{\sigma}( \psym{v} ) = \sigma(v)$ for 
$v \in \branch{\psym{t}}{\psym{\rr}}$. Note that we then have $(\psym{\sigma}(v): v \in \branch{\psym{t}}{\psym{\rr}}) \in \mathfrak{S}(\psym{t},\psym{\rr})$.

\begin{figure}[t]
\hspace{-0.5cm}
\includegraphics[page=2]{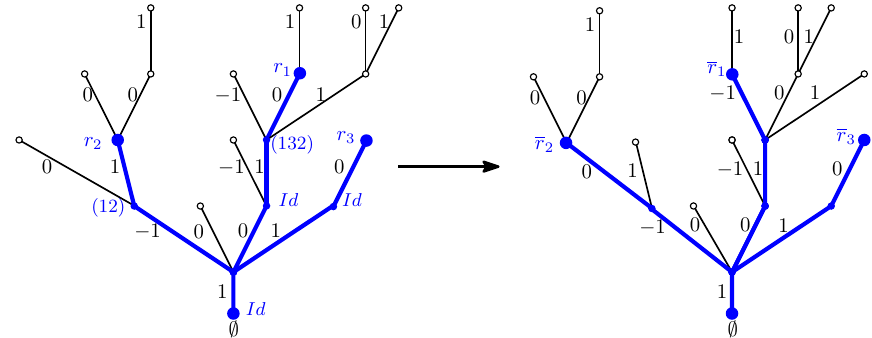}
  \caption{Illustration of partial symmetrization of a tree $\rt=(t,d)$. The vertices in $\rr$ are represented by bigger, solid blue disks. The edges of $t\angles{\rr}$ are thicker and are blue. On the left, $\sigma \in \mathfrak{S}(t,\rr)$ is indicated by listing the permutation at each vertex of $\branch{t}{\rr}$. On the right, the partial symmetrization $\overline{\rt}$ of $\rt$ with respect to $\rr$ and $\sigma$ is shown, as are the images of the elements of $\rr$.}
\label{fig:coupling}
\end{figure}

We remark that $t$ and $\psym t$ are isomorphic as rooted trees, but need not be isomorphic as plane trees or as labelled trees. Here are some comments regarding partial symmetrization. 
\begin{itemize}
\item In forming $\psym\rt$, the order of the children at branchpoints of $t\angles{\rr}$ is not changed. This implies that $t\langle \rr\rangle$ and $\psym t \langle \psym \rr \rangle$ are isomorphic as plane trees. 
\item In particular, if $\rr=(r_1,\ldots,r_k)$ is increasing with respect to lexicographic order in $t$ then $\psym{\rr}=(\psym{r}_1,\ldots,\psym{r}_k)$ is likewise ordered lexicographically in $\psym{t}$. 
\item Partial symmetrization is invertible: $\rt$, $\rr$ and $\sigma$ may be recovered from $\psym{\rt}, \psym{\rr}$ and $\psym{\sigma}$. When we wish to make the dependencies of the symmetrization more explicit, we write $\mathrm{sym}(\rt,r,\sigma)$ instead of $(\psym{\rt},\psym{r},\psym{\sigma})$. 
\end{itemize}

\begin{prop}\label{prop:random_subset}
Let $\rT=(T,D)$ have law $\LGW(\mu,\nu,n)$, let $\rR=(R_1,\ldots,R_j)$ be a vector of $j$ independent and uniformly random vertices of $T$, and let $\Sigma=(\Sigma^v,v \in \branch{T}{\rR})$ be a uniformly random element of $\mathfrak{S}(T,\rR)$. Write $\psym{\rT}=(\psym T,\psym D)$ for the partial symmetrization of $\rT$ with respect to $\Sigma$ and $\rR$, and write $\psym{\rR}=(\psym{R}_1,\ldots,\psym{R}_j)$ for the images of $(R_1,\ldots,R_j)$ in $\psym T$. Then $(\rT,\rR)$ and $(\psym \rT,\psym \rR)$ are identically distributed. 
\end{prop}

\begin{proof}
Fix any pair $\rt'=(t',d') \in \cT_n^{\mathrm{vl}}$, and any vector $\rr' =(r_1',\ldots,r_j')$ of nodes of $\rt'$. 
Next fix $\sigma' \in \mathfrak{S}(t',\rr')$, and 
let $(\rt,r,\sigma)$ be the unique triple for which $\mathrm{sym}(\rt,r,\sigma)=(\rt',r',\sigma')$. 
Then 
\[
\p{(\rT,R,\Sigma)=(\rt,r,\sigma)}
= \p{\rT=\rt} \cdot \frac{1}{n^j} \cdot \prod_{v \in \branch{t}{\rr}} \frac{1}{k_t(v)!}
= \p{\rT=\rt'} \cdot \frac{1}{n^j} \cdot \prod_{v'\in \branch{t'}{\rr'}} \frac{1}{k_{t'}(v')!}\, .
\]
The second equality holds since $\rt$ and $\rt'$ are isomorphic as rooted trees and the vector of labels at each vertex of $t$ is the same as at the one at its image in $t'$. Since $\sigma' \in \mathfrak{S}(t',\rr')$ is arbitrary and $|\mathfrak{S}(t',\rr')|= \prod_{v'\in \branch{t'}{\rr'}} k_{t'}(v')!$, it follows that 
\[
\p{(\psym{\rT},\psym{R})=(\rt',r')} = \frac{\p{\rT=\rt'}}{n^j}\, ,
\]
as required. 
\end{proof}
\begin{cor}\label{cor:fdd_identity}
Let $\rT=(T,D)$ and $\rR=(R_1,\ldots,R_j)$ be as in Proposition~\ref{prop:random_subset}. 
Let $\hat \rT=(\hat T,\hat D)$ have law $\LGW(\mu,\hat{\nu},n)$ and let $\hat\rR = (\hat R_1,\ldots,\hat R_j)$ be a vector of $j$ uniformly random vertices of $\hat T$. 

Write $(V_1,\ldots,V_j)$ and $(\hat V_1,\ldots,\hat V_j)$ for the lexicographic orderings of $\rR$ and of $\hat \rR$. For $1 \le i \le j$ let 
\[
A_i=\sum_{v \in \bbr{\emptyset,V_i}: p(v) \in \branch{T}{\rR}} D(p(v),v) \quad \mbox{and} \quad \hat A_i=\sum_{v \in \bbr{\emptyset,\hat V_i}: p(v) \in \branch{\hat T}{\hat \rR}} \hat D(p(v),v)\, .
\]
Then 
\[
(|V_1|,\ldots,|V_j|,A_1,\ldots,A_j) \eqdist (|\hat V_1|,\ldots,|\hat V_j|,\hat A_1,\ldots,\hat A_j)\, .
\]
\end{cor}
\begin{proof}
Let $(\psym{\rT},\psym{\rR})$ be as in Proposition~\ref{prop:random_subset}. 
If $\{p(u),u\}$ is an edge of $T$ and $p(u) \in \branch{T}{\rR}$ then the partial symmetrization uniformly permutes the children of $p(u)$. Since displacements are not permuted, and are independent on child edges of distinct vertices, it follows that the random variables
\[
\{\psym{D}(p(\psym{u}),\psym{u}):
p(u) \in \branch{T}{\rR}\} = 
\{\psym{D}(p(\psym{u}),\psym{u}):
p(\psym{u}) \in \branch{\psym{T}}{\psym{\rR}}\}. 
\] 
are independent and uniformly distributed on $\{-1,0,1\}$. The conclusion of Proposition~\ref{prop:random_subset} then implies the same holds for the random variables 
$
\{D(p(u),u):
p(u) \in \branch{T}{\rR}\}.
$

Finally, the trees $T$ and $\hat T$ have the same law, so $(|V_1|,\ldots,|V_j|) \eqdist (|\hat V_1|,\ldots,|\hat V_j|)$. More strongly, the subtrees $T\angles{\rR}=T\angles{\rV}$ and $\hat T\angles{\hat\rR}=\hat T\angles{\hat\rV}$ are identically distributed. By the definition of $\hat{\nu}$, the displacements $\{\hat D(p(\hat u),\hat u):
p(\hat u) \in \branch{\hat T}{\hat\rR}\}$ are independent and uniform on $\{-1,0,1\}$, and the result follows. 
\end{proof}

\addtocontents{toc}{\SkipTocEntry}
\subsection{Proof of Lemma~\ref{lem:fdds}}\label{sub:fdd}
For $n \ge 1$ let $\rT_n=(T_n,D_n)$ have law $\LGW(\mu,\nu,n)$. Fix $j \ge 1$, let $U_1,\ldots,U_j$ be independent Uniform$[0,1]$ random variables 
independent of the trees $\rT_n$, and let $U_1^{\uparrow}\ldots,U_{j}^{\uparrow}$ be the increasing ordering of $U_1,\ldots,U_j$.

For $1 \le i \le j$, let $u_i$ be such that 
\[
\{r_{\rT_n}(\floor{(2n-2)\cdot U_i}),r_{\rT_n}(\lceil (2n-2)\cdot U_i \rceil )\} = \{p(u_i),u_i\}\, ,
\]
so that $\{p(u_i),u_i\} \in E(T_n)$ is the edge of $T_n$ being traversed at time $U_i$ by $C_{\rT_n}$

Next, write $u_1^\uparrow,\ldots,u_j^\uparrow$ for the lexicographic ordering of $u_1,\ldots,u_j$. It is straightforward that if none of $u_1,\ldots,u_j$ is an ancestor of another, then the order statistics of $u_1,\ldots,u_j$ and of $U_1,\ldots,U_j$ coincide. In this case, for each $1 \le i \le j$, at time $U_i^\uparrow$ the edge $\{p(u_i^\uparrow),u_i^\uparrow\}$ is being traversed by $C_{\rT_n}$. Furthermore, the probability that one of $u_1,\ldots,u_j$ is an ancestor of another is easily seen to tend to zero as $n \to \infty$.

Recalling the notation $|u_i| = d_{\rT_n}(\emptyset,u_i)$, now observe that $|C_{\rT_n}(U_i)-|u_i|| \leq 1$ and $|Z_{\rT_n}(U_i)-X_{\rT_n}(u_i)| \le 1$ for all $1 \le i \le j$. If none of $u_1,\ldots,u_j$ is an ancestor of another then it follows from the preceding paragraph that $|C_{\rT_n}(U_i^\uparrow)-|u_i^\uparrow|| \leq 1$ and $|Z_{\rT_n}(U_i^\uparrow)-X_{\rT_n}(u_i^\uparrow)| \le 1$ for all $1 \le i \le j$. As this occurs with probability tending to one, to prove the lemma it suffices to show that 
\begin{equation}\label{eq:conv_toshow}
\Big((3n)^{-1/2}|u_i^\uparrow|,(4n/3)^{-1/4}X_{\rT_n}(u_i^\uparrow)\Big)_{1 \le i \le j} \convdist (\be(U_i^\uparrow),Z(U_i^\uparrow))_{1 \le i \le j}\, .
\end{equation}

The elements of $(u_1,\ldots,u_j)$ are independent and uniformly distributed over $T_n\setminus \{\emptyset\}$. We may thus couple $(u_1,\ldots,u_j)$ with a sequence $\rW=(w_1,\ldots,w_j)$ of independent uniformly random elements of $T_n$ so that 
$\p{(u_1,\ldots,u_j)\ne (w_1,\ldots,w_j)} \to 0$ as $n \to \infty$. (Here and below we suppress the dependence of $(w_1,\ldots,w_j)$ on $n$ for readability.) But if $(u_1,\ldots,u_j)= (w_1,\ldots,w_j)$ then the lexicographic reorderings of these vectors are also equal. 
Writing $(w_1^\uparrow,\ldots,w_j^\uparrow)$ for the lexicographic ordering of $(w_1^\uparrow,\ldots,w_j^\uparrow)$, it follows that replacing $u_i^\uparrow$ by $w_i^\uparrow$ for $1 \le i \le j$ does not affect the convergence (or lack thereof) in (\ref{eq:conv_toshow}). 

Next, for each $1 \le i \le j$ write 
$A_i = \sum_{v \in \bbr{\emptyset,w_i^\uparrow} : p(v) \in \branch{T_n}{\rW}} D_n(p(v),v)$. The tree $T_n\angles{R}$ has at most $j$ leaves, so $|A_i - X_{\rT_n}(w_i^\uparrow)| \le j-1$. It follows that replacing $X_{\rT_n}(w_i^\uparrow)$ by $A_i$ for each $1 \le i \le j$ likewise does not affect whether or not (\ref{eq:conv_toshow}) converges in distribution. 
It thus suffices to establish the convergence 
\[
\Big((3n)^{-1/2}|w_i^\uparrow|,(4n/3)^{-1/4}A_i\Big)_{1 \le i \le j} \convdist (\be(U_i^\uparrow),Z(U_i^\uparrow))_{1 \le i \le j}\, .
\]
   
Now let $\hat{\rT}_n=(\hat{T}_n,\hat{D}_n)$ have law $\LGW(\mu,\hat{\nu},n)$, let $\hat{\rW}=(\hat{w}_1,\ldots,\hat{w}_j)$ be uniformly random vertices of $T_n$ and let $(\hat{w}_1^\uparrow,\ldots,\hat{w}_j^\uparrow)$ be their lexicographic reordering. 
With $\hat{A}_i = \sum_{v \in \bbr{\emptyset,w_i^\uparrow}: p(v) \in \branch{\widehat{T}_n}{\widehat{\rR}}} \hat{D}_n(p(v),v)$, Corollary~\ref{cor:fdd_identity} implies that we may replace $w_i^\uparrow$ by $\hat{w}_i^\uparrow$ and $A_i$ by $\hat{A}_i$, without affecting distributional convergence. We may even replace $A_i$ by $X_{\widehat{T}_n}(\hat w_i^\uparrow)$,  since $|\hat{A}_i-X_{\widehat{T}_n}(\hat w_i^\uparrow)| \le j-1$.

In sum, by the above reductions, it suffices to prove that 
\begin{equation}\label{eq:conv_toshowc}
\left((3n)^{-1/2}|\hat{w}_i^\uparrow|,(4n/3)^{-1/4}X_{\widehat{\rT}_n}(\hat w_i^\uparrow)\right)_{1 \le i \le j} \convdist (\be(U_i^\uparrow),Z(U_i^\uparrow))_{1 \le i \le j}\, ,
\end{equation}
To accomplish this we essentially reverse the above chain of reductions, and conclude by applying a known convergence result for globally centered snakes. 

Let $V_1,\ldots,V_j$ be independent Uniform$[0,1]$ random variables 
independent of everything else and let $V_{1}^\uparrow,\ldots,V_{j}^\uparrow$ be the increasing ordering of $V_1,\ldots,V_j$. Then let $v_i$ be such that $\{p(v_i),v_i\}$ is being traversed at time $V_{i}$ by $C_{\hat \rT_n}$, and let $(v_1^\uparrow,\ldots,v_j^\uparrow)$ be the lexicographic ordering of $(v_1,\ldots,v_j)$. 

Reprising the argument from the start of the proof, we see that with probability tending to one, 
for all $1 \le i \le j$ the edge $\{p(v_i^\uparrow),v_i^\uparrow\}$ is being traversed at time $V_{i}^\uparrow$.  
When this occurs, we have $|C_{\widehat \rT_n}(V_i^\uparrow)-d_{\widehat \rT_n}(\emptyset,v_i^\uparrow)| \le 1$ and $|Z_{\widehat \rT_n}(V_i^\uparrow)-X_{\widehat \rT_n}(v_i^\uparrow)| \le 1$. It then follows from Corollary~\ref{cor:conv_toosymmetric}  that 
\begin{equation}\label{eq:convVi}
\left((3n)^{-1/2}d_{\widehat\rT_n}(\emptyset,v_i^\uparrow),(4n/3)^{-1/4}X_{\widehat \rT_n}(v_i^\uparrow)\right)_{1 \le i \le j} \convdist (\be(V_i^\uparrow),Z(V_i^\uparrow))_{1 \le i \le j}\, .
\end{equation}
Finally, $v_1,\ldots,v_j$ are independent uniformly random non-root vertices of $\hat{\rT}$, so the total variation distance between the laws of $(v_1,\ldots,v_j)$ and of $(\hat{w}_1,\ldots,\hat{w}_j)$ tends to zero. It follows that the total variation distance between the laws of $(v_1^\uparrow,\ldots,v_j^\uparrow)$ and $(\hat{w}_1^\uparrow,\ldots,\hat{w}_j^\uparrow)$ also tends to zero, so we may replace $v_i^\uparrow$ by $\hat w_i^{\uparrow}$ in (\ref{eq:convVi}) without changing the limit. The right-hand sides of (\ref{eq:conv_toshowc}) and (\ref{eq:convVi}) are identically distributed, so this completes the proof. \qed

\addtocontents{toc}{\SkipTocEntry}
\subsection{Proof of Lemma~\ref{lem:tightness}}

For $n \ge 1$ let $\rT_n=(T_n,D_n)$ have law $\LGW(\mu,\nu,n)$. 
Recall that $Z_{\rT_n}$ is obtained from $X_{\rT_n}$ by the identity $Z_{\rT_n}(i/(2n-2))=X_{\rT_n}(r(i))$ and by linear interpolation. We shall prove that for all $\eps > 0$ there exists $\delta > 0$ such that 
\begin{equation}\label{eq:tightness_tp}
\limsup_{n \to \infty} \p{\sup_{|i-j| \le \delta n} |X_{\rT_n}(r(i))-X_{\rT_n}(r(j))|>\eps n^{1/4}} < \eps.
\end{equation}
In the above supremum, it should be understood that we restrict to $i,j \in [2n-2]$, but we omit this from the notation. Due to the relation between $Z_{\rT_n}$ and $X_{\rT_n}$, this immediately implies tightness of the family laws of $(Z_{\rT_n},n \ge 1)$, and so proves the lemma.

For each $n \ge 1$ let $\Sigma_n$ be a uniformly random element of $\mathfrak{S}(\rT_n)$, and let $\hat{\rT}_n$ be the symmetrization of $\rT_n$ with respect to $\Sigma_n$. 
By Corollary~\ref{cor:conv_toosymmetric}, as $n \to \infty$, 
\begin{equation}\label{eq:limitingprocess}
\left( (3n)^{-1/2}C_{\widehat\rT_n}(t), (4n/3)^{-1/4}Z_{\widehat\rT_n}(t)\right)_{0\le
t \le 1}\,
\convdist \, ({\bf e}(t), Z(t))_{0\le t \le 1},
\end{equation}
for the topology of uniform convergence on $C([0,1],\R)^2$. 
It follows in particular that the family of laws of the processes $(Z_{\widehat \rT_n},n \ge 1)$ is tight. Since $Z_{\widehat{\rT}_n}$ and $X_{\widehat{\rT}_n}$ are related in the same way as $Z_{\rT_n}$ and $X_{\rT_n}$, this implies that for all $\eps > 0$ there exists $\alpha=\alpha(\eps) > 0$ such that 
\begin{equation}\label{eq:zrntight}
\sup_{n \ge 1} \p{\sup_{|i-j| \le \alpha n} |X_{\widehat \rT_n}(\widehat{r}(i))-X_{\widehat \rT_n}(\hat{r}(j))|>\eps n^{1/4}} < \eps,
\end{equation}
where we write $\hat{r}$ for the contour exploration of $\hat{\rT}_n$. 
We also fix $\eps > 0$ and let $\alpha=\alpha(\eps)$ be small enough that \eqref{eq:zrntight} holds. 

Recall from Section~\ref{sec:limitobject} the definition of the Brownian CRT $(\cT_{\be},d_{\cT_{\be}})$. As noted in that section, the process $Z$ can be seen as having domain $(\cT_{\be},d_{\cT_{\be}})$ and remains a.s uniformly continuous on this domain. It follows that for all $\eps>0$, there exists $\beta>0$ such that: 
\[
\p{\sup\{|Z(x)-Z(y)|,\ \text{for }x,y\in [0,1] \text{ such that }d_{\cT_{\be}}(x,y)\leq \beta\}>\eps}<\eps.
\]
Together with the convergence in (\ref{eq:limitingprocess}) and the relation between $Z_{\widehat \rT_n}$ and $X_{\widehat \rT_n}$, this implies that for all $\eps > 0$ there exists $\beta=\beta(\eps) > 0$ such that
\begin{equation}\label{eq:ConTnZn}
\sup_{n \ge 1} 
\p{
\sup_{u,v \in \widehat T_n: d_{\widehat\rT_{n}}(u,v)\leq \beta n^{1/2}}
|X_{\widehat{\rT}_n}(u)-X_{\widehat{\rT}_n}(v)|
> \eps n^{1/4}
}
< \eps\, .
\end{equation}

For $v \in T_n$ we write $\hat{v}$ for the image of $v$ in $\hat T_n$. The only subtlety of the proof is that $v$ and $\hat{v}$ may be visited at very different times in the contour explorations of $T_n$ and $\hat T_n$. In other words, we typically do not have $\hat r(i) = \widehat{r(i)}$. 
Observe, however, that for any $i$ and $j$, the paths $\bbr{r(i),r(j)}$ in $T_n$ and $\bbr{\widehat{r(i)},\widehat{r(j)}}$ in $\hat{T}_n$ are identical: they have the same length and visit edges with the same labels, in the same order. In particular, we have 
\begin{align*}
d_{T_n}(r(i),r(j))&=d_{\widehat{T}_n}(\widehat{r(i)},\widehat{r(j)})\, , \qquad
X_{\rT_n}(r(i))-X_{\rT_n}(r(j))
 =X_{\widehat{\rT}_n}(\widehat{r(i)})-X_{\widehat{\rT}_n}(\widehat{r(j)})\, .
\end{align*}
Taking $\beta=\beta(\eps)$ as above, (\ref{eq:ConTnZn}) then yields that for all $n \ge 1$,
\begin{align*}
 	& \p{\sup_{|i-j| \le \delta n} |X_{\rT_n}(r(i))-X_{\rT_n}(r(j))| > \eps n^{1/4}}\\
= & \p{\sup_{|i-j| \le \delta n} |X_{\hat \rT_n}(\widehat{r(i)})-X_{\widehat \rT_n}(\widehat{r(j)})| > \eps n^{1/4}}\\
\le & 
\p{\exists u,v \in \widehat T_n:d_{\widehat T_n}(u,v) \le \beta n^{1/2}, |X_{\widehat \rT_n}(u)-X_{\widehat \rT_n}(v)| \ge \eps n^{1/4}} \\
& +  \p{\exists i,j: |i-j| \le \delta n, d_{\widehat T_n}(\widehat{r(i)},\widehat{r(j)}) > \beta n^{1/2}} \\
\le & 
\eps + \p{\exists i,j: |i-j| \le \delta n, d_{\widehat T_n}(\widehat{r(i)},\widehat{r(j)}) > \beta n^{1/2}} \\
= & 
\eps + \p{\exists i,j: |i-j| \le \delta n, d_{T_n}(r(i),r(j)) > \beta n^{1/2}}. 
\end{align*}

Now note that 
\[
\sup\{d_{T_n}(r(i),r(j)): |i-j| \le \delta n\} \le 2 \sup\{|C_{T_n}(x)-C_{T_n}(y)|: |x-y| \le \delta\}. 
\]
By the distributional convergence in (\ref{eq:aldous}) and the a.s.\ continuity of Brownian excursion, it follows that if $\delta>0$ is sufficently small then 
\[
\sup_n \p{\exists i,j: |i-j| \le \delta n, d_{T_n}(r(i),r(j)) > \beta n^{1/2}} < \eps\, .
\]
For such $\delta$ we then have 
\[
\sup_n \p{\sup_{|i-j| \le \delta n} |X_{\rT_n}(r(i))-X_{\rT_n}(r(j))| > \eps n^{1/4}}< 2\eps\, ,
\]
which establishes (\ref{eq:tightness_tp}) and completes the proof. \qed 

\section{Blossoming trees, labelling, and distances}\label{sec:label_dist_determinist}

The goal of this section is to {\em deterministically} relate labels in a validly-labelled plane tree with the distances in the corresponding triangulation. For the remainder of Section~\ref{sec:label_dist_determinist}, we fix $n \in \N$ and $(T,\hat{\xi}) \in \cT_n$, let $\xi\in \cC(T)$ be such that $\rT=(T,\xi)$ is balanced. Finally, let $\rG=(G,c)=\chi(T,\xi)$ and let $(T',\xi',D)=\phi_n(T,\hat\xi)$.

Writing $\cB$ for the buds of $T$, we suppose throughout that $V(T')=V(T)\setminus\cB=V(G)\setminus \{A,B\}$. Finally, define $Y=Y_{\rT}$ as in Section~\ref{sec:labels-displacements}, and note that since $\rT$ is balanced, $Y(v) \ge 2$ for all $v \in V(T)$. It will be useful to extend the domain of $Y$ by setting $Y(A)=1$ and $Y(B)=2$, and we adopt this convention. 
\addtocontents{toc}{\SkipTocEntry}
\subsection{Bounding distances using leftmost paths}
To warm up, we prove a basic lemma bounding the difference between labels of adjacent vertices. 
\begin{lem}\label{lem:3diff}
For all 
$\{u,w\} \in E(G)$, $|Y(u)-Y(w)| \le 3$. 
\end{lem}
\begin{proof}
First, recall from Page~\pageref{eq:scdefine} that if $u \in V(T)$ and $\{u,A\} \in E(G)$ or $\{u,B\} \in E(G)$ then there is a corner $\zeta$ incident to $u$ with $\lambda(\zeta) \le 3$, so $Y(u) \le 3$. From this, if $\{u,w\} \cap \{A,B\} \ne \emptyset$ then the result is immediate. 
Next, if $\{u,v\} \in E(T)$ then it is an inner edge of $T$, in which case $Y(u)-Y(v)=D_{\{u,v\}}(T,\xi) \in \{-1,0,1\}$. 
Finally, if $\{u,w\} \not \in E(T)$ but $u,w \in V(T)$ then there are corners $c^1,c^2$ of $T$ such that $\v(c^1)=u$, $\v(c^2)=w$, and either $c^2=s(c^1)$ or $c^1=s(c^2)$. Assuming by symmetry that $c^2=s(c^1)$, we have $\lambda(c^2)=\lambda_{\rT}(c^1)-1$. Since the labels on corners incident to a single vertex differ by at most two, the result follows in this case. 
\end{proof}
The above lemma, though simple, already allows us to prove the labels provide a lower bound for the graph distance to $A$ in $G$, up to a constant factor.
\begin{cor}\label{cor:3diff}
For all $u \in V(G)$, $d_G(u,A) \ge Y(u)/3$.
\end{cor}
\begin{proof}
Let $(u_0,u_1,\ldots,u_l)$ be a shortest path from $u=u_0$ to $A=u_l$ in $G$. 
By Lemma~\ref{lem:3diff}, since $Y(A)=1$ we have 
$Y(u) = |Y(u_0)-Y(u_l)-1| < 3l$, 
so $d_G(u,A)=l \ge Y(u)/3$. 
\end{proof}
We next aim to prove a corresponding upper bound. For this we use the leftmost paths briefly introduced in Section~\ref{sec:overview}. Let $(G,c)=\chi(T,\xi)$ as above, and let $\overrightarrow{E}$ be its unique minimal 3-orientation (defined in Section~\ref{sec:def-ori}). 
Given an oriented edge $e =uw$ with $\{u,w\} \in E$ and $x \in V(G)$, a {\em path} from $e$ to $x$ is a path $Q=(v_0,v_1,\ldots,v_m)$ in $G$ with $v_0v_1=uw$ and $v_m=x$. (In the preceding, we do not require that $uw \in \overrightarrow{E}$.) Given $e=\{u_0,u_1\} \in E(G)$ with $u_0u_1 \in \overrightarrow{E}$, the {\em leftmost path} from $e$ to $A$ is the unique directed path $P(e)=P_{(G,c)}(e)=(u_0,u_1,\ldots,u_\ell)$ with $u_{\ell}=A$ such that for each $1 \le i \le \ell-1$, $u_iu_{i+1}$ is the first outgoing edge incident to $u_i$ when considering the edges incident to $u_i$ in clockwise order starting from $\{u_{i-1},u_i\}$. The following fact establishes two basic properties of leftmost paths. 
\begin{fact}\label{fact:LMP}
For all $e \in E(G)$, $P(e)$ is a simple path. 
Furthermore, if $P(e)=(u_0,u_1,\ldots,u_{\ell})$ and $P(e')=(v_0,v_1,\ldots,v_m)$ are distinct leftmost paths to $A$ with $u_0=v_0=u$, and $u_i=v_j$ for some $i,j > 0$, then $u_{i+k}=v_{j+k}$ for all $0 \le k \le \ell-i=m-j$. 
\end{fact}
\begin{proof}
Let $P(e)=(u_0,u_1,\ldots,u_{\ell})$ be the leftmost path from $u_0u_1$ to $A$. Suppose there are $0 \le i < j \le \ell$ such that $u_i=u_j$, and choose such $i,j$ for which $|j-i|$ is minimum. Then $C=(u_i,u_{i+1},\ldots,u_j)$ is an oriented cycle with $j-i$ vertices; let $V' \subset V(G)$ be the vertices lying on or to the right of this cycle. Since $\overrightarrow{E}$ is minimal, $C$ is necessarily a clockwise cycle, so $\v(c) \not\in V'$. Also, neither $A$ nor $B$ are in any directed cycles, and it follows that $\{A,B,\v(c)\} \cap V'=\emptyset$. Since $\overrightarrow{E}$ is a 3-orientation it follows that for all $x \in V'$, $\deg_{\overrightarrow{E}}^+(x)=3$. Furthermore, for all $x \in V'\setminus \{u_i\}$, since $P(e)$ is a leftmost path, all out-neighbours of $x$ are elements of $V'$. Writing $G'$ for the sub-map of $G$ induced by $V$, it follows that $|E(G')| \ge 3|V'\setminus\{u_i\}|$. But $G'$ is a simple planar map, and $C$ is a face of $G'$ of degree $j-i \ge 3$. It follows by Euler's formula that $|E(G')| \le 3|V'|-3-(j-i) \le 3|V'\setminus\{u_i\}|-3$, a contradiction. 

The proof that two leftmost paths merge if they meet after their starting point follows the same lines and is left to the reader.
\end{proof}
The next proposition provides a key connection between the corner labelling $\lambda=\lambda_{\rT}$ and the lengths of leftmost paths. 
\begin{prop}\label{prop:LabelsLMP}
For any edge $e=\{u,w\} \in E(G)$ with $uw \in \overrightarrow{E}$ and $u\ne B$, $\lambda(\kappa^{\ell}(u,w)) = |P(e)|$. 
\end{prop}
\begin{proof}
First, a simple counting argument shows that if $\{x,y_1\}$ is an inner edge of $T$, and $\{x,y_2\}$ is the first stem following $\{x,y_1\}$ in clockwise order around $x$, then writing $\zeta$ for the corner of $T$ incident to $y_2$ we have $\lambda(\kappa^{r}(x,y_1))=\lambda(\zeta)$. 
Recall the definition of the {\em successor} function $s$ from (\ref{eq:label_bijection}) and the equivalent definition from Section~\ref{sec:bij_with_labels}. Since $\{x,y_2\}$ is a stem, in $G$, $y_2$ is identified with $s(\zeta)$, and by definition $\lambda(s(\zeta))=\lambda(c)-1$. 

Next, recall the definition of the labelling  $\lambda^*=\lambda^*_{\rT}:\cC(G) \to \Z^{\ge 0}$ from the end of Section~\ref{sec:bij_with_labels}. It follows from that definition that for any oriented edge $xy \in \overrightarrow{E}$, $\lambda^*(\kappa^r(y,x))=\lambda^*(\kappa^{\ell}(x,y))-1$. In other words, the label on the left decreases by exactly one when following any oriented edge. 

Now write $P(e)=(u_0,u_1,\ldots,u_{\ell})$. Since there are no edges oriented away from $P(e)$ leaving $P(e)$ to the left, it follows from the two preceding paragraphs that for  $0 < i < \ell$ we have $\lambda^*(\kappa^r(u_{i+1},u_i))=\lambda^*(\kappa_{\ell}(u_i,u_{i+1}))-1=\lambda^*(\kappa^r(u_i,u_{i-1}))-1$, 
so 
\[
\lambda^*(\kappa^r(u_{\ell},u_{\ell-1}))=\lambda^*(\kappa^{\ell}(u_0,u_1))-\ell=\lambda^*(\kappa^{\ell}(u,w))-\ell\, .
\] 
Finally, $\lambda^*(\kappa^r(u_{\ell},u_{\ell-1}))=1$ by definition since $u_{\ell}=A$ and $u_{\ell-1}\neq \v(\xi)$. We thus obtain 
$\lambda^*(\kappa^{\ell}(u,w))=\ell+1=|P(e)|$. 
\end{proof}
\begin{cor}\label{cor:LabelsLMP}
For all $u \in V(G)$, $d_G(u,A) \le Y(u)-1$.
\end{cor}
\begin{proof}
Recall the convention that $Y(B)=2$ and $Y(A)=1$; since also $Y(\v(\xi))=2$, it suffices to prove the result for $u \in V(G)\setminus\{A,B,\v(\xi)\}$. For such $u$, if $\{u,w\}$ is the first stem incident to $u$ in clockwise order around $u$ starting from $\{u,p(u)\}$, then $Y(u)=\lambda(\kappa^{\ell}(u,w))$. The claim then follows from Proposition~\ref{prop:LabelsLMP}. 
\end{proof}

\addtocontents{toc}{\SkipTocEntry}
\subsection{Bounding distances between two points using modified leftmost paths.}\label{sub:2points}
In this section we use arguments similar to those of the preceding section, this time to prove deterministic upper bounds on pairwise distances in $G$. 
Fix two inner corners $\zeta_1,\zeta_2$ of $T$, and define \nomenclature[YTc]{$\check Y_\rT(\zeta_1,\zeta_2)$}{For $\zeta_1,\zeta_2 \in \mathcal{C}(\rT)\backslash \mathcal{B}$, $\check Y_\rT(\zeta_1,\zeta_2) = \min\{Y_\rT(w): \exists\zeta \in [\zeta_1,\zeta_2]_{\mathrm{cyc}}, w=\v(\zeta)\}$. 
}
\[
\check Y_\rT(\zeta_1,\zeta_2) = \min\{Y_\rT(w): \exists
\zeta \in [\zeta_1,\zeta_2]_{\mathrm{cyc}}, w=\v(\zeta)\}. 
\]
\begin{prop}\label{prop:2points}
For all $u,v\in V(G)\backslash \{A,B\}$, and for any corners $\zeta_u,\zeta_v$ of $T$ respectively incident to $u$ and $v$, we have
\[
d_G(u,v) \le Y_\rT(u)+Y_\rT(v) 
- 2\max\{\check Y(\zeta_u,\zeta_v), \check Y(\zeta_v,\zeta_u)\}+ 6.
\]
\end{prop}
 Before proving the proposition, we establish some preliminary results. Given an edge $e=u_0u_1$ with $\{u_0,u_1\}\in E(T)$ for which $u_0\notin \cB(T)$, we define the \emph{modified leftmost path} from $e$ to $A$ to be the unique (not necessarily oriented) path $Q(e)=(u_0,u_1,\ldots,u_\ell)$ in $G$ with $u_\ell = A$ and such that for each $1\le i\le \ell-1$, $u_iu_{i+1}$ is the first edge (considering the edges incident to $u_i$ in clockwise order starting from $\{u_{i-1},u_i\}$) which is either an outgoing edge (with respect to the orientation $\overrightarrow{E}$) incident to $u_i$ or an \emph{inner edge of $T$}. Equivalently, it is the leftmost oriented path, with the modified orientation obtained by viewing edges of $E(T)$ as unoriented (or as oriented in both directions).

We view $Q(e)$ as an oriented path from $e$ to $A$ (though the edge orientations given by the path need not agree with $\overrightarrow{E}$); we may thus speak of the left and right side of $Q(e)$.
\begin{fact}\label{fact:lengthMLMP}
For $1\leq i\leq\ell-1$, $ \lambda^*(\kl {u_{i}}{u_{i+1}})= \lambda^*(\kl {u_{i-1}}{u_{i}}) -1 $. In other words, the labels along the left of a modified leftmost path decrease by one along each edge. 
\end{fact}
\begin{proof}
First, by the definitions of $\lambda$ and $\lambda^*$, for any edge $\{u_{i-1},u_{i}\}$ of a modified leftmost path, $\lambda^*(\kr{u_{i}}{u_{i-1}})=\lambda^*(\kl{u_{i-1}}{u_{i}})-1$. 
Moreover, from the definition of a modified leftmost path, there is no stem incident to $u_{i}$ in $T$ that lies strictly between $\{u_{i-1},u_{i}\}$ and $\{u_{i},u_{i+1}\}$ (in clockwise order around $u_i$ starting from $\{u_{i-1},u_i\}$).  Hence $\kl {u_{i}}{u_{i+1}}=\kr{u_{i}}{u_{i-1}}$ in $T$ (see the proof of Proposition~\ref{prop:LabelsLMP} for more details). The result follows. 
\end{proof}
Given $\{u,v\} \in E(G)$, if $\{u,v\} \not\in E(T)$ and $\{u,v\} \ne \{A,B\}$ then by symmetry we may assume there is an edge $\{u,b\} \in E(T)$ such that $v=\v(s(b))$. In this case, by a slight abuse of notation we write $\kl u v = \kl u b$.

In the statement and proof of the next fact, write $L=\lambda^*(\kl{u_0}{u_1})$ and let $M = \min\{\lambda(\xi):\xi \in \cC(T),\kl{u_0}{u_1} \pc \xi\}$, where we view $\{u_0,u_1\}$ as an edge of $E(T)$. By the discussion on Page~\pageref{eq:scdefine}, $M \in \{2,3\}$ and $M=3$ precisely if $\bar \xi \pc \kl{u_0}{u_1}$, where $\bar \xi$ is the unique element of $\cC(T)\setminus \{\xi\}$ for which $(T,\bar \xi)$ is balanced.

Let $c^*_e(0)=\kl{u_0}{u_1}$, and for $1 \le j \le L-M$ let $c^*_e(j)$ be the first corner $\zeta$ following $c^*_e(0)$ in $T$ for which $\lambda(\zeta)=L-j$. 
For $1 \le j \le L-M$, $c^*_e(j)$ is necessarily an inner corner of $T$. 
\begin{fact}\label{fact:mincorner}
For all $0\leq j \leq L-M$, $c^*_e(j) = \kl{u_j}{u_{j+1}}$, so $\v_T(c^*_e(j))=u_j\in Q(e)$.
\end{fact}
Before giving the proof, observe that this property need not hold for a regular leftmost path; this is the reason we require modified leftmost paths.
\begin{proof}
For $j=0$ this holds by definition; we now fix $j \ge 1$ and argue by induction. The definition of $\lambda$ yields that $c^*_e(j)\pct c^*_e(j+1)$, for any $0\le j<L-2$. We consider two cases. First, suppose $\{u_{j-1},u_{j}\}$ is an inner edge of $T$. Let $w \in V(T)$ be such that $\kr{u_{j}}{u_{j-1}}=(\{u_{j-1},u_{j}\},\{u_{j},w\})$. If $w$ is an inner vertex then 
$w=u_{j+1}$. Likewise, if $w$ is a blossom then $\v(s(w))=u_{j+1}$. In either case, $\kl{u_{j}}{u_{j+1}}=\kr{u_{j}}{u_{j-1}}$ in $T$.  
Hence $\kl{u_{j}}{u_{j+1}}$ is the corner immediately following $\kl{u_{j-1}}{u_{j}}$ in the contour exploration of $T$.
By Fact~\ref{fact:lengthMLMP}, $\lambda^*(\kl{u_{j}}{u_{j+1}})=\lambda^*(\kl{u_{j-1}}{u_{j}})-1$ , and $c^*_e(j-1)=\kl{u_{j-1}}{u_{j}}$ by the inductive hypothesis. It follows that $c^*_e(j)=\kl{u_{j}}{u_{j+1}}$. 

Second, suppose $\{u_{j-1},u_{j}\}$ is not an inner edge. By definition, there is no edge in $T$ incident to $u_{j}$ and lying strictly between $\{u_{j-1},u_{j}\}$ and $\{u_{j},u_{j+1}\}$ in clockwise order around $u_{j}$. Hence, in $T$, $s(u_{j-1})=\kl{u_{j}}{u_{j+1}}$. In this case the result follows by the definition of $s(u_{j-1})$ and by induction.
\end{proof}

\begin{centering}
\begin{figure}[ht]
\hspace{-0.5cm}
 \includegraphics[width=0.8\linewidth,page=3]{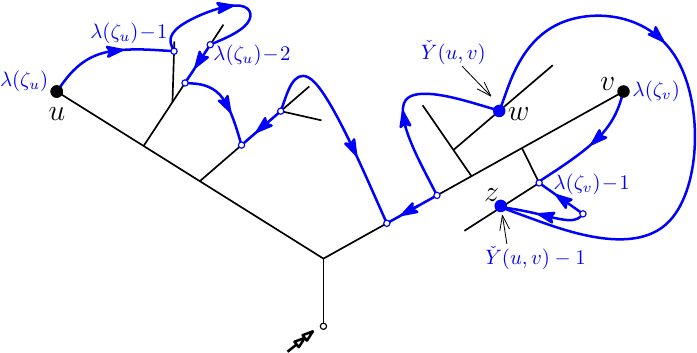}
\caption{\label{fig:MLMP}Path between $u$ and $v$ formed by concatenating sections of two modified leftmost paths. Arrows indicate orientation in $\protect \overrightarrow{E}$ . Straight arrows along the path are edges of $E(T)$; curved arrows are edges of $E(G)\setminus E(T)$.}
\end{figure}
\end{centering}

\begin{proof}[Proof of Proposition~\ref{prop:2points}]
By symmetry, we may assume that $\zeta_u \pct \zeta_v$. Let $s_u$ and $s_v$ be the edges lying to the right of these corners in $T$, and let $e_u$ and $e_v$ be the corresponding edges in $G$. Write $Q(e_u)=(u_0,u_1,\ldots,u_{\lambda(\zeta_u)-1})$ with $u_0=u$ and $e_u=\{u_0,u_1\}$, and likewise write $Q(e_v)=(v_0,v_1,\ldots,v_{\lambda(\zeta_v)-1})$. 
Observe that $\lambda^*(\kl{u_0}{u_1})=\lambda(\zeta_u)$ and likewise $\lambda^*(\kl{u_0}{u_1})=\lambda(\zeta_u)$.

We assume for simplicity that $\check Y(u,v) > 3$ (when $\check Y(u,v) \le 3$ there is a minor case analysis involving the presence of vertices $A$ and $B$ in $Q(e_u)$ and $Q(e_v)$; the details are straightforward and we omit them). 
By Fact~\ref{fact:mincorner} and the definition of $\check Y(u,v)$, necessarily $c_{e_u}(\lambda(\zeta_u)-\check Y(u,v)+1)$ and 
$c^{*}_{e_v}(\lambda(\zeta_v)-\check Y(u,v)+1)$ are incident to the same vertex $z=u_{\lambda(\zeta_u)-\check Y(u,v)+1} = v_{\lambda(\zeta_v)-\check Y(u,v)+1}$. Let $P$ be the concatenation of the subpath of $Q(e_u)$ from $u=u_0$ to $z$ with the subpath of $Q(e_v)$ from $z$ to $v_0=v$ (see Figure~\ref{fig:MLMP} for an illustration). Then $P$ connects $u$ and $v$ in $G$, so 
\begin{align*}
d_G(u,v) &\le |P|-1\\
&=\lambda(\zeta_u)+\lambda(\zeta_v)-2\check Y(\zeta_u,\zeta_v)+2\\
&\leq Y(u)+Y(v) - 2\check Y(\zeta_u,\zeta_v)+ 6,
\end{align*} where the last inequality comes from the fact that $|\lambda(\zeta_u)-Y(u)|\leq 2$ and $|\lambda(\zeta_v)-Y(v)|\leq 2$. A symmetric argument proves the existence of a path $P'$ in $G$ between $v$ and $u$ of length $\lambda(\zeta_u)+\lambda(\zeta_v)-2\check Y(v,u)+6$; this gives the desired bound.
\end{proof}


\addtocontents{toc}{\SkipTocEntry}
\subsection{Winding numbers and distance lower bounds}
 It turns out that the lower bound on $d_G(u,A)$ given by Corollary~\ref{cor:3diff} can be improved by considering winding numbers around $u$; we now remind the reader of their definition. 

Consider a closed curve $\gamma:[0,1] \to \R^2\setminus \{0\}$, and parametrize $\gamma$ in polar coordinates as $((r(t),\theta(t)),0 \le t \le 1)$ so that $\theta$ is a continuous function. We define the {\em winding number of $\gamma$} around zero to be $(\theta(1)-\theta(0))/(2\pi)$. 
Next, fix a reference point $r \in \S^2$. For $x \in \S^2\setminus\{r\}$ and a closed curve $\gamma:[0,1] \to \S^2\setminus\{r,x\}$, let $\varphi:\S^2\setminus\{r\}$ to $\R^2$ be a homeomorphism with $\varphi(x)=0$, and 
 define the winding number $\mathrm{wind}_r(x,\gamma)$ of $\gamma$ around $x$ to be the winding number of $\varphi\circ\gamma:[0,1]\to \R^2$ around zero. It is straightforward that this definition does not depend on the choice of $\varphi(x)$. 
 
In what follows, it is useful to imagine having chosen a particular representative from the equivalence class of $G$, or in other words a particular planar embedding (it is straightforward to verify that the coming arguments do not depend on which embedding is chosen).  Let $r$ be any point in the interior of the face of $G$ incident to $c$.
\begin{dfn}\label{dfn:windingnumber}
Fix an oriented edge $e=uw \in \overrightarrow{E}$ and a simple path $Q=(v_0,v_1,\ldots,v_m)$ from $u$ to $A$. 
Define the {\em winding number $w(Q,e)=w_{\rG}(Q,e)$ of $Q$ around $e$} as follows. Write $P(e)=(u_0,u_1,\ldots,u_{\ell})$. Note that $u_0=v_0=u$, $u_1=w$ and $u_{\ell}=v_m=A$. Form a cycle $C=(v_0,v_1,\ldots,v_m=u_{\ell},u_{\ell-1},\ldots,u_0)$. Then 
fix a point $x$ in the interior of the face incident to $\kr{u}{w}$, 
and let $w(Q,e)=\mathrm{wind}_r(x,C)$. 
\end{dfn}
In the preceding definition, we conflate $C$ with its image in $\S^2$ under the embedding of $G$ (and likewise with $x$); 
it is straightforward to verify that $w(Q,e)$ does not depend on the choice of such embeddings. 
\begin{prop}\label{prop:windingbound}
For all $e=uw \in \overrightarrow{E}$, if $Q$ is a simple path from $u$ to $A$ then $|Q| \ge |P(e)|+2(w(Q,e)-2)$. 
\end{prop}
\begin{proof}
Write $P(e)=(u_0,u_1,\ldots,u_{\ell})$. Let $R=(w_0,w_1,\ldots,w_k)$ be a simple path meeting $P(e)$ only at $w_0$ and $w_k$, with $w_0=u_i$, $w_k=u_j$ for some $0 \le i < j \le \ell$. 
If $j < \ell$ then let $\hat{c}=\kappa^r(u_j,u_{j+1})$ and if $j=\ell$ (so $u_j=A$) then let $\hat{c}$ be the corner of the root face of $(G,c)$ incident to $A$.

We say $R$ {\em leaves $P(e)$ from the right} if $i > 0$ and the corner $\kappa^r_G(u_i,u_{i+1})$ precedes $\kappa^r_G(u_{i},w_1)$ in clockwise order around $u_i$ starting from $\kappa^r_G(u_i,u_{i-1})$. Otherwise say that $R$ leaves $P(e)$ from the left; in particular, if $i=0$ then $R$ leaves from the left by convention. Likewise, $R$ {\em returns to $P(e)$ from the right if $\hat{c}$ precedes $\kappa^r(u_j,w_{k-1})$} in clockwise order around $u_j$ starting from $\kappa^r(u_j,u_{j-1})$; otherwise say that $R$ returns to $P(e)$ from the left. 

The key to the proof is the following set of inequalities. Note that $k=|R|-1$.
\begin{enumerate}
\item If $R$ leaves $P(e)$ from the right and returns from the left then $k \ge j-i-2$.
\item If $R$ leaves $P(e)$ from the left and returns from the left then $k \ge j-i$.
\item If $R$ leaves $P(e)$ from the left and returns from the right then $k \ge j-i-1+2(\I{i>0}+\I{j < \ell})$.
\item If $R$ leaves $P(e)$ from the right and returns from the right then $k \ge j-i-1+2\I{j < \ell}$.
\end{enumerate}
For later use, we say $R$ has {\em type 1} if $R$ leaves $P(e)$ from the right and returns from the left, and define types $2,3$, and $4$ accordingly. Let $C=(w_0,\ldots,w_k=u_j,\ldots,u_i)$ be the cycle contained in the union of $R$ and $P(e)$. We provide the details of the bounds from (1) and (3), as (2) and (4) are respectively similar. 

Note that although $C$ does not respect the orientation of edges given by $\overrightarrow{E}$, it is nonetheless an oriented cycle, so it makes sense to speak of the right- and left-hand sides of $C$. 
For~(1), let $V'$ be the set of vertices on or to the right of $C$, and let $G'$ be the submap of $G$ induced by $V'$. All faces of $G'$ have degree three except $C$, which has degree $k+j-i$. By Euler's formula it follows that $|E(G')|=3|V'|-3-(k+j-i)$. 

For $i < m < j$, since $P(e)$ is a leftmost path, $|\{x \in V': u_mx \in \overrightarrow{E}\}|=1$. Also, since $R$ returns from the left, we must have $w_{k-1}u_j \in \overrightarrow{E}$ (or else $w_{k-1}=u_{j+1}$, which contradicts that $P(e)$ meets $R$ only at its endpoints), so $|\{x \in V': u_jx \in \overrightarrow{E}\}|=0$.
Since $\overrightarrow{E}$ is a $3$-orientation, it follows that $E(G') \le 3|V'|-2(j-i)-1$, which combined with the equality of the preceding paragraph yields that $k \ge j-i-2$. 

For (3) let $V'$ be the set of vertices on or to the left of $C$. Euler's formula again yields $|E(G')|=3|V'|-3-(k+j-i)$. For $x \in V'$ not lying on $C$, we have $x \not\in\{A,B,\v(c)\}$, so since $\overrightarrow{E}$ is a $3$-orientation, $|\{y \in V':xy \in \overrightarrow{E}\}|=3$. 
For $i < m < j-1$ we have $m < \ell-1$, so $u_m$ is not on the root face; since $R$ returns from the right, it follows that $|\{y \in V': u_my \in \overrightarrow{E}\}|=3$. Lastly, $|\{x \in V': u_{j-1}x \in \overrightarrow{E}\}| \ge 1$ since $u_{j-1},u_j$ lies on $C$, and likewise $|\{x \in V': u_{i}x \in \overrightarrow{E}\}| \ge 1$. The edges of $R$ are disjoint from the sets of edges counted above, so 
\[
|E(G')| \ge 3|V'\setminus\{w_1,\ldots,w_{k},u_i,u_{j-1}\}|+2+k = 3|V'|-2k-4. 
\]
Combined with the equality given by Euler's formula this yields $k \ge (j-i)-1$. Next, since $P(e)$ is leftmost, $u_m \not\in\{A,B,\v(c)\}$ for $m < \ell-1$, which is straightforwardly seen. Thus, if $j < \ell$ then since $\overrightarrow{E}$ is a 3-orientation, we in fact have $|\{x \in V': u_{j-1}x \in \overrightarrow{E}\}|=3$, and the same counting argument yields that $k \ge (j-i)+1$. Similarly, if $i > 0$ then 
$|\{x \in V': u_{i}x \in \overrightarrow{E}\}| =3$ and again $k \ge (j-i)+1$. Finally, if $0 < i < j < \ell$ then the same argument yields $k \ge (j-i)+3$. 

To conclude, subdivide the path $Q$ into edge-disjoint sub-paths $R_1,\ldots,R_t$, each of which is either a sub-path of $P(e)$ or else meets $P(e)$ only at its endpoints. We assume $R_1,\ldots,R_t$ are ordered so that $Q$ is the concatenation of $R_1,\ldots,R_t$, so in particular, $u=u_0$ is the first vertex of $R_1$, $A=u_{\ell}$ is the last vertex of $R_t$, and for $1 \le s < t$ the last vertex of $R_s$ is the first vertex of $R_{s+1}$.  

For $1 \le i \le 4$, let $n_i$ be the number of sub-paths of type $i$ among $\{R_1,\ldots,R_t\}$. Since $R_t$ is the only sub-path that intersects the root face, and $R_1$ is the only sub-path which may contain $u=u_0$, the above inequalities and a telescoping sum give
\[
|Q|= 1+\sum_{s=1}^t (|R_i|-1) \ge |P(e)|- 2n_1 +3n_3 +n_4-2(\I{R_1~\mathrm{has~type}~3}+\I{R_t~\mathrm{has~type}~3~\mathrm{or}~4})\, .
\]
In particular, we obtain the bound $|Q| \ge |P(e)|+2(n_3-n_1-2)$. 
Finally, sub-paths that leave from the right and return from the left correspond to clockwise windings of $C$ around $u$, and subpaths that leave from the left and return from the right correspond to counterclockwise windings of $C$ around $u$. It follows that $n_3-n_1$ 
is precisely the winding number $w(Q,e)$; this completes the proof. 
\end{proof}
In what follows, if $C$ is an oriented cycle in $G$ then we write $V^l(C)$ (resp.\ $V^r(C)$) for the sets of vertices lying on or to the left (resp.\ on or to the right) of $C$, and note that $V^l(C) \cap V^r(C)= V(C)$. 
\begin{prop}\label{prop:cyclecut1}
For all $e=uw \in \overrightarrow{E}$, if $Q$ is a shortest path from $u$ to $A$ and $w(Q,e)<0$ then 
there is a cycle $C$ in $G$ such that $G[V^l(C)]$ and $G[V^r(C)]$ each have diameter at least $\lfloor -w(Q,e)/2 \rfloor - 2$, and such that $\max_{y \in V^h(C)} Y_{\rT}(y) - \min_{y \in V^h(C)} Y_{\rT}(y) \ge \lfloor -w(Q,e)/2 \rfloor$ for $h \in \{l,r\}$. 
\end{prop}
\begin{proof}
We write $Q=(u_0,u_1,\ldots,u_{\ell})$, and 
partition $Q$ into edge-disjoint sub-paths $R_1,\ldots,R_t$ as at the end of the proof of Proposition~\ref{prop:windingbound}. For $1 \le s \le t$ and $1 \le i \le 4$, let $n_i(s)$ be the number of sub-paths of type $i$ among $\{R_1,\ldots,R_s\}$. 
If $u_0u_1 \ne e$ then by definition $R_1$ leaves $P(e)$ from the left, so $n_1(1)=0$. 

Let $m = \lfloor -w(Q,e)/2 \rfloor$, 
and let $s$ be minimal so that $n_1(s)-n_3(s)=m$; necessarily, $R_s$ has type 1. 
Also, $s \ge m+1$, and since $n_1(t)-n_3(t)=-w(Q,e)\ge 2m$ we also have $m \le t-m$. 
Write $R_s=(w_0,w_1,\ldots,w_k)$, with $w_0=u_i$, $w_1=u_j$ for distinct $i,j \in \{1,\ldots,\ell\}$. By reversing 
$R_s$ if necessary, we may assume $i < j$,\footnote{It is not hard to prove that there is always some shortest path $Q$ for which the ordered sequence of intersections with $P(e)$ respect the orientation of $P(e)$, so that there is no need to reverse $R_s$ to ensure $i < j$. However, we do not require such a property for the current proof.} 
and write 
$C=(w_0,w_1,\ldots,w_k,u_{j-1},\ldots,u_i=w_0)$.
Since $m+1 \le s \le t-m$, the concatenation of $R_1,\ldots,R_{s-1}$ has length at least $m$ and so does the concatenation of $R_{s+1},\ldots,R_t$. Since $Q$ is a shortest path from $u$ to $A$, it follows that 
 for $0 \le i \le k$, $d_G(u,w_i) \ge m+i+1$ and $d_G(A,w_i) \ge m+(k-i)$. 

Fix $0 < a < j-i$ and let $S$ be a shortest path from $u$ to $u_{i+a}$. The concatenation of $S$, $(u_{i+a},\ldots,u_j)$, and $R_{s+1},\ldots,R_t$ has $d_G(u,u_{i+a})+(j-i-a)+d_G(u_j,A)$ edges. 
On the other hand, by the inequality in (1) from the proof of Proposition~\ref{prop:windingbound}, we have $k \ge j-i-2$, so $Q$ has at least $d_G(u,u_i)+j-i-2+d_G(u_j,A)$ edges. Since $Q$ is a shortest path, it follows that 
\[
d_G(u,u_{i+a}) \ge d_G(w,u_i)+a-2 \ge m+a-2 \ge m-1. 
\]
A similar argument shows that for all $0 < a < j-i$, $d_G(A,u_{i+a}) \ge m-1$. 
Finally, one of $G[V^l(C)]$ or $G[V^r(C)]$ contains $R_1,\ldots,R_s$, and the other contains $R_{s},\ldots,R_t$. Therefore, each of $G[V^l(C)]$ and $G[V^r(C)]$ contains at least $m$ vertices of $P(e)$; since vertex labels strictly decrease along $P(e)$, the final claim of the proposition follows. 
\end{proof}
\begin{prop}\label{prop:cyclecut2}
For all $e=uw \in \overrightarrow{E}$, if $Q$ is a shortest path from $u$ to $A$ and $w(Q,e)< -2$ then 
there is an oriented cycle $C$ in $G$ of length at most $6(|Q|-1)/(-w(Q,e)-2)$ such that $G[V^l(C)]$ and $G[V^r(C)]$ each have diameter at least $\lfloor -w(Q,e)/3 \rfloor - 2$ 
and such that $\max_{y \in V^h(C)} Y_{\rT}(y) - \min_{y \in V^h(C)} Y_{\rT}(y) \ge \lfloor -w(Q,e)/3 \rfloor$ for $h \in \{l,r\}$.
\end{prop}
\begin{proof}
The proof is very similar to that of Proposition~\ref{prop:cyclecut1}, so we omit most details. 
Partition $Q$ into $R_1,\ldots,R_t$ and define $n_i(s)$, $1 \le s \le t$, $1 \le i \le 4$ as before. 
Let $m = \lfloor -w(Q,e)/3 \rfloor$. There are at least $m$ values of $s$ such that $R_s$ has type 1 and 
$m+1 \le n_1(s)-n_3(s) \le 2m$; among these, let $s^{\star}$ minimize $|R_{s^{\star}}|$. Then 
\[
d_G(u,A) \ge m\cdot (|R_{s^{\star}}|-1) \ge \frac{-(w(Q,e)+2)}{3} \cdot (|R_{s^{\star}}|-1). 
\]
The sub-path of $P(e)$ joining the endpoints of $R_{s^{\star}}$ has at most two more edges than $R_{s^{\star}}$, so the cycle formed by this sub-path of $P(e)$ and $R_{s^{\star}}$ has at most $2|R_{s^{\star}}| \le 6(d_G(w,A)+1)/(-w(Q,e)-2)=6(|Q|-1)/(-w(Q,e)-2)$ vertices. The remainder of the proof closely follows that of Proposition~\ref{prop:cyclecut1}. 
\end{proof}

\section{Labels approximate distances for random triangulations}\label{sec:label_approx}
Fix $n \in \N$, and let $(T_,\hat{\xi})$ be uniformly distributed in $\cT_n$. (We will later take $n \to \infty$, but suppress the dependence of $(T,\hat{\xi})$ on $n$ for readability.) Define $(T',\xi',D)=\phi_n(T,\hat \xi)$. Next, as in Corollary~\ref{cor:uniftreelaw}, let $\xi^1,\xi^2\in \cC(T)$ be such that $(T,\xi^i)$ is balanced for $i\in\{1,2\}$. Conditionally given $(T,\hat \xi)$, choose $\xi\in \{\xi_1,\xi_2\}$ uniformly at random. Write $\rT=(T,\xi)$ and define $\rG=(G,c)=\chi(\rT)$. By Corollary~\ref{cor:uniftreelaw}, $(G,c)$ is uniformly distributed in $\triangle_n^\circ$ and $\rT' = (T',\xi',D')$ is uniformly distributed in $\cT_n^{\mathrm{vl}}$.

Again define $Y=Y_{\rT}$ as in Section~\ref{sec:labels-displacements} and again extend $Y$ to $V(G)$ by taking $Y(A)=1$ and $Y(B)=2$. We note that, since the function $\phi_n$ identifies $T'$ as a subtree of $T$, $Y(v)$ is defined for $v \in V(T')$. 

Using Corollary~\ref{cor:LabelsLMP} and Proposition~\ref{prop:windingbound}, we now show that with high probability, the labelling $Y:V(G) \to \Z^{\ge 0}$ gives distances to $A$ in $G$ up to a uniform $o(n^{1/4})$ correction. 
\begin{thm}\label{thm:distance}
For all $\eps > 0$, 
\[
\lim_{n \to \infty} \p{\exists~u \in V(G): d_G(u,A) \not\in [Y(u)-\eps n^{1/4},Y(u)-1]}=0\, .
\]
\end{thm}
The upper bound $d_G(u,A)\le Y(u)-1$ holds deterministically by Corollary~\ref{cor:LabelsLMP}. To prove the lower bound (in probability), we begin by stating a lemma whose proof, postponed to the end of the section, is based on soft convergence arguments and the continuity of the Brownian snake. 
Recall the definition of the contour exploration $(r_{\rT}(j),0 \le j \le 2n-2)$. 
Given $0 \le i \le 2n-2=2|V(T)|-2$ and $\Delta > 0$, let 
\begin{align*}
g_{\rT}(i,\Delta) &= \sup\left\{j < i: \left|Y(r_{\rT}(j))-Y(r_{\rT}(i))\right| \ge \Delta \mbox{ or } j=0\right\} \\
d_{\rT}(i,\Delta) & = \inf\left\{j > i: \left|Y(r_{\rT}(j))-Y_{\rT}(r_{\rT}(i))\right| \ge \Delta \mbox{ or } j=2n-2\right\}\, .
\end{align*}
Then let $N(i,\Delta)=\{v \in V(T): \exists~g_{\rT}(i,\delta) \le j \le d_{\rT}(i,\Delta), r_{\rT}(j)=v\}$ 
be the set of vertices of $T$ visited by the contour exploration between times 
$g_{\rT}(i,\Delta)$ and $d_{\rT}(i,\Delta)$. 
\begin{lem}\label{lem:tree_displacement_size_bound}
For all $\eps > 0$ and $\beta > 0$, there exist $\alpha > 0$ and $n_0 \in \N$ such that for $n \ge n_0$, 
\[
\p{\inf\left\{|N(i,\beta n^{1/4})|:0 \le i \le 2n-2\right\} \ge \alpha n} \ge 1-\eps\, .
\]
\end{lem}
\begin{proof}[Proof of Theorem~\ref{thm:distance}]
As mentioned, we need only prove the lower bound. It suffices to show that 
for all $\eps > 0$, 
\[
\limsup_{n \to \infty} \p{\exists~e=uv \in \overrightarrow{E}: d_G(u,A) < Y_{\rT}(u)-6(\eps n^{1/4}+2)} \le 4\eps
\] 
(we have done a little anticipatory selection of constants in the preceding formula). 
Write $\diam(G)$ for greatest distance between any two vertices of $G$. 
By Corollary~\ref{cor:LabelsLMP} , $\diam(G) \le 2\max_{u \in V(G)} (Y_{\rT}(u)-1)=2(\max_{u \in V(T)} Y_{\rT}(u)-\min_{u \in V(T)} Y_{\rT}(u))+2$, 
so by Fact~\ref{fact:rerooting}, $\diam(G) \le \max_{u \in V(T')} X_{(T',\xi',D)}(u) - \min_{u \in V(T')} X_{(T',\xi',D)}(u)+8$. 
Finally, 
\[
\max_{u \in V(T')} X_{\rT'}(u) - \min_{u \in V(T')} X_{\rT'}(u) = 
\max_{x \in [0,1]} Z_{\rT'}(x)-\min_{x \in [0,1]} Z_{\rT'}(x)\, ,
\]
and Proposition~\ref{prop:snake} implies that $(\max_{x \in [0,1]} Z_{\rT'}(x)-\min_{x \in [0,1]} Z_{\rT'}(x))n^{-1/4}$ converges in distribution as $n=|V(T')| \to \infty$, to an almost surely finite random variable. It follows that there is $y=y(\eps) > 0$ such that $\p{\mathrm{diam}(G) \ge y n^{1/4}} < \eps$. Choose such $y$, and let $B$ be the event that $G$ contains a cycle $C$ of length at most $2y/\eps$ 
such that with $V^l(C)$ and $V^r(C)$ as defined earlier, for $h \in \{l,r\}$ we have 
\[
\max_{u \in V^h(C)} Y_{\rT}(u) - \min_{u \in V^h(C)} Y_{\rT}(u) \ge \eps n^{1/4}. 
\]
Next, suppose there exists $e=uv\in \overrightarrow{E}$ for which $d_G(u,A) < Y_{\rT}(u)-6(\eps n^{1/4}+2)$.  Fix such an edge $e$, and any shortest path $Q$ from $u$ to $A$; by Proposition~\ref{prop:windingbound} we have 
$w(Q,e) \le -3\eps n^{1/4}-2$. It follows from Proposition~\ref{prop:cyclecut2}
that either $\mathrm{diam}(G) \ge y n^{1/4}$ or else $B$ occurs. 
It thus suffices to show that 
\begin{equation}\label{eq:inviewof}
\p{B,\mathrm{diam}(G) \le y n^{1/4}} \le 3\eps. 
\end{equation}
Suppose $B$ occurs, let $C$ be as in the definition of $B$, and 
let $F$ be the subgraph of $T$ induced by $V(T)\setminus V(C)$. Then $F$ is a forest, 
and each component of $F$ is contained within $G[V^l(C)]$ or $G[V^r(C)]$ since $T$ is a subgraph of $G$. Also, for $\{u,w\} \in E(G)$ we have $|Y_{\rT}(u)-Y_{\rT}(w)| \le 1$. It follows that, for $h\in \{l,r\}$, if $G[V^h(C)]$ contains $k$ components of $F$ then one such component $T_h$ must have 
\[
 \max_{u \in V(T_h)} Y_{\rT}(u) - \min_{u \in V(T_h)} Y_{\rT}(u) > \eps n^{1/4}/k-1.
\]
 But $F$ has at most $|E(C)| \le 2(|Q|-1)/(\eps n^{1/4})$ connected components. 
 When $\mathrm{diam}(G) \le y n^{1/4}$ we have $2(|Q|-1)/(\eps n^{1/4}) \le 2y/\eps$, 
 so for each $h \in \{l,r\}$, some component $T_h$ of $F$ contained in $G[V^h(C)]$ must have 
 \[
 \max_{u \in V(T_h)} Y_{\rT}(u) - \min_{u \in V(T_h)} Y_{\rT}(u) \ge \frac{\eps^2 n^{1/4}}{2y}-1. 
 \]
Using again that labels of adjacent vertices differ by at most one, if $\mathrm{diam}(G) \le y n^{1/4}$ then 
for $h \in \{l,r\}$ there is $v_h \in V^h(C)$ such that 
\[
\min_{v \in V(C)} |Y(v_h)-Y_{\rT}(v)| \ge \frac{\eps^2 n^{1/4}}{4y}-\frac{1}{2} - \frac{2y}{\eps}\, .
\]
Now for $h \in \{l,r\}$ let $j_h=j_h(\rT)=\inf\{0 \le i \le 2n-2: r_{\rT}(i)=v_h\}$. 
Also, fix any $\beta \in (0,\eps^2/2y)$. By Lemma~\ref{lem:tree_displacement_size_bound} there is $\alpha > 0$ such that for $n$ sufficiently large, 
\[
\p{\min(N(j_l,\beta n^{1/4}),N(j_r,\beta n^{1/4})) \le \alpha n} \le \eps. 
\]
For $n$ large enough that $\eps^2 n^{1/4}/(4y)-1/2 - 2y/\eps > \beta n^{1/4}$, 
for $h \in \{l,r\}$ we also have $N(j_h,\beta n^{1/4}) \subset V^h(C)$, and it follows that for $n$ sufficiently large 
\begin{align}
& \p{B,\mathrm{diam}(G) \le y n^{1/4}} \nonumber\\
\le & 2\eps + \p{\exists~C~\mbox{a cycle in}~G, |C| \le\frac{2y}{\eps},~\min(|V^l(C)|,|V^r(C)|)\ge\alpha n}\, .\label{eq:final_bound}
\end{align}  
The event in the last probability is that $G$ 
contains a separating cycle of length at most $2y/\eps$ that separates $G$ into two subtriangulations, each of size at least $\alpha n$. The number $t_{n,m}$ of simple triangulations of an $(m+2)$-gon with $n$ inner vertices has been computed in \cite{Brown64}, and has the asymptotic form $t_{n,m}\sim A_m \alpha^n n^{-5/2}$, where $A_m$ and $\alpha$ are explicit constants. (Observe that, in this notation, the number of rooted simple triangulations with $n$ vertices is equal to $t_{n-3,1}$.) For $K \in \N$ and $\alpha > 0$, denote by $\Gamma_K(\alpha)$ the event that a random simple triangulation with $n$ vertices admits a separating cycle $\gamma_n$ of length at most $K$ that separates $G_n$ into two components each of size at least $\alpha n$. Then 
\begin{equation}\label{eq:dectrig}
\p{\Gamma_K(\alpha)} \sim (t_{n-3,1})^{-1}
\sum_{k=1}^{K-2}\int_\alpha^{1-\alpha}t_{\floor{un},k}t_{\floor{(1-u)n},k}du \sim A_{K,\alpha}n^{-5/2}, 
\end{equation}
where $A_{K,\alpha}$ depends only on $\alpha$ and $K$. 
The event within the last probability in (\ref{eq:final_bound}) is contained within the event $\Gamma_{\lceil 2y/\eps\rceil}(\alpha)$, so for $n$ sufficiently large its probability is at most~$\eps$. 
In view of (\ref{eq:inviewof}), this completes the proof. 
\end{proof}
\begin{proof}[Proof of Lemma~\ref{lem:tree_displacement_size_bound}] 
Fix $\eps > 0$ and $\beta > 0$ as in the statement of the lemma. List the elements of $V(T_n')$ according to lexicographic order in $\rT_n$ as $v_n(1),\ldots,v_n(n)$, and for $1 \le m \le n$ let $i_n(m) = \inf\{i \ge 0:r_{\rT_n}(i)=v_n(m)\}$.  

By considering the height process, a straightforward argument (almost identical that given for equations (12) and (13) of \cite{le2005random}) shows that 
\[
\sup_{0 \le t \le 1} \Big|\frac{i_n(\lfloor tn \rfloor)}{2n-2}-t\Big| \convdist 0\, . 
\]
from which it follows that for any $\delta > 0$, 
\[
\p{\frac{\inf\left\{|N(i,\beta n^{1/4})|:0 \le i \le 2n-2\right\}+\delta n}{\inf\{d_{\rT_n}(i,\beta n^{1/4})-g_{\rT_n}(i,\beta n^{1/4}):0 \le i \le 2n-2\}} < \frac 1 2} \to 0. 
\]
In particular, given $\alpha > 0$, for $n$ large, if $d_{\rT_n}(i,\beta n^{1/4})-g_{\rT_n}(i,\beta n^{1/4}) > \alpha n$ for all $0 \le i \le 2n-2$ then with high probability $\inf\left\{|N(i,\beta n^{1/4})|:0 \le i \le 2n-2\right\} > \alpha n/3$. It therefore suffices to prove there exists $\alpha > 0$ such that for all $n$ sufficiently large, 
\begin{equation}\label{eq:lemmaproof_toprove}
\p{\inf\{d_{\rT_n}(i,\beta n^{1/4})-g_{\rT_n}(i,\beta n^{1/4}):0 \le i \le 2n-2\} \ge \alpha n} > 1-\eps\, .
\end{equation}
By Proposition~\ref{prop:snake} and Skorohod's embedding theorem, we now work in a space in which 
\begin{equation}\label{eq:lemmaproof_skorohod}
\left( (3n)^{-1/2}C_{\rT_n}(t), (4n/3)^{-1/4}Z_{\rT_n}(t)\right)_{0\le
t \le 1}\,
\convas \, ({\bf e}(t), Z(t))_{0\le t \le 1}\, .
\end{equation}
Let $A=A(Z)=\inf\{|x-y|: x,y \in [0,1], |Z(x)-Z(y)| > \beta/(2\cdot(4/3)^{1/4})\}$, or let $A(Z)=1$ if the set in the preceding infimum is empty. When $(d_{\rT_n}(i,\beta n^{1/4})-g_{\rT_n}(i,\beta n^{1/4}))/(2n-2)< 1$ either $d_{\rT_n}(i,\beta n^{1/4})\neq 0$ or $g_{\rT_n}(i,\beta ^{1/4})\neq 2j-2$, so either $Z_{\rT_n}(d_{\rT_n}(i,\beta n^{1/4})/(2n-2))-Z_{\rT_n}(i/(2n-2)) > \beta n^{1/4}$ or 
$Z_{\rT_n}(i/(2n-2))-Z_{\rT_n}(g_{\rT_n}(i,\beta n^{1/4})/(2n-2)) > \beta n^{1/4}$. 
By (\ref{eq:lemmaproof_skorohod}), it follows that a.s.\ 
\[
(2n-2)^{-1}\cdot \inf\{d_{\rT_n}(i,\beta n^{1/4})-g_{\rT_n}(i,\beta n^{1/4}):0 \le i \le 2n-2\} > A
\]
for all $n$ sufficiently large. Finally, since $Z$ is a.s.\ uniformly continuous on $[0,1]$, almost surely $A > 0$, and (\ref{eq:lemmaproof_toprove}) follows immediately. 
\end{proof}

\section{The proof of Theorem~\ref{thm:main} for triangulations}\label{sec:mainthm}
%
For each $n \in \N$, construct a map encoding $\rP_n=(\rM_n,\rT_n)$ as follows. Let $(T,\hat{\xi})$ be uniformly random in $\cT_n$ and let $\rT_n=(T_n,\xi_n,D_n) := \phi_n(T,\hat{\xi})$.  
Next let $\xi^1,\xi^2 \in \cC(T)$ be such that $(T,\xi^i)$ is balanced for $i\in\{1,2\}$. Conditionally given $(T,\hat{\xi})$ choose $\xi \in \{\xi^1,\xi^2\}$ uniformly at random. Then let $\rM_n = (M_n,\zeta_n) :=\chi_n(T,\xi)$. 

To prove Theorem~\ref{thm:main} for triangulations, we verify that $\rP=(\rP_n,\ n\in \N)$ is a good sequence of map encodings, with sequences $a_n = (3n)^{-1/2}$ and $b_n = (4n/3)^{-1/4}$. Assuming this, to conclude note that, by Corollary~\ref{cor:uniftreelaw}, $(M_n,\zeta_n)$ is a uniformly random element of $\triangle_{n+2}^{\circ}$. 
Since $b_{n+2}/b_n \to 1$ as $n \to \infty$, the result then follows from Theorem~\ref{prop:CSTriple}. 

In what follows, our usage of the notation $C_n$, $Z_n$, $X_n$ and $r_n$ agrees with that of Section~\ref{sec:cstriples}.
By Proposition~\ref{prop:label_bij} and~Corollaries~\ref{cor:uniftreelaw} and~\ref{cor:vldist}, $\rT_n$ has law $\LGW(\mu,\nu,n)$, where the law of $\mu$ is given by (\ref{eq:offspring}) and $\nu=(\nu_k,k \ge 1)$ is as in Corollary~\ref{cor:vldist}. 
Condition~{\bf 1.}\ then holds by Proposition~\ref{prop:snake}. Condition~{\bf 2.}(i)\ is immediate from the construction as $\rM_n$ contains only two vertices ($A$ and $B$) that are not elements of $V(T_n)$.  
\smallskip

Next, recall that $m=m(n)=2|V(T_n)|-2$. For $u,v \in V(M_n)\backslash\{A,B\}$, let $i$ and $j$ be such that $u=r_n(i)$ and $v=r_n(j)$. Then, combining Fact~\ref{fact:rerooting} and Proposition~\ref{prop:2points} (with $\zeta_u=\xi_{\rT_n}(i)$ and $\zeta_v = \xi_{\rT_n(j)}$) yields:
\[
d_{M_n}(u,v)\le Z_{n}(i/m) + Z_{n}(j/m) - 
2\max\left(\check{Z}_{n}(i/m,j/m),\check{Z}_{n}(j/m,i/m)\right)+18\, ,
\]
where the additive constant $18$ arises from the $6$ in Proposition~\ref{prop:2points}, plus four times the additive error of $3$ from Fact~\ref{fact:rerooting}.
It follows that for all $0\le i,j \le m$,
\[
d_{M_n}(r_n(i),r_n(j)) \le Z_{n}(i/m) + Z_{n}(j/m) - 
2\max\left(\check{Z}_{n}(i/m,j/m),\check{Z}_{n}(j/m,i/m)\right)+18\, ,
\]
which verifies {\bf 3.}(i). Condition~{\bf 3.}(ii) follows directly from Theorem~\ref{thm:distance}. It remains to establish {\bf 2.}(ii). 
Since $X_n( \xi_n)=0$, it follows from (\ref{eq:3cons}) that 
\[
b_n\cdot |d_{M_n}(\zeta_n,\xi_n)+\check{Z}_n(0,1)| \convdist 0\, ,
\]
so by {\bf 1.}, 
\begin{equation}\label{eq:triang_conv1}
b_n d_{M_n}(\zeta_n,\xi_n) \convdist -\check{Z}(0,1) \eqdist Z(V)-\check{Z}(0,1)\, ,
\end{equation}
where $V \eqdist \mathrm{Uniform}[0,1]$ is independent of $Z$; the last equality in distribution is from~(\ref{eq:invariance_rerooting}). 
Now let $V_n$ be a uniformly random element of $V(T_n)$. Arguing from (\ref{eq:3cons}) and {\bf 1.}\ as above, we obtain 
\begin{equation}\label{eq:triang_conv2}
b_n d_{M_n}(V_n,\zeta_n) \convdist Z(V)-\check{Z}(0,1). 
\end{equation}

Next, recall that $(M_n,\zeta_n)$ is uniformly random in $\triangle_{n+2}^{\circ}$. It follows that, conditionally given $M_n$, $\zeta_n$ is a uniformly random element of $\cC(M_n)$; let $c_n$ be another uniformly random element of $\cC(M_n)$, independent of $\zeta_n$ and of $V_n$. It follows that 
\begin{equation}\label{eq:triang_conv3}
d_{M_n}(V_n,\zeta_n) \eqdist d_{M_n}(V_n,c_n). 
\end{equation}

Let $\overrightarrow{E}$ be the minimal $3$-orientation associated to $\rM_n$. 
Writing $c_n=(\{x_n,y_n\},\{y_n,z_n\})$, let $\tilde \v(c_n)=y_n$ if $y_nz_n \in \overrightarrow{E}$, 
and $\tilde \v(c_n)=z_n$ otherwise. Note that $\tilde \v(c_n)$ is either equal to or incident to $\v(c_n)$, so $|d_{M_n}(V_n,\v(c_n))-d_{M_n}(V_n,\tilde\v(c_n))|\leq 1$. 
Further, since $c_n$ is a uniformly random corner of $M_n$, $\{y_n,z_n\}$ is a uniformly random edge of $M_n$, so for all $v \in V(M_n)$, $\p{\tilde\v(c_n)=v}$ is proportional to the outdegree of $v$ in $\overrightarrow{E}$. Since all inner vertices of $M_n$ have outdegree 3 in $\overrightarrow E$, and $c_n$ is independent of $V_n$, we may couple $\tilde\v(c_n)$ with a uniformly random element $U_n$ of $V(T_n)$, independent of $V_n$, such that $\p{U_n\neq \tilde \v(c_n)} \to 0$ as $n \to \infty$. 
Furthermore, $d_{M_n}(\v(c_n),U_n) \le 1$ on $\{U_n=\tilde\v(c_n)\}$, so 
$b_n d_{M_n}(c_n,U_n) \to 0$ 
in probability, as $n \to \infty$.
It then follows from (\ref{eq:triang_conv2}) and (\ref{eq:triang_conv3}) that 
\[
b_nd_{M_n}(V_n,U_n) \convdist Z(V)-\check{Z}(0,1)
\]

With (\ref{eq:triang_conv1}), this establishes {\bf 2.}(ii) and completes the proof. \qed

\section{The proof of Theorem~\ref{thm:main} for quadrangulations}\label{seq:quad}
The results on which the proof for simple triangulations rely all have nearly exact analogues for simple quadrangulations, which makes the proof for quadrangulations quite straightforward. In this section, we state the required results, with an emphasis on the details that differ between the two cases.

\addtocontents{toc}{\SkipTocEntry}
\subsection{Simple quadrangulations and blossoming trees}\label{sub:bijquad}
The counterpart of the bijection between simple triangulations and $2$-blossoming trees is a bijection between simple quadrangulations and 1-blossoming trees, due to Fusy \cite{FusyPhD}. In this section, by ``blossoming trees'' we mean 1-blossoming trees, and write $\cTq n$ for the set of blossoming trees with $n$ inner vertices. 
Fix a blossoming tree $T$. Given a stem $\{b,u\}$ with $b \in \cB(T)$, if $bu$ is followed by \emph{three} inner edges in a clockwise contour exploration of $T$ -- $uv$, $vw$ and $wz$, say -- then the \emph{local closure} of $\{b,u\}$ 
consists in removing the blossom $b$ (from both $V(T)$ and $\cB$) and its stem, and adding a new edge $\{u,z\}$. 

After all local closures have been performed, all unclosed blossoms are incident to a single face $f$. A simple counting argument shows that there exist exactly two edges $\{u,v\}$ and $\{x,y\}$ of $f$ such that $u,v,x$ and $y$ are each incident to one unclosed stem; between any two other consecutive unclosed stems, there are two edges of $f$. Assume by symmetry that $f$ lies to the left of both $uv$ and $xy$, and write $\xi_C=\kr v u $, $\xi_D = \kr y x$, $C=\v(\xi_C)$ and $D=\v(\xi_D)$ (see Figure~\ref{fig:part-clo-quad}). 

Given $\xi \in \cC(T)$, the planted blossoming tree $(T,\xi)$ is {\em balanced} if $\xi=\xi_C$ or $\xi=\xi_D$. Suppose $\xi \in \{\xi_C,\xi_D\}$ and write $v=\v(\xi)$. Let $S_{CD}$ (resp.\ $S_{DC}$) be the set of non-blossom vertices $u$ incident to an unclosed blossom in the partial closure, such that in the planted tree $(T,C)$ (resp.\ $(T,D)$) we have $C \pc v \pcst D$ (resp.\ $D\pc v \pcst C$).

To finish the construction, remove the remaining blossoms and their stems. 
Add two additional vertices $A$ and $B$ within the outer face, 
and an edge between $A$ (resp.~$B$) and each of the vertices of $S_{CD}$ (resp.\ of $S_{DC}$). 
In the resulting map, define a corner $c$ by $c=(\{C,B\},\{C,A\})$ if $v=C$ or $c=(\{D,A\},\{D,B\})$ if $v=D$. 
Finally, add an edge between 
$A$ and $B$ in such a way that, after its addition, 
$c$ lies on the same face as $A,B$, and $v$ (see Figure~\ref{fig:clo-quad-ori}). Write $\chi_{\ssq}(\rT)$ for the resulting map.

Fix a planted planar quadrangulation $(Q,\xi)$, and view $(Q,\xi)$ as embedded in $\R^2$ so that the face containing $\xi$ is the unique unbounded face. A \emph{$2$-orientation} of a $(Q,\xi)$ is an orientation for which $\alpha(v)=2$ for each vertex $v$ not incident to the root face and, listing the vertices of the unbounded face in clockwise order as $v,A,B,w$ with $v=\v(\xi)$, we have $\alpha(A)=0$, $\alpha(B)=\alpha(v)=1$ and $\alpha(w)=2$. Write $\overrightarrow{E}_{\ssq}$ for the resulting quadrangulation. Ossona de Mendez \cite{OssonaThesis} showed that a quadrangulation admits a 2-orientation if and only if it is simple, and in this case admits a unique {\em minimal} 2-orientation.
\begin{prop}[\cite{FusyPhD}]\label{prop:bijEric}
The closure operation $\chi_{\ssq,n}$ is a bijection between the set $\cT_{\ssq,n}^{\circ}$ of balanced 1-blossoming trees with $n$ inner vertices and the set $\square^{\circ}_{n+2}$ of planted quadrangulations with $n+2$ vertices. 
Furthermore, for $\rT\in \cT_{\ssq,n}$, $\chi_{\ssq,n}(\rT)$ is naturally endowed with its minimal 2-orientation by viewing stems of $\rT$ as oriented toward blossoms, and all other edges as oriented toward the root. 
\end{prop}

\begin{figure}[ht]
\hspace{-0.5cm}
 \subfigure[A balanced 1-blossoming
tree,]{\includegraphics[width=.3\linewidth,page=1]{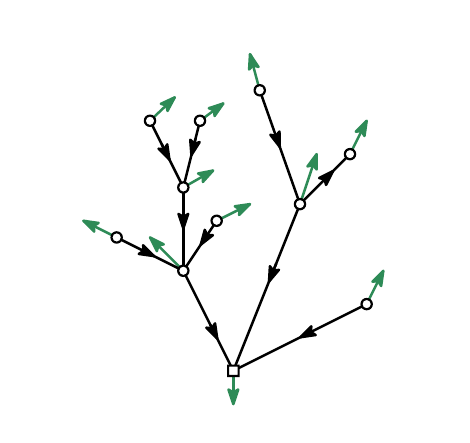}}\qquad
  \subfigure[\label{fig:part-clo-quad}its partial closure,]{\includegraphics[width=.3\linewidth,page=3]{Pictures/cloture-quadrangulation.pdf}}
\qquad
\subfigure[\label{fig:clo-quad-ori}the rooted simple quadrangulation obtained, endowed with its unique
minimal 2-orientation]{\includegraphics[width=.3\linewidth,page=4]{Pictures/cloture-quadrangulation.pdf}}
\caption{The closure of a balanced 1-blossoming tree into a simple quadrangulation.} 
\label{fig:clo-quad}
\end{figure}

\addtocontents{toc}{\SkipTocEntry}
\subsection{Sampling simple quadrangulations}\label{sub:treequad}
Given a blossoming tree $\rT_\ssq=(T,\xi)$, define $\lambda_{\ssq}:=\lambda_{\ssq,\rT_\ssq}:\cC(T) \to \Z$ as follows. Let $(\xi_{\ssq}(i),0 \le i \le 2|V(T)|-2)$ be the contour ordering from Section~\ref{sub:PlaneTrees}, with $\xi_0=\xi$. 
Let $\lambda_{\ssq}(\xi_0)=2$ and, for $0 \le i < 2|V(T)|-3$, set 
\[
\lambda_{\ssq}(\xi_{\ssq}(i+1)) = 
\begin{cases}
\lambda_{\ssq}(\xi_{\ssq}(i)) -1  & \text{if } \xi_{\ssq}(i) \not\in \cB,\  \xi_{\ssq}(i+1) \not\in \cB , \\
\lambda_{\ssq}(\xi_{\ssq}(i))  & \text{if } \xi_{\ssq}(i) \not\in \cB,\ \xi_{\ssq}(i+1) \in \cB , \\
\lambda_{\ssq}(\xi_{\ssq}(i)) +2  & \text{if } \xi_{\ssq}(i) \in \cB,\ \xi_{\ssq}(i+1) \not\in \cB , 
\end{cases}
\]
As opposed to Section~\ref{sec:bij_with_labels}, here the label increases by 2 after each stem.  

It is not hard to see that $\rT_\ssq=(T,\xi)$ is balanced if and only if $\xi$ is incident to one stem and $\lambda_{\ssq}(c) \ge 2$ for all $c \in \cC(T)$. With the same definition of successors for corners, and the same construction as in Section~\ref{sec:bij_with_labels}, this labelling yields another description of the bijection from~Section~\ref{sub:bijquad}.
Let $\cT^{\mathrm{vl}}_{\ssq,n}$ be the set of triples $(T,\xi',d)$ where $(T,\xi')$ is a planted plane tree and $d=(d_e, e \in E(T))$ is a $\pm 1$ labeling of $E(T)$ such that for all $v \in V(T)$, listing the edges from $v$ to its children in lexicographic order as $e_1,\ldots,e_k$, the sequence $d_{e_1},\ldots,d_{e_k}$ is non-decreasing.

Let $X_{\ssq} \eqdist \mathrm{Geometric}(2/3)$, and let $B_{\ssq}$ have law given by 
\begin{equation}\label{eq:offspring_q}
\p{B_{\ssq}=k} = \frac{(k+1)\p{X_\ssq=k}}{\e{(X_\ssq+1)}}\,, \mbox{ for } k \in \N.
\end{equation}
The following 
is the analogue of Corollary~\ref{cor:vldist} for quadrangulations. 
\begin{prop}
Let $(T',\xi')$ be a Galton--Watson tree with offspring distribution $B_{\ssq}$ conditioned to have $n$ vertices. Conditionally given $(T',\xi')$, independently for each $v \in V(T')$, list the children of $v$ in clockwise order as $v1,\ldots,vk$ and let $(D_{\{v,vj\}},1 \le j \le k$ be a random vector with law $\nu_{\ssq,k}$, where $\nu_{\ssq,k}$ is the uniform law over non-decreasing vectors $(d_1,\ldots,d_k) \in \{-1,1\}^k$. Finally, let $D=(D_e,e \in E(T'))$. 
Then $(T',\xi',D)$ is uniformly distributed in $\cT_{\ssq,n}^{\mathrm{vl}}$ and the closure $\chi_{\ssq,n}(T',\xi',D)$ is uniformly distributed in $\square_{n+2}^{\circ}$.
\end{prop}
The proof of Proposition~\ref{prop:snake} extends immediately to this setting and we obtain the following convergence (see Appendix~\ref{sec:notes} for the computation of the constants). 
\begin{prop}\label{prop:snakequad}
For $n \in \N$ let $\rT_n=(T_n,\xi_n,D_n)$ be a uniformly random element of $\cT_{\ssq,n}^{\mathrm{vl}}$.   
Then as $n \to \infty$, 
\begin{equation}\label{eq:cvsnake_q}
\left( \frac{3}{4n^{1/2}}C_{\rT_n}(t), \left(\frac{3}{8n}\right)^{1/4}Z_{\rT_n}(t)\right)_{0 \le t \le 1}
\,
\convdist \, ({\bf e}(t), Z(t))_{0\le t \le 1},
\end{equation}
for the topology of uniform convergence on $C([0,1],\R)^2$.
\end{prop}

\addtocontents{toc}{\SkipTocEntry}
\subsection{Labels and distance in simple quadrangulations}\label{sub:disquad}
We next state analogues of the results of Sections~\ref{sec:label_dist_determinist} and~\ref{sec:label_approx} for quadrangulations. Fix $n \in \N$ and $(T,\xi) \in \cT_{\ssq,n}^{\circ}$, let $(Q,c)=\chi_{\ssq,n}(T,\xi)$ be endowed with its minimal 2-orientation $\overrightarrow{E}_{\ssq}$ and let $(T',\xi',D)\in \cT^{\mathrm{vl}}_{\ssq,n}$ be the validly-labelled tree associated to $(T,\xi)$. Finally, write $\rQ=(Q,c)$ and $\rT_\ssq=(T,\xi)$. 

The definition of leftmost paths for simple quadrangulations is an obvious modification of that for triangulations. Together with the fact that (with $Y_T$ defined as before) for $\{u,w\}\in E(Q)$, $|Y_T(u)-Y_T(w)|\leq 3$, we obtain the following facts. The lemma is a counterpart of Lemma~\ref{lem:3diff} and Corollary~\ref{cor:LabelsLMP}; the proposition is a counterpart of Proposition~\ref{prop:2points}, and uses an identical definition for $\check{Y}_{\rT_{\ssq}}(u,v)$. 
\begin{lem}\label{lem:LabelsLMPquad}
For all $u \in V(Q)$, $Y_{\rT_\ssq}(u)/3\le d_Q(u,A)\le Y_{\rT_\ssq}(u)-1$.
\end{lem}
\begin{prop}\label{prop:2pointsquad}
For all $u,v \in V(Q)$, 
\[
d_Q(u,v)\leq Y_{\rT_{\ssq}}(u)+Y_{\rT_{\ssq}}(v)-2\max\{\check Y_{\rT_{\ssq}}(u,v), \check Y_{\rT_{\ssq}}(v,u)\}+2. 
\]
\end{prop}
The winding number introduced in Definition~\ref{dfn:windingnumber} is used in the following analogue of Proposition~\ref{prop:windingbound}. 
\begin{prop}\label{prop:windingboundquad}
For all $e=uw \in \overrightarrow{E}_{\ssq}$, if $Q$ is a simple path from $e$ to $A$ then $|Q| \ge |P(e)|+2(w(Q,e)-1)$. 
\end{prop}
\begin{proof}
The proof of Proposition~\ref{prop:windingbound} extends readily to the case of quadrangulations. Keeping the same notation, the following inequalities (whose proofs are left to the reader) allow one to conclude along the same lines. 
\begin{enumerate}
\item If $R$ leaves $P(e)$ from the right and returns from the left then $k \ge j-i-2$.
\item If $R$ leaves $P(e)$ from the left and returns from the left then $k \ge j-i$.
\item If $R$ leaves $P(e)$ from the left and returns from the right then $k \ge j-i+2(\I{i>0}+\I{j<\ell})$. 
\item If $R$ leaves $P(e)$ from the right and returns from the right then $k \ge j-i+2\I{j<\ell}$. 
\qedhere
\end{enumerate}
\end{proof}
Combining Lemma~\ref{lem:LabelsLMPquad} and Proposition~\ref{prop:windingboundquad}, we obtain that with probability tending to one, distances to $A$ in $Q$ are given by labels in $\rT_{\ssq}$ up to a $o(n^{1/4})$ perturbation. 
\begin{thm}\label{thm:distancequad}
For all $\eps > 0$, 
\[
\lim_{n \to \infty} \p{\exists~u \in V(Q): d_Q(u,A) \not\in [Y_{\rT_\ssq}(u)-\eps n^{1/4},Y_{\rT_\ssq}(u)-1]}=0\, .
\]
\end{thm}
\begin{proof}
The only element of the proof of Theorem~\ref{thm:distance} that cannot be directly applied here is the approximation of $\p{\Gamma_K(\alpha)}$ given in~(\ref{eq:dectriq}) that relies on the number $t_{n,m}$ of simple triangulations of an $(m+2)$-gon. 
This has an easy fix: for $\alpha>0$, write $\Gamma_{\ssq,K}(\alpha)$ for the event that a uniformly random simple quadrangulation $Q_n$ with $n$ faces admits a separating cycle of length at most $K$, separating $Q_n$ into two components each of size at least $\alpha n$. An explicit expression for the number $q_{n,m}$ of simple quadrangulations of a $2m$-gon with $n$ inner vertices is derived in \cite{Brownquad}, and has the asymptotic form $q_{n,m}\sim A_m \alpha^n n^{-5/2}$, where $A_m$ and $\alpha$ are explicit constants. (Observe that, in this notation, the number of rooted simple quadrangulations with $n$ vertices is equal to $q_{n-4,2}$.). Then 
\begin{equation}\label{eq:dectriq}
\p{\Gamma_K(\alpha)} \sim (q_{n-4,2})^{-1}
\sum_{k=0}^{K}\int_\alpha^{1-\alpha}t_{\floor{un},k}t_{\floor{(1-u)n},k}du \sim A_{\ssq,K,\alpha}n^{-5/2}, 
\end{equation}
where $A_{\ssq,K,\alpha}$ depends only on $\alpha$ and $K$. 
\end{proof}

\section{Acknowledgements}
LAB thanks Gr\'egory Miermont, Gilles Schaeffer, and Jean-Fran\c{c}ois Le Gall for enlightening discussions before and during the preparation of the manuscript. LAB received support from NSERC and FQRNT while conducting this research, and thanks both. LAB also thanks Gilles Schaeffer and CNRS UMI 3457 for facilitating his visit to the Laboratoire d'Informatique de l'\'Ecole Polytechnique in May 2012, during which the research progressed substantially. 

MA thanks \'Eric Fusy, Dominique Poulalhon and Gilles Schaeffer for fruitful and stimulating discussions. 
MA acknowledges the support of the ERC under the agreement ``ERC StG 208471 - ExploreMap'' and of the ANR under the agreement ``ANR 12-JS02-001-01 - Cartaplus''. MA also thanks Bruce Reed, CNRS UMI 3457 and LIRCO for facilitating her visits to the mathematics department of McGill University, which played a major role in this work.

Both LAB and MA thank an anonymous referee, whose detailed comments and suggestions  greatly improved the paper.

\appendix

\section{Notes about constants}\label{sec:notes}
In this section we briefly derive the forms of the constant coefficients arising in Theorem~\ref{thm:main} and Proposition~\ref{prop:snake}. 

For simple triangulations, we work with a critical Galton--Watson tree with a offspring distribution $B$ uniquely specified by the following facts.
\begin{enumerate}
\item Criticality: $\e{B}=1$. 
\item There exists $p \in (0,1)$ such that if $G$ is Geometric$(p)$ then the law of $B$ is given by setting, for each $k \in \N$, 
\[
\p{B=k} = \frac{{k+2 \choose 2} \p{G=k}}{\e{{G+2 \choose 2}}}\, .
\]
\end{enumerate}
From these conditions, a straightforward calculation shows that $p=3/4$, and another easy computation yields that $\E{B^2} = 7/3$ so $\V{B}=4/3$. In the notation of Section~\ref{sec:snake}, this yields $\sigma_{\mu}=2/3^{1/2}$.

Next, the displacement $D$ between a node in our tree and a uniformly selected child is equal to one of $\{-1,0,1\}$, each with equal probability; it follows that $\E{D^2}=2/3$. 
Let $\nu_k$ be the law of the displacement vector for a vertex with $k$ children, then let $\hat\nu_k$ be its symmetrization (as defined in Section~\ref{sec:snake}), and write $\hat\nu_k^i$ for the $i$'th marginal of $\hat\nu_k$. It follows that 
$\sigma^2_{\hat\nu_k^i}=2/3$ for all $1 \le i \le k$, so $\sigma_{\hat\nu}^2 =\sigma_{\nu}^2=2/3$ and 
$(\sigma_{\mu}/2)^{1/2}/\sigma_{\nu}=(3/4)^{1/4}$. Together with Theorem~\ref{prop:CSTriple}, this explains the values of constants relating to triangulations. 

We remark that the scaling required for convergence of triangulations in Theorem~\ref{thm:main} agrees with the intuition described in \cite{bouttier10distance}, Section~4.1.
It differs by a factor $8^{1/4}$ from the scaling for general triangulations that arises in Theorem~1.1 of \cite{jf}, which can be understood as follows. First, in \cite{jf}, the index $n$ denotes the number of faces rather than the number of vertices, which accounts for a factor $2^{1/4}$. The size of the simple core of a loopless triangulation with $m$ vertices is typically $\sim m/2$ (see Table~4 of \cite{banderier2001maps}); this explains another factor $2^{1/4}$. Finally, the loopless core of a simple triangulation with $m$ vertices is again typically of order $\sim m/2$ (this is not proved in \cite{banderier2001maps} but may be handled using the same technology); this explains the final factor $2^{1/4}$. 
The latter factor does not arise in considering quadrangulations, which can not contain loops; this may be viewed as explaining the different form of the constant for triangulations versus those of bipartite maps in Theorem~1.1 of \cite{jf}.

For simple quadrangulations, we work with a critical Galton--Watson tree with offspring distribution $B$ uniquely specified by the following facts.
\begin{enumerate}
\item Criticality: $\e{B}=1$. 
\item There exists $p \in (0,1)$ such that if $G$ is Geometric$(p)$ then the law of $B$ is given by setting, for each $c \in \N$, 
\[
\p{B=c} = \frac{(c+1) \p{G=c}}{\E{G+1}}\, .
\]
\end{enumerate}
From these calculations, a straightforward calculation shows that $p=2/3$, and another easy computation then yields that $\E{B^2}=\frac{5}{2}$, so $\V{B}=3/2$.  Next, the displacement $D$ between a node $v$ and a uniformly selected child is equal to $-1$ or to $1$, each with equal probability, so has $\E{D}=0$ and $\V{D}=1$. 
Using Theorem~\ref{prop:CSTriple} as above then yields the scaling for quadrangulations in Theorem~\ref{thm:main}, and agrees with the two-point calculation for simple quadrangulations by Bouttier and Guitter \cite{bouttier10distance}. 


\section{Convergence for good sequences of map encodings}\label{sec:cstripleproof}
Throughout the section we let $\rP=(\rP_n, n \ge 0)$ be a good sequence of random map encodings. We write $\rP_n=(\rM_n,\rT_n)$, $\rM_n=(M_n,\zeta_n)$ and $\rT_n=(T_n,\xi_n)$, and let $C_n,Z_n,r_n$ and $X_n$ be as in Section~\ref{sec:cstriples}.

Next, for $n \ge 1$, list the the vertices of $T_n$ according to their lexicographic order as $v_{n}(1),\ldots,v_{n}(|T_n|)$. Given $1 \le j \le |V(T_n)|$, let $i_n(j) = \inf\{i: r_n(i)=v_n(j)\}$ be the index at which $v_n(j)$ first appears in the contour exploration. Let $m=m_n=2|V(T_n)|-2$
\begin{lem}\label{lem:cscontour}
As $n \to \infty$, we have 
\[
\sup_{0 \le t \le 1} \left|\frac{i_n(\lfloor t\cdot |V(T_n)|\rfloor)}{m_n} - t\right| \convdist 0\, . 
\]
\end{lem}
\begin{proof}
As in Lemma~\ref{lem:tree_displacement_size_bound}, a straightforward argument using the height process (following (12) and (13) of \cite{le2005random}) shows that when $m_n \ge 2$,  deterministically 
\[
\sup_{0 \le t \le 1} \left|\frac{i_n(\lfloor t\cdot |V(T_n)|\rfloor)}{m_n} - t\right|
\le 
\frac{\max\{|v|, v \in V(T_n)\}+4}{m_n}\, .
\]
Since $m_n \to \infty$ it thus suffices to show that $(\max\{|v|, v \in V(T_n)\}+4)/m_n \convdist 0$. To see this, let $U$ and $V$ be independent Uniform$[0,1]$ random variables.
If the latter convergence fails to hold then for infinitely many $n$, with uniformly positive probability, a single path from the root in $\rT_n$ contains a macroscopic proportion of the elements of the vertices of $T_n$. It follows easily that 
\[
\limsup_{n \to \infty} \p{U < V,C_{n}(U) = \min_{U \le x \le V} C_{n}(x)} > 0\, .
\]
On the other hand, $\p{U < V,\be(U) = \min_{U \le x \le V} \be(x)} = 0$, 
so the preceding equation implies that $\be$ is not the distributional limit of any rescaling of $C_{n}$. Thus {\bf 1.} does not hold, a contradiction. 
\end{proof}

\begin{proof}[Proof of Theorem~\ref{prop:CSTriple}]
We claim that it suffices to establish 
\begin{equation}\label{eq:rhomu}
(V(T_n),b_nd_{M_n},\mu_n) \convdist (S,d,\mu). 
\end{equation}
for $\dghp$. 
(In the above, by $d_{M_n}$ we really mean the distance on $V(T_n)$ induced by $d_{M_n}$. This slight notational abuse should cause no confusion.) 
Indeed, suppose the latter convergence holds. By Skorohod's representation theorem, we may work in a space in which the convergence (\ref{eq:rhomu}) is almost sure. Fix $\eps > 0$, and let 
$E_n$ be the event that $\max_{v \in V(M_n)} b_n\cdot d_{M_n}(v,V(T_n)) \le \eps/2$ 
and $\dgh^2(V(T_n),b_nd_{M_n},\v(\xi_n),\v(\zeta_n)),(S,d,\rho,u^\star)) \le \eps/2$. Now let $\cR_n^0 = \{(x,y) \in V(T_n) \times V(M_n); b_n d_{M_n}(x,y) \le \eps/2\}$; then $\cR_n^0$ has distortion at most~$\eps$. Furthermore, $(\v(\zeta_n),\v(\zeta_n)) \in \cR_n^0$ and $(\v(\xi_n),\v(\xi_n) \in \cR_n^0$). Let $\nu_n$ be the probability measure on $V(T_n) \times V(M_n)$ whose restriction to $\{(v,v): v \in V(T_n)\}$ is the uniform probability measure. Then $\nu_n$ is a coupling of $\mu_n$ (as a measure on $V(T_n)$) and $\mu_n$ (as a measure on $V(M_n)$), and $\nu_n(\cR_n^0)=1$. On $E_n$ we have that $\cR_n^0$ is a correspondence, so on $E_n$, \[
\dghp((V(M_n),b_nd_{M_n},\mu_n),(V(T_n),b_nd_{M_n},\mu_n)) \le 
\eps/2,
\]
and by the triangle inequality it follows that on $E_n$, 
\[
\dghp((V(M_n),b_nd_{M_n},\mu_n),(S,d,\mu)) \le \eps\, .
\]
Finally, in this space, since $\rP$ is good sequence of map encodings, $\p{E_n} \to 1$ as $n \to \infty$, and it follows that 
$(V(M_n),b_nd_{M_n},\mu_n) \convdist (S,d,\mu)$ for $\dghp$. We thus turn our attention to proving (\ref{eq:rhomu}). 
\medskip

The first part of our argument is based on that of \cite{legall07topological}, Proposition 3.2; the second part follows closely the argument of Section 8.3 of \cite{jf}. 
Define a function $d_n:[0,1]^2 \to [0,\infty)$ as follows. 
Define as above $m=m_n=2|V(T_n)|-2$, and for $i,j \in [m]$, let 
$d_n(i/m,j/m)=d_{M_n}(r_n(i),r_n(j))$. 
Then extend $d_n$ to $[0,1]^2$ by ``bilinear interpolation'': 
if $(x,y) = ((i+\alpha)/m,(j+\beta)/m)$ for $0 \le \alpha < 1$, $0 \le \beta < 1$ then let 
\begin{align*}
d_n(x,y) & = \alpha \beta d_n((i+1)/m,(j+1)/m) + 
\alpha (1-\beta) d_n((i+1)/m,j/m) \\
& + (1-\alpha) \beta d_n(i/m,(j+1)/m) + 
(1-\alpha)(1- \beta) d_n(i/m,j/m)\, .
\end{align*}
Using {\bf 1.}, we now work in a space in which 
\begin{equation}\label{eq:joint_conv_1}
\pran{a_nC_{n},b_nZ_{n}} \convas (\be,Z). 
\end{equation}
We will show that in such a space, additionally 
\begin{equation}\label{eq:joint_conv_toprove}
b_nd_n \convas d^* \,, 
\end{equation}
for the topology of uniform convergence on $C([0,1]^2)$. 
Assume (\ref{eq:joint_conv_toprove}) holds, and for $n \in \N$, consider the correspondence $\cR_n$ between $(S,d)$ and $(V(T_n),b_nd_{M_n})$ given by letting $[s] \in [0,1]/\{d^*=0\}=S$ correspond to $r_n(i)$ if and only if $\lceil s\cdot m \rceil=i$, for $0 \le i \le m$. \footnote{A similar technique is used at the end of Section~8 of \cite{jf}.} 
By (\ref{eq:joint_conv_toprove}), the distortion of $\cR_n$ tends to zero. 

Let $\mu_n^-$ be the uniform probability measure on $V(T_n) \setminus \{\v(\xi_n)\}$. Define a coupling between $\mu_n^-$ and $\mu$ as follows. Fix $s \in [0,1]$. Let $f_1(s) = [s] \in S$. If $s=i/m$ then let $f_2(s)=r_n(i)$. If $s \in (i/m,(i+1)/m)$ and 
$\{r_n(i),r_n(i+1)\}=\{u,p(u)\} \in E(T_n)$ then let $f_2(s)=u$. Finally, let $f=(f_1,f_2):[0,1] \to S \times V(T_n)$, let $\lambda$ denote one-dimensional Lebesgue measure on $[0,1]$, and let $\nu=f_*\lambda$. Write $\pi$ and $\pi'$ for the projection maps from $S\times V(T_n)$ to $S$ and to $V(T_n)$, respectively. We clearly have $\pi_* \nu= \mu$. Also, for each edge $e \in E(T)$, there are precisely two indices 
$i,j \in \{0,1,\ldots,m_n\}$ for which $\{r_n(i),r_n(i+1)\}=\{u,p(u)\}$; it follows that $\pi'_*\nu=\mu_n^-$. 

For any pair $([s],r_n(i)) \in \cR_n$, either $f_2(s)=r_n(i)$ or $f_2(s)=p(r_n(i))$, the two possibilities due to the two directions in which the edge $\{p(r_n(i)),r_n(i)\}$ is traversed during the contour exploration. We thus let 
\[
\cR_n^+ = \{([s],w): ([s],w) \in \cR_n\mbox{ or }([s],p(w)) \in \cR_n\}\, .
\]
Since $\cR_n$ was a correspondence, $\cR_n^+$ is again a correspondence, and $\nu(\cR_n^+)=1$. Finally, $\dis(\cR_n^+) \le \dis(\cR_n)+2b_n$ so $\dis(\cR_n^+) \to 0$ as $n \to \infty$. It follows by definition that 
\[
(V(M_n),b_nd_{M_n},\mu_n^-) \convdist (S,d,\mu)
\]
for $\dghp$. Finally, the Prokhorov distance between $\mu_n^-$ and $\mu_n$ is $1/|V(T_n)|$, which tends to zero as $n \to \infty$. We may therefore replace $\mu_n^-$ by $\mu_n$ and the preceding convergence still holds, which establishes (\ref{eq:rhomu}) and so proves the theorem. It thus remains to prove  (\ref{eq:joint_conv_toprove}).

Define a function $d_n^{\circ}:[0,1]^2 \to [0,\infty)$ as follows: for $x,y \in \{i/m,0 \le i \le m\}$, let 
\[
d_n^{\circ}(x,y) = 
Z_{n}(x)+ Z_{n}(y) - 
2\max\left(\check{Z}_{n}(x,y),\check{Z}_{n}(y,x)\right)\, .
\]
Then extend $d_n^{\circ}$ to $[0,1]^2$ by bilinear interpolation as with $d_n$.
Recalling that for integer $0\le i\le m$, $Z_{n}(i/m) = X_n(r_n(i))$, it follows straightforwardly from {\bf 1.}\ that for all $\eps, \delta > 0$, 
\[
\limsup_{n \to \infty} \p{\sup_{|x-y| \le \delta} b_n d_n^{\circ}(x,y) \ge \eps } 
\le \p{\sup_{|x-y| \le \delta} (Z(x) + Z(y) - 2\max(\check{Z}(x,y),\check{Z}(y,x))) \ge \eps}\, 
\]
(the derivation of this inequality is spelled out in a little more detail in \cite{legall07topological}, Section 3). Since $Z$ is almost surely continuous, it follows that for any $\eta > 0$ and $k \in \N$, there exist $\delta_k > 0$ and $n_k \in \N$ such that for all $n \ge n_k$, 
\[
\p{\sup_{|x-y| \le \delta_k} b_n d_n^{\circ}(x,y) \ge 2^{-(k+1)}} \le \frac{\eta}{2^{k+1}}\, .
\]
Next, by {\bf 3}(i), after increasing $n_k$ if necessary, for $n \ge n_k$, 
\begin{equation}\label{eq:eps-bound}
\p{\sup_{x,y \in [0,1]} b_n(d_n(x,y) - d_n^{\circ}(x,y)) \ge 2^{-(k+1)}} \le \frac{\eta}{2^{k+1}}\, .
\end{equation}
By decreasing $\delta_k$ if necessary, we may also ensure that for $n < n_k$, 
\[
\p{\sup_{|x-y| \le \delta_k} b_n\max(d_n(x,y),d_n^{\circ}(x,y)) \le 2^{-(k+1)}} = 1\, .
\]
Combining the three preceding estimates yields that for all $n \ge 1$, 
\[
\p{\sup_{|x-y| \le \delta_k} b_n d_n(x,y) \ge 2^{-k}} \le \frac{\eta}{2^k}\, ,
\]
so for all $n$,  
\[
\p{\forall k,\sup_{|x-y| \le \delta_k} b_nd_n(x,y) < 2^{-k}} \ge 1-\eta\, .
\]
In other words, with $(\delta_k)_{k\geq 0}$ as above, for all $n$, with probability at least $1-\eta$ the function
$b_nd_n$ belongs to the compact 
\[
K = \{ f \in C([0,1]^2,\R): f(0,0)=0, \forall k,\sup_{|x-y| \le \delta_k} f(x,y)\le 2^{-k}\}\, ,
\]
which implies that $\{b_nd_n,n \in \N\}$ is tight in $C([0,1]^2,\R)$\, .
For the remainder of the proof, we let $\tilde{d} \in C([0,1]^2,\R)$ be any almost sure subsequential limit of $b_nd_n$; we suppress the subsequence from the notation for readability. 

Recall that we work in a space where (\ref{eq:joint_conv_1}) holds. 
In such a space, it follows from the continuity of $Z$ that 
$b_nd_n^{\circ} \convas d_Z$, where $d_Z:[0,1]^2 \to \R$ is as defined in Section~\ref{sec:limitobject}. 
By (\ref{eq:eps-bound}), it follows that for any $\eta > 0$ and $p \ge 1$, 
\[
\limsup_{n \to \infty} 
\p{\sup_{x,y \in [0,1]} (b_nd_n(x,y) - d_Z(x,y)) \ge 2^{-p}} \le\eta2^{-p}\, ,
\]
so a.s.\ $\tilde{d} \le d_Z$. 

We next claim that a.s.\ $\tilde{d}(x,y)=0$ for all $x \ne y \in [0,1]$ for which $x \sim_{\be} y$. 
To see this, suppose that $x \sim_{\be} y$ for some $x,y \in [0,1]$, and 
assume by symmetry that $x < y$. Continuing to write $m=m_n=2|V(T_n)|-2$, 
(\ref{eq:joint_conv_1}) implies there exist random integer sequences $(x_n,n \in \N)$ and $(y_n,n \in \N)$ such that $x_n/m_n \convas x$, $y_n/m_n \convas y$, and 
\[
C_{n}(x_n/m_n) = C_{n}(y_n/m_n) = \min\{C_{n}(z): x_n \le m_n \cdot z \le y_n\}\, . 
\]
It follows that $r_n(x_n)=r_n(y_n)$,or equivalently that $i_n(x_n)=i_n(y_n)$,
so 
\[
d_n(x_n/m_n,y_n/m_n) = d_{M_n}(r_n(i_n(x_n)),r_n(i_n(y_n)))=0\, .
\]
Since $b_nd_n \convas \tilde{d}$ (along a subsequence) and $x_n/m_n \convas x$, $y_n/m_n \convas y$, this implies that 
\[
0= d_n(x_n/m_n,y_n/m_n)\convas \tilde{d}(x,y)\, , 
\]
so $\tilde{d}(x,y) \aseq 0$ as claimed. 

Since, almost surely, $\tilde{d} =0$ on $\{\{x,y\}: x \sim_{\be} y\}$, and $\tilde{d} \le d_Z$, we must have $\tilde{d} \le d^*$ since $d^*$ is the largest pseudo-metric on $[0,1]$ satisfying these constraints. We now show that in fact, almost surely $\tilde{d}=d^*$. 

Let $U,V$ be independent and uniform on $[0,1]$. Letting $I_n$ be minimal such that $Z_{n}(I_n/m) = \check{Z}_n(0,1)$, then by {\bf 1.}\ we have 
\[
b_n X_{n}(r_n(I_n)) = b_n \check{Z}_{n}(0,1) \convdist 
\check{Z}(0,1) \eqdist -d^*(U,V)\, ,
\]
the last identity by (\ref{eq:invariance_rerooting}) (which is Corollary~7.3 of \cite{jf}). 
Since $\v(\xi_n)=r_n(0)$
and $X_n(r_n(0))=0$, by (\ref{eq:3cons}) we also have 
\[
\lim_{n \to \infty} \p{b_n\cdot|d_{M_n}(\zeta_n,\xi_n) + X_n(r_n(I_n))|> \eps} = 0\, ,
\]
so since $X_n(r_n(I_n))=\check{Z}_{n}(0,1)$, we obtain 
\[
b_n d_{M_n}(\zeta_n,\xi_n) \convdist d^*(U,V)\, .
\]
Since the Prokhorov metric topologizes weak convergence, 
by {\bf 2.}(ii) it follows that for $U_n$ and $V_n$ two independent random elements of $R_n$, then
\[
b_nd_{M_n}(U_n,V_n) \convdist d^*(U,V)\, .
\]
Now let $1\leq J_n,K_n\leq |v(T_n)|$ be such that $v_n(J_n)=U_n$ and $v_n(K_n)=V_n$. 
The preceding convergence implies $b_nd_n(J_n,K_n) \convdist d^*(U,V)$. 
Lemma~\ref{lem:cscontour} implies that $(J_n,K_n) \convdist (U,V)$, 
so the tightness of the collection $(b_nd_{n},n \ge 1)$ then yields 
\[
b_nd_n(U,V) \convdist d^*(U,V)\, .
\]
Finally, along the subsequence where $b_nd_n \convas \tilde{d}$, we also have 
$b_nd_n(U,V) \convdist \tilde{d}(U,V)$, 
so it must be that $\tilde{d}(U,V) \eqdist d^{*}(U,V)$. Since a.s.\ 
$\tilde{d} \le d^{*}$, it must therefore hold that $\tilde{d} \aseq d^{*}$. 

We have now shown that in the space where (\ref{eq:joint_conv_1}) holds, any subsequential limit $\tilde{d}$ of $b_nd_n$ must satisfy $\tilde{d} \aseq d^*$; this implies that in fact, in this space we have $b_nd_n \convas \tilde{d}$, which establishes (\ref{eq:joint_conv_toprove}) and so completes the proof. 
\end{proof}

\small
\addtocontents{toc}{\SkipTocEntry}
\printnomenclature[3.1cm]
\normalsize


\small
\begin{thebibliography}{34}
\providecommand{\natexlab}[1]{#1}
\providecommand{\url}[1]{\texttt{#1}}
\expandafter\ifx\csname urlstyle\endcsname\relax
  \providecommand{\doi}[1]{doi: #1}\else
  \providecommand{\doi}{doi: \begingroup \urlstyle{rm}\Url}\fi

\bibitem{cartessimples}
Marie Albenque, Olivier Bernardi, Gwendal Collet, and Fusy \'Eric.
\newblock Convergence of simple maps to the {B}rownian map.
\newblock In preparation.

\bibitem{albenque2013generic}
Marie Albenque and Dominique Poulalhon.
\newblock A generic method for bijections between blossoming trees and planar
  maps.
\newblock {\em Electron. J. Comb.}, 22(2), 2015.

\bibitem{AldCRT2}
David Aldous.
\newblock The continuum random tree. {II}. {A}n overview.
\newblock In {\em Stochastic analysis (Durham, 1990)}, volume 167 of {\em
  London Math. Soc. Lecture Note Ser.}, pages 23--70. Cambridge Univ. Press,
  Cambridge, 1991.

\bibitem{MR1465433}
Jan Ambj{\o}rn, Bergfinnur Durhuus, and Thordur Jonsson.
\newblock {\em Quantum geometry}.
\newblock Cambridge Monographs on Mathematical Physics. Cambridge Univ. Press,
  1997.

\bibitem{banderier2001maps}
Cyril Banderier, Philippe Flajolet, Gilles Schaeffer, and Mich{\`e}le Soria.
\newblock Random maps, coalescing saddles, singularity analysis, and {A}iry
  phenomena.
\newblock {\em Random Structures Algorithms}, 19(3-4):194--246, 2001.

\bibitem{beltran13quad}
Johel Beltran, Jean-Fran{\c{c}}ois Le~Gall, et~al.
\newblock Quadrangulations with no pendant vertices.
\newblock {\em Bernoulli}, 19(4):1150--1175, 2013.

\bibitem{bernardi2012bij}
Olivier Bernardi and {{\'E}}ric Fusy.
\newblock A bijection for triangulations, quadrangulations, pentagulations,
  etc.
\newblock {\em J. Combin. Theory Ser. A}, 119(1):218--244, 2012.

\bibitem{billingsley}
P.~Billingsley.
\newblock {\em Convergence of probability measures}, volume 316.
\newblock Wiley-Interscience, 1999.

\bibitem{bouttier10distance}
J\'er\'emie Bouttier and Emmanuel Guitter.
\newblock Distance statistics in quadrangulations with no multiple edges and
  the geometry of minbus.
\newblock {\em J. Phys. A}, 43(20):205207, 31, 2010.

\bibitem{Brown64}
William~G Brown.
\newblock Enumeration of triangulations of the disk.
\newblock {\em Proc. London Math. Soc}, 14(3):746--768, 1964.

\bibitem{Brownquad}
William~G Brown.
\newblock Enumeration of quadrangular dissections of the disk.
\newblock {\em Canad. J. Math}, 17(3):302--317, 1965.

\bibitem{bbi}
Dmitri Burago, Yuri Burago, and Sergei Ivanov.
\newblock {\em A course in metric geometry}, volume~33 of {\em Graduate Studies
  in Mathematics}.
\newblock American Mathematical Society, 2001.

\bibitem{duplantier2011liouville}
Bertrand Duplantier and Scott Sheffield.
\newblock Liouville quantum gravity and {KPZ}.
\newblock {\em Invent. Math.}, 185:333--393, 2011.

\bibitem{EvansWinter}
Steven~N Evans and Anita Winter.
\newblock Subtree prune and regraft: a reversible real tree-valued markov
  process.
\newblock {\em Ann. Probab.}, 34(3):918--961, 2006.

\bibitem{FusyPhD}
{\'E}ric Fusy.
\newblock {\em Combinatoire des cartes planaires et applications
  algorithmiques}.
\newblock PhD thesis, {\'E}cole Polytechnique, 2010.

\bibitem{FusyPoulalhonSchaeffer}
Eric Fusy, Dominique Poulalhon, and Gilles Schaeffer.
\newblock Dissections and trees, with applications to optimal mesh encoding and
  to random sampling.
\newblock In {\em Proceedings of the sixteenth annual ACM-SIAM symposium on
  Discrete algorithms}, pages 690--699. Society for Industrial and Applied
  Mathematics, 2005.

\bibitem{garban2012bourbaki}
Christophe Garban.
\newblock Quantum gravity and the {KPZ} formula.
\newblock In {\em S\'eminaire {B}ourbaki}, volume~64, 2011-2012.

\bibitem{JanMar}
Svante Janson and Jean-Fran{\c{c}}ois Marckert.
\newblock Convergence of discrete snakes.
\newblock {\em J. Theoret. Probab.}, 18(3):615--645, 2005.

\bibitem{LeGallSnake}
Jean-Fran{\c{c}}ois Le~Gall.
\newblock {\em Spatial branching processes, random snakes and partial
  differential equations}.
\newblock Birkhauser Basel, 1999.

\bibitem{le2005random}
Jean-Fran{\c{c}}ois Le~Gall.
\newblock Random trees and applications.
\newblock {\em Probab. Surv.}, 2:245--311, 2005.

\bibitem{legall07topological}
Jean-Fran{\c{c}}ois Le~Gall.
\newblock The topological structure of scaling limits of large planar maps.
\newblock {\em Invent. Math.}, 169(3):621--670, 2007.

\bibitem{LeGallGeo}
Jean-Fran{\c{c}}ois Le~Gall.
\newblock Geodesics in large planar maps and in the {B}rownian map.
\newblock {\em Acta Math.}, 205(2):287--360, 2010.

\bibitem{jf}
Jean-Fran{\c{c}}ois Le~Gall et~al.
\newblock Uniqueness and universality of the {B}rownian map.
\newblock {\em The Annals of Probability}, 41(4):2880--2960, 2013.

\bibitem{legall2008scaling}
Jean-Fran{\c{c}}ois Le~Gall and Fr{{\'e}}d{{\'e}}ric Paulin.
\newblock Scaling limits of bipartite planar maps are homeomorphic to the
  2-sphere.
\newblock {\em Geom. Funct. Anal.}, 18(3):893--918, 2008.

\bibitem{MarckertLineage}
Jean-Fran{\c{c}}ois Marckert.
\newblock The lineage process in {G}alton--{W}atson trees and globally centered
  discrete snakes.
\newblock {\em Ann. Appl. Probab.}, 18(1):209--244, 2008.

\bibitem{MarckertMiermont07}
Jean-Fran{\c{c}}ois Marckert and Gr{\'e}gory Miermont.
\newblock Invariance principles for random bipartite planar maps.
\newblock {\em Ann. Probab.}, 35(5):1642--1705, 2007.

\bibitem{MR2294979}
Jean-Fran{\c{c}}ois Marckert and Abdelkader Mokkadem.
\newblock Limit of normalized quadrangulations: the {B}rownian map.
\newblock {\em Ann. Probab.}, 34(6):2144--2202, 2006.

\bibitem{Miermont08}
Gr\'egory Miermont.
\newblock On the sphericity of scaling limits of random planar
  quadrangulations.
\newblock {\em Electron. Commun. Probab}, 13:248--257, 2008.

\bibitem{MiermontTessellations}
Gr{\'e}gory Miermont.
\newblock Tessellations of random maps of arbitrary genus (mosa{\i}ques sur des
  cartes al{\'e}atoires en genre arbitraire).
\newblock {\em Ann. Sci. {\'E}c. Norm. Sup{\'e}r.}, 42(fasc.\ 5):725--781,
  2009.

\bibitem{miermont13brownian}
Gr{\'e}gory Miermont.
\newblock The {B}rownian map is the scaling limit of uniform random plane
  quadrangulations.
\newblock {\em Acta mathematica}, 210(2):319--401, 2013.

\bibitem{OssonaThesis}
Patrice Ossona~de Mendez.
\newblock {\em Orientations bipolaires}.
\newblock PhD thesis, {\'E}cole des Hautes Etudes en Sciences Sociales, Paris,
  1994.

\bibitem{PouSch06}
Dominique Poulalhon and Gilles Schaeffer.
\newblock Optimal coding and sampling of triangulations.
\newblock {\em Algorithmica}, 46(3):505--527, 2006.

\bibitem{Schnyder}
Walter Schnyder.
\newblock Planar graphs and poset dimension.
\newblock {\em Order}, 5:323--343, 1989.

\bibitem{stephenson05circle}
Kenneth Stephenson.
\newblock {\em Introduction to circle packing}.
\newblock Cambridge Univ. Press, Cambridge, 2005.







\end{thebibliography}



\normalsize
\end{document}